\documentclass[12pt]{amsart}

\usepackage{amsfonts, amstext, amsmath, amsthm, amscd, amssymb, enumitem, url}
\usepackage{tikz-cd}
\usepackage{xcolor}
\usepackage{svg}
\usepackage{pdfpages}
\usepackage{pdflscape}

\usepackage{microtype}
\usepackage[parfill]{parskip}

\usepackage[colorlinks,citecolor=cyan,linkcolor=magenta]{hyperref}
\usepackage[nameinlink]{cleveref}
\usepackage{subfig}

\newtheorem{thm}{Theorem}[section]

\newtheorem{lemma}[thm]{Lemma}
\newtheorem{prop}[thm]{Proposition}
\newtheorem{claim}[thm]{Claim}
\newtheorem{constr}[thm]{Construction}

\newtheorem*{prop:vbs}{\Cref{prop:vbs}}
\newtheorem*{thm:noperfectfits}{\Cref{thm:noperfectfits}}

\theoremstyle{definition}
\newtheorem{defn}[thm]{Definition}
\newtheorem{rmk}[thm]{Remark}

\numberwithin{equation}{section}

\renewcommand{\epsilon}{\varepsilon}

\usepackage{stmaryrd}
\newcommand{\cut}{\!\bbslash\!}

\newcommand{\ind}{\mathrm{ind}}
\renewcommand{\top}{\mathrm{top}}
\newcommand{\red}{\mathrm{red}}
\newcommand{\Homeo}{\mathrm{Homeo}}
\newcommand{\Isom}{\mathrm{Isom}}
\newcommand{\brloc}{\mathrm{brloc}}

\begin{document}

\title{Veering branched surfaces, surgeries, and geodesic flows}

\author{Chi Cheuk Tsang}
\address{University of California, Berkeley \\
    970 Evans Hall \#3840 \\
    Berkeley, CA 94720-3840}
\email{chicheuk@math.berkeley.edu}
\thanks{Chi Cheuk Tsang was partially supported by a grant from the Simons Foundation \#376200.}

\maketitle

\begin{abstract}
We introduce veering branched surfaces as a dual way of studying veering triangulations. We then discuss some surgical operations on veering branched surfaces. Using these, we provide explicit constructions of some veering branched surfaces whose dual veering triangulations correspond to geodesic flows of negatively curved surfaces. We construct these veering branched surfaces on (i) complements of Montesinos links whose double branched covers are unit tangent bundles of negatively curved orbifolds, and (ii) complements of full lifts of filling geodesics in unit tangent bundles of negatively curved surfaces, when the geodesics have no triple intersections and have ($n \geq 4$)-gons as complementary regions. As an application, this provides explicit Markov partitions of geodesic flows on negatively curved surfaces. 
In an appendix, we classify the drilled unit tangent bundles which admit a veering triangulation corresponding to a geodesic flow, by characterizing when there are no perfect fits. 
\end{abstract}

\section{Introduction} \label{sec:intro}

Veering triangulations were introduced by Ian Agol in \cite{Ago11} as a tool for studying mapping tori of pseudo-Anosov homeomorphisms. The veering triangulations combinatorially encapsulate the pseudo-Anosov monodromies by encoding a periodic folding sequence of train tracks. Recently, due to work of Agol, Gu\'eritaud, Schleimer, Segerman, Landry, Minsky, Taylor, and the author (\cite{Gue16}, \cite{SS19}, \cite{SS20}, \cite{LMT21}, \cite{SS21}, \cite{AT22}, \cite{SS23}, \cite{SS22c}), it has been realized that this is merely a special case of a bigger picture: veering triangulations can be used to combinatorially encode pseudo-Anosov flows without perfect fits in general.

We briefly explain what we mean by this, for the definitions, precise statements and references see \Cref{sec:background}. Given a pseudo-Anosov flow $\phi$ on an orientable closed 3-manifold $N$ and a collection of closed orbits $\mathcal{C}$, one can construct a veering triangulation on the cusped 3-manifold $N \backslash \mathcal{C}$, provided that $\phi$ and $\mathcal{C}$ satisfy a technical condition called \textit{no perfect fits}. Conversely, given a veering triangulation on a cusped 3-manifold $M$, one can construct a pseudo-Anosov flow on the Dehn filled closed 3-manifold $M(s)$ which satisfies the no perfect fits condition, provided that the filling coefficient $s$ satisfies a natural intersection number condition. In fact, work of Schleimer and Segerman, to appear, shows that the two directions of construction are inverses of each other. Even without knowledge of this, however, the two directions of construction already allow us to study pseudo-Anosov flows in new ways using veering triangulations. 
One of the advantages of doing so is that veering triangulations are discrete objects, and so the associated constructions can often be described explicitly using some finite collection of data. One of the goals of this paper is to demonstrate this feature.

As we pointed out, the application of veering triangulations to study suspension flows of pseudo-Anosov mapping tori was the original motivation behind their conception. Another class of pseudo-Anosov flows that are frequently studied are \textit{geodesic flows} of negatively curved surfaces. These are flows on the unit tangent bundle of surfaces (or orbifolds in general) with a negatively curved Riemannian metric, which sends a vector $v$ to the vector $\gamma'(t)$ at time $t$, if $\gamma$ is the geodesic with initial velocity $v$. One reason why these are popular to study is that together with suspension flows of Anosov maps, they constitute all of the algebraic Anosov flows in dimension 3, meaning they are the only Anosov flows that are induced by a left invariant vector field on the right quotient of a 3-dimensional Lie group, see \cite{Tom75}.

Another reason is that throughout the development of low-dimensional topology, it has always been an important topic to study closed curves on a surface, in particular their topological type and their growth rates, see for example \cite{FLP79}, \cite{MM99}, \cite{Mir16}. Now, closed orbits of the geodesic flow of a negatively curved surface correspond precisely to isotopy classes of closed curves, hence one can hope to answer some of the surface-theoretic questions by studying these geodesic flows on 3-manifolds. 

Yet another reason is that geodesic flows form the prototype for contact Anosov flows. These are Anosov flows which are also Reeb flows to some contact form. Foulon and Hasselblat showed in \cite{FH13} that one can construct many examples of contact Anosov flows starting with geodesic flows, and Barbot showed in \cite{Bar01} that contact Anosov flows must be skew $\mathbb{R}$-covered in general, hence induce representations into $\Homeo(\mathbb{R})$.

In view of the importance of geodesic flows, it is natural to ask: What do the veering triangulations that correspond to them look like? In this paper, we provide some answers to this question by introducing the tool of \textit{veering branched surfaces}, and explicitly constructing examples of these which are dual to veering triangulations corresponding to geodesic flows.

We first explain the motivation behind veering branched surfaces. Given a veering triangulation of a 3-manifold $M$, one can combinatorially construct its \textit{unstable branched surface} $B$. In terms of the correspondence between veering triangulations and pseudo-Anosov flows, the unstable branched surface carries the unstable lamination of the flow in the filled 3-manifold $M(s)$. We observe that using some basic properties, we can completely characterize, among all branched surfaces, the ones that arise as the unstable branched surface of some veering triangulation. Moreover, from the unstable branched surface one can recover the veering triangulation by taking the dual ideal triangulation. Hence by defining a veering branched surface to be a branched surface that satisfies the characterizing properties of an unstable branched surface, we see that studying veering branched surfaces is an equivalent, and in a sense, dual, way of studying veering triangulations. We state this concretely as the following proposition. For definitions see \Cref{sec:vbs}.

\begin{prop:vbs}
Let $B$ be a veering branched surface in an oriented 3-manifold $M$ whose complementary regions are all cusped torus shells. Then the dual ideal triangulation of $B$ is a veering triangulation $\Delta$ of $M$, and $B$ can be identified with the unstable branched surface for $\Delta$.
\end{prop:vbs}

The advantage of working with veering branched surfaces however, is that one can construct and manipulate branched surfaces inside desirable ambient 3-manifolds, instead of working with triangulations that comprise the 3-manifolds themselves. Indeed, in this paper we introduce the generalized notion of an \textit{almost veering branched surface} and two basic surgical operations: \textit{horizontal} and \textit{vertical surgery}, which one can perform on almost veering branched surfaces to construct veering branched surfaces thus their dual veering triangulations, but might otherwise be difficult to visualize from the perspective of the triangulation.

One should compare our horizontal and vertical surgery with Schleimer and Segerman's \textit{veering Dehn surgery}. In all three cases, an annulus or a M\"obius band carried by the 2-skeleton of the veering triangulation is slit open and tetrahedra are inserted within. In Schleimer and Segerman's terminology, an annulus with all \textit{flat} edges is slit open for horizontal surgery, while an annulus with some \textit{sharp} edges is slit open for vertical surgery. More details of veering Dehn surgery will appear in \cite{SS22b}, and comparing these surgery operations will be a topic of future work.

The horizontal surgery operation, or rather variants of it, is the main tool we use for constructing veering branched surfaces in this paper.
Even though our techniques work more generally, for this paper we will focus on the following two particular settings, and only indicate how to generalize in the remarks.

\begin{constr} \label{constr:genus0}
Consider a closed orientable genus zero 2-dimensional orbifold $S$ with negative Euler characteristic. Let $c$ be a simple closed curve on $S$ that passes through all the cone points. Let $\overset{\leftrightarrow}{c}$ be the \textit{full lift} of $c$ in the unit tangent bundle of $S$, defined to be $\{\pm c'(t)\} \subset T^1 S$. There is an involution on $T^1 S \backslash \overset{\leftrightarrow}{c}$ induced by reflecting $S$ across $c$, and the quotient of this involution is a \textit{Montesinos link} complement.

In this setting, we can construct explicit veering branched surfaces on $T^1 S \backslash \overset{\leftrightarrow}{c}$. These veering branched surfaces are dual to veering triangulations which correspond to the geodesic flow on $T^1 S$. Moreover, they can be quotiented down to veering branched surfaces on the corresponding Montesinos link complements. 
\end{constr}

\begin{constr} \label{constr:highergenus}
Consider a closed orientable surface $S$ with negative Euler characteristic. Let $c$ be a filling collection of mutually nonparallel curves on $S$, which has no triple intersections and whose complementary regions in $S$ are $(n \geq 4)$-gons. 
Define the full lift of $c$, $\overset{\leftrightarrow}{c}$, as above.

In this setting, we can construct explicit veering branched surfaces on $T^1 S \backslash \overset{\leftrightarrow}{c}$. These veering branched surfaces are dual to veering triangulations which correspond to the geodesic flow on $T^1 S$.
\end{constr}

Even though our constructions of the veering branched surfaces are explicit, we do not claim to have explicit descriptions of the veering triangulations themselves, due to the fact that computing dual triangulations is a difficult task by hand. Without knowing the triangulations however, there are still useful invariants which can be computed directly from their dual veering branched surfaces.

We point out one of these in particular. Given a veering triangulation of a 3-manifold $M$, one can combinatorially define its \textit{reduced flow graph} $\Phi_{\red}$, which is a directed graph naturally embedded in $M$. In terms of the correspondence between veering triangulations and pseudo-Anosov flows, the reduced flow graph encodes a Markov partition for the pseudo-Anosov flow on $M(s)$. In particular one can study the periodic orbits of the pseudo-Anosov flow by studying cycles carried by $\Phi_{\red}$. Indeed, a quantitative approach of this has been carried out in \cite{LMT20} and \cite{LMT21}. 

Now by dualizing the definition of the (reduced) flow graph, it is not difficult to read it off from the veering branched surface dual to the given veering triangulation. Hence we can in particular determine the reduced flow graphs of the veering triangulations dual to the veering branched surfaces in \Cref{constr:highergenus}, see for example \Cref{fig:hexnonamarkov}. These will then encode explicit Markov partitions for geodesic flows.

\begin{constr} \label{constr:markovpart}
Let $S$ be a closed orientable 2-dimensional orbifold with negative Euler characteristic. We can construct explicit Markov partitions for the geodesic flow on $T^1 S$. If $S$ is a surface, we can arrange for the Markov partition to have $-108 \chi(S)$ flow boxes.
\end{constr}

The problem of representing geodesic flows in terms of explicit systems of symbolic dynamics has a long history, see for example \cite{Ser81}, \cite{Ser86}, \cite{AF91}, \cite{KU07}. However, it is not always made clear how the corresponding Markov partitions look like. The point is that the graph encoding a Markov partition on its own does not contain instructions of how to join up the top and bottom faces of the flow boxes at each vertex, and this is crucial information if one wants to study problems about knottedness and linkedness of orbits. Some of the more recent work, for example \cite{Ghy07}, \cite{Pin14}, \cite{Deh15}, and \cite{DP18}, do contain this additional information, and in this paper this information is included naturally as part of our approach. 

Another remark is that, to the author's knowledge, the approach in most of the previous work is geometric, making use of some auxiliary hyperbolic metric (with \cite{Ghy07} and \cite{Pin14} being notable exceptions), whereas our approach is entirely topological. Finally, we remark that our methods also give explicit Markov partitions for (the nonwandering set of) geodesic flows on cusped hyperbolic surfaces, via a doubling trick.

Coming out of this paper, an obvious direction for future research is to construct veering triangulations or branched surfaces for geodesic flows with other types of orbits drilled out. The careful reader will notice, however, that we mentioned the technical condition of no perfect fits must be satisfied for this task to be possible. One should therefore ascertain when this no perfect fits condition holds, before trying to do the construction in general. We make some progress towards this by characterizing exactly when there are no perfect fits. The exact result is the following, see \Cref{sec:directproof} for definitions.

\begin{thm:noperfectfits}
Let $\Sigma$ be a closed oriented hyperbolic surface and $c$ be a collection of oriented closed geodesics. Then the geodesic flow on $T^1 \Sigma$ has no perfect fits relative to the lift $\overset{\rightarrow}{c}$ if and only if every oriented closed geodesic $d$ on $\Sigma$ has a positive intersection point with some element of $c$. 

In particular, if $c$ is a collection of closed geodesics, then the geodesic flow on $T^1 \Sigma$ has no perfect fits relative to the full lift $\overset{\leftrightarrow}{c}$ if and only if $c$ is filling.
\end{thm:noperfectfits}

Another direction for future work is to understand better the veering triangulations dual to the veering branched surfaces constructed in this paper. For example, we suspect that their canonical shearing decompositions (introduced in \cite{SS23}) admit a neat description. It is also worth investigating whether these triangulations are geometric. 

To that end, we compiled, in an indirect way, tables of all the veering triangulations dual to the veering branched surfaces in \Cref{constr:genus0} that appear in the veering triangulation census \cite{VeeringCensus}, and have included these tables in \Cref{sec:table}. By studying the triangulations listed in the tables carefully, one might be able to work out some patterns and infer the form of the veering triangulations, or at least some of their invariants, in general.

Here is an outline of this paper. In \Cref{sec:background}, we recall some background knowledge about veering triangulations, pseudo-Anosov flows, geodesic flows and Montesinos links. In \Cref{sec:vbs}, we define veering branched surfaces and almost veering branched surfaces, and show how veering branched surfaces are dual to veering triangulations. In \Cref{sec:surgery}, we introduce some surgical operations on almost veering branched surfaces. The two basic ones are horizontal and vertical surgery. We also explain a variant of horizontal surgery, which we call halved concurrent horizontal surgery, which will be used extensively in \Cref{sec:genus0}. 

In \Cref{sec:genus0}, we explain \Cref{constr:genus0} by separating into a few different cases depending on the number and order of the cone points. In \Cref{sec:highergenus}, we explain \Cref{constr:highergenus} using some of the knowledge from \Cref{sec:genus0}. From this we will derive \Cref{constr:markovpart}. In \Cref{sec:questions}, we discuss some questions and future directions coming out of this paper. 

There are two appendices. In \Cref{sec:directproof}, we prove \Cref{thm:noperfectfits}. In \Cref{sec:table}, we identify the veering triangulations we constructed on Montesinos link complements in \Cref{constr:genus0} which appear in the veering triangulation census, and compile their IsoSig codes in some tables.

{\bf Acknowledgements.} I would like to thank Ian Agol and Michael Landry for their support and encouragement throughout this project. I would like to thank Saul Schleimer, Henry Segerman, Mario Shannon, and Jonathan Zung for helpful conversations. I would like to thank Pierre Dehornoy and Caroline Series for comments on an earlier version. I would like to thank the anonymous referees' comments for improving the presentation of the paper, in particular for a helpful suggestion regarding \Cref{fig:hsurinv}.

{\bf Notational conventions.} Throughout this paper, 
\begin{itemize}
    \item $X \cut Y$ will denote the metric completion of $X \backslash Y$ with respect to the induced path metric from $X$. In addition, we will call the components of $X \cut Y$ the complementary regions of $Y$ in $X$.
    \item $\widetilde{X}$ will denote the universal cover of $X$, unless otherwise stated.
    \item Notation such as $\mathbb{R}_t$, $[0,1]_t$, etc. will mean that we use the variable $t$ as the coordinate on $\mathbb{R}$, $[0,1]$, etc.
\end{itemize}

\section{Background} \label{sec:background}

\subsection{Veering triangulations} \label{subsec:vt}

We recall the definition of a veering triangulation.

An \textit{ideal tetrahedron} is a tetrahedon with its 4 vertices removed. The removed vertices are called the \textit{ideal vertices}. 

Let $M$ be the interior of a compact oriented 3-manifold with torus boundary components. An \textit{ideal triangulation} of $M$ is a decomposition of $M$ into ideal tetrahedra glued along pairs of faces.

A \textit{taut structure} on an ideal triangulation is a labelling of the dihedral angles by $0$ or $\pi$, such that 
\begin{itemize}
    \item Each tetrahedron has exactly two dihedral angles labelled $\pi$, and they are opposite to each other.
    \item The angle sum around each edge in the triangulation is $2\pi$.
\end{itemize}

A \textit{transverse taut structure} is a taut structure along with a coorientation on each face, such that for any edge labelled $0$ in a tetrahedron, exactly one of the faces adjacent to it is cooriented inwards.

A \textit{transverse taut ideal triangulation} is an ideal triangulation with a transverse taut structure.

\begin{defn} \label{defn:vt}
A \textit{veering structure} on a transverse taut ideal triangulation of $M$ is a coloring of the edges by red or blue, so that if we look at each tetrahedron with a $\pi$-labelled edge in front, the four outer $0$-labelled edges, starting from an end of the front edge and going counter-clockwise, are colored red, blue, red, blue, respectively. We call such a tetrahedron a \textit{veering tetrahedron}.

A \textit{veering triangulation} is a transverse taut ideal triangulation with a veering structure.
\end{defn}

\Cref{fig:veertet} shows a veering tetrahedron in a veering triangulation.

\begin{figure} 
    \centering
    \fontsize{14pt}{14pt}\selectfont
    \resizebox{!}{4cm}{
\begingroup%
  \makeatletter%
  \providecommand\color[2][]{%
    \errmessage{(Inkscape) Color is used for the text in Inkscape, but the package 'color.sty' is not loaded}%
    \renewcommand\color[2][]{}%
  }%
  \providecommand\transparent[1]{%
    \errmessage{(Inkscape) Transparency is used (non-zero) for the text in Inkscape, but the package 'transparent.sty' is not loaded}%
    \renewcommand\transparent[1]{}%
  }%
  \providecommand\rotatebox[2]{#2}%
  \newcommand*\fsize{\dimexpr\f@size pt\relax}%
  \newcommand*\lineheight[1]{\fontsize{\fsize}{#1\fsize}\selectfont}%
  \ifx\svgwidth\undefined%
    \setlength{\unitlength}{325.82965819bp}%
    \ifx\svgscale\undefined%
      \relax%
    \else%
      \setlength{\unitlength}{\unitlength * \real{\svgscale}}%
    \fi%
  \else%
    \setlength{\unitlength}{\svgwidth}%
  \fi%
  \global\let\svgwidth\undefined%
  \global\let\svgscale\undefined%
  \makeatother%
  \begin{picture}(1,0.40856788)%
    \lineheight{1}%
    \setlength\tabcolsep{0pt}%
    \put(0,0){\includegraphics[width=\unitlength,page=1]{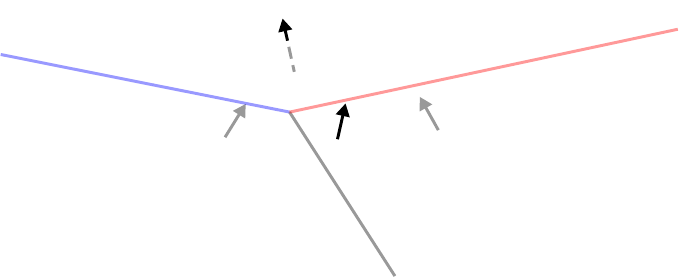}}%
    \put(0.45702311,0.36605622){\color[rgb]{0,0,0}\makebox(0,0)[lt]{\lineheight{1.25}\smash{\begin{tabular}[t]{l}$\pi$\end{tabular}}}}%
    \put(0.49904816,0.14324942){\color[rgb]{0,0,0}\transparent{0.40000001}\makebox(0,0)[lt]{\lineheight{1.25}\smash{\begin{tabular}[t]{l}$\pi$\end{tabular}}}}%
    \put(0.76168485,0.11769163){\color[rgb]{0,0,1}\makebox(0,0)[lt]{\lineheight{1.25}\smash{\begin{tabular}[t]{l}$0$\end{tabular}}}}%
    \put(0.31075696,0.09433859){\color[rgb]{1,0,0}\makebox(0,0)[lt]{\lineheight{1.25}\smash{\begin{tabular}[t]{l}$0$\end{tabular}}}}%
    \put(0.25420744,0.29133508){\color[rgb]{0,0,1}\transparent{0.40000001}\makebox(0,0)[lt]{\lineheight{1.25}\smash{\begin{tabular}[t]{l}$0$\end{tabular}}}}%
    \put(0.62556736,0.30455855){\color[rgb]{1,0,0}\transparent{0.40000001}\makebox(0,0)[lt]{\lineheight{1.25}\smash{\begin{tabular}[t]{l}$0$\end{tabular}}}}%
    \put(0,0){\includegraphics[width=\unitlength,page=2]{veertet.pdf}}%
  \end{picture}%
\endgroup%
}
    \caption{A tetrahedron in a transverse veering triangulation. There are no restrictions on the colors of the top and bottom edges.} 
    \label{fig:veertet}
\end{figure}

We next recall the definitions of the unstable branched surface and the (reduced) flow graph associated to a veering triangulation.

\begin{defn} \label{defn:branchsurf}
Let $M$ be a 3-manifold. A \textit{branched surface} $B$ is a compact subset of $M$ locally of the form of one of the pictures in \Cref{fig:branchsurflocal}.

\begin{figure}
    \centering
    \resizebox{!}{2cm}{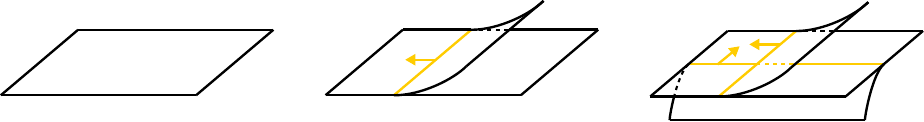} 
    \caption{The local models for branched surfaces. The arrows indicate the maw coorientation of the branch locus.}
    \label{fig:branchsurflocal}
\end{figure}

The set of points where $B$ is locally of the form of \Cref{fig:branchsurflocal} middle or right is called the \textit{branch locus} of $B$ and is denoted by $\brloc(B)$. The points where $B$ is locally of the form of \Cref{fig:branchsurflocal} right are called the \textit{triple points} of $B$. The complementary regions of $\brloc(B)$ in $B$ are called the \textit{sectors} of $B$. 

The branch locus $\brloc(B)$ is a union of smoothly embedded circles. We call each such circle a \textit{component} of $\brloc(B)$. Equivalently, one can consider the complementary regions of $B$ in $M$. The boundary of these regions consist of smoothly embedded faces meeting along \textit{cusp circles}. Each component of $\brloc(B)$ is the image of such a cusp circle. Each component of $\brloc(B)$ has a canonical coorientation on $B$, which we call the \textit{maw coorientation}, given locally by the direction from the side with more sectors to the side with less sectors. See the arrows in \Cref{fig:branchsurflocal}.

The sectors of $B$ are surfaces with boundary, with \textit{corners} at where the boundary locally switches from lying along one component of the branch locus to another. We define the \textit{index} of a surface with corners $S$ as $\ind(S):=\chi_{\top}(S)-\frac{1}{4}\text{\#corners}$, where $\chi_{\top}(S)$ is the Euler characteristic of the underlying topological surface. This definition of index is additive: if a surface with corners $S$ is divided by a collection of curves and arcs into surfaces with corners $S_1,...,S_k$, then $\ind(S)=\sum \ind(S_i)$. We will also call the complementary regions of the corners in $\partial S$ the \textit{sides} of $S$.

If every component of $\brloc(B)$ contains at least one triple point, then one can define a cellular structure on $\brloc(B)$ by declaring the 0-cells to be the triple points, and the 1-cells to be the complementary regions of the triple points in $\brloc(B)$. Furthermore, if each sector is topologically a disc, then we can define a cellular structure on $B$ by declaring the 2-cells to be the sectors. In this scenario, we can define the \textit{dual ideal triangulation} to $B$: Let $N$ be the space obtained by attaching cones over each component of $\partial (M \cut B)$ onto $B$. Construct a triangulation $\Delta'$ of $N$ by first placing a vertex at each cone point. Then for each sector of $B$, pick a point in the interior of the sector and cone it off, i.e. join it to the cone points of the two cones on either side of the sector along straight paths, to form edges of $\Delta'$. Then for each 1-cell in $\brloc(B)$, pick a point in its interior and join it to the points we chose in the interior of the 3 sectors the edge is adjacent to, along disjoint paths in those sectors, then cone off these paths to form faces of $\Delta'$. Finally, define the complementary regions of the 2-complex we constructed so far in $N$ to be the tetrahedra. This gives us a triangulation $\Delta'$ of $N$. Now delete all the cone points to get an ideal triangulation $\Delta$ of a regular neighborhood of $B$ in $M$. 
\end{defn}

\begin{defn} \label{defn:cusptorus} (\cite{Mos96})
Consider a solid torus $D^2 \times S^1$. Let $l$ be a nonempty collection of parallel simple closed curves on its boundary which are not parallel to the meridian. Let $p>0$ be the geometric intersection number between the meridian and $l$. Then the 3-manifold obtained by placing cusp circles along $l$ is called a \textit{$p$-cusped solid torus}, or \textit{cusped solid torus} for short.

Similarly, consider a solid torus with its core drilled out, $S^1 \times [0,\infty) \times S^1$. Let $l$ be a nonempty collection of parallel simple closed curves on its boundary. Then the 3-manifold obtained by placing cusp circles along $l$ is called a \textit{cusped torus shell}.
\end{defn}

\begin{defn} \label{defn:unstablebranchsurf}
Let $\Delta$ be a veering triangulation of a 3-manifold $M$. For each tetrahedron of $\Delta$, define a branched surface inside by placing a quadrilateral with vertices on the top and bottom edges and the two side edges of the same color as the top edge, then adding a triangular sector for each side edge of the opposite color to the top edge, as in \Cref{fig:branchsurf} left. These branched surfaces in each tetrahedron can be arranged to match up across faces, thus glue up to a branched surface in $M$, which we call the \textit{unstable branched surface} $B$.
\end{defn}

We record some simple yet important properties of the unstable branched surface.

\begin{prop} \label{prop:branchsurf}
Let $\Delta$ be a veering triangulation of a 3-manifold $M$ and let $B$ be its unstable branched surface.
\begin{enumerate}[label=(\roman*)]
    \item Each sector of $B$ is a disc with $4$ corners.
    \item Each component of $M \cut B$ is a cusped torus shell.
    \item The components of $\brloc(B)$ can be oriented in a way such that at each triple point, the orientation of each component induces the maw coorientation on the other component.
    \item Consider $B$ as a cell complex as in \Cref{defn:branchsurf}. Then $\Delta$ is the dual ideal triangulation to $B$.
\end{enumerate}
\end{prop}

\begin{proof}
All of these are straightforward, but we will write down some references for the interested reader. For (i), see \cite[Section 6]{SS19}. For (ii), see \cite[Proposition 2.9]{AT22}. For (iii), orient the components within each tetrahedron to go from the top faces to the bottom faces. For (iv), see \cite[Section 6]{SS19} again.
\end{proof}

From now on, we will implicitly orient the components of $\brloc(B)$ as in (iii) above. 

\begin{defn}
Each component of $M \cut B$ contains an end of $M$, by \Cref{prop:branchsurf} (ii). We call the cusp circles on a component of $M \cut B$ the ladderpole curves on the corresponding end of $M$. We call the collection of all ladderpole curves on an end of $M$ the \textit{ladderpole class} on that end.
\end{defn}

\begin{figure} 
    \centering
    \resizebox{!}{3.6cm}{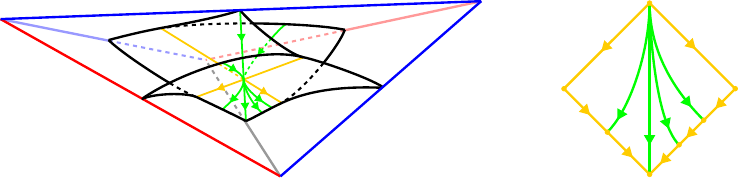}
    \caption{Left: The portion of the unstable branched surface and the flow graph within each veering tetrahedron. Right: The portion of the flow graph on each sector of the unstable branched surface.}
    \label{fig:branchsurf}
\end{figure}

\begin{defn}[Landry-Minsky-Taylor {\cite{LMT20}}] \label{defn:flowgraph}
Let $\Delta$ be a veering triangulation of a 3-manifold $M$. Define the \textit{flow graph} $\Phi$ to be a directed graph with the set of vertices equal to the set of edges of $\Delta$, and adding 3 edges for each tetrahedron, going from the top edge and the two side edges of opposite color to the top edge into the bottom edge.

$\Phi$ can be naturally embedded in the unstable branched surface $B$, hence in $M$, by placing each vertex at the top corner of the sector of $B$ its corresponding edge of $\Delta$ meets, and placing the edges that exit that vertex within that sector of $B$. See \Cref{fig:branchsurf} right. Note that the tangent planes to $B$ determine a framing of the edges of $\Phi$ in $M$.

Define a \textit{planar ordering} of a set to be an equivalence class of linear orderings up to complete reversal. (The motivation of this terminology comes from the fact that one can rotate a line by $\pi$ in a plane, reversing the linear ordering of a set of elements on it.) The embedding of $\Phi$ in $B$ also determines planar orderings on the sets of incoming and outgoing edges at each vertex of $\Phi$.

We will sometimes abuse notation and include the embedding of $\Phi$ in $M$, the framing of its edges in $M$, and the planar orderings of the incoming and outgoing edges at each vertex as part of the data of $\Phi$.
\end{defn}

\begin{defn} \label{defn:redflowgraph}
A proper subgraph $G'$ of a directed graph $G$ is an \textit{infinitesimal component} if there are no edges from vertices in $G'$ to vertices outside of $G'$.

The \textit{reduced flow graph} $\Phi_{\red}$ is the maximal subgraph of the flow graph $\Phi$ that has no infinitesimal components.
\end{defn}

\begin{prop} \label{prop:redflowgraph}
The infinitesimal components of $\Phi$ consist of disjoint cycles, and $\Phi_{\red}$ can be obtained by deleting these cycles along with the edges that enter them.
\end{prop}

\begin{proof}
This is shown in \cite[Section 3]{AT22}. In fact, in that paper we show something stronger: the disjoint cycles must lie in special subsets of the veering triangulation called walls. However, we will not need this additional fact in this paper.
\end{proof}

$\Phi_{\red}$ inherits from $\Phi$ an embedding in $M$, a framing of its edges in $M$, and planar orderings of the incoming and outgoing edges at each vertex. Again, we will sometimes abuse notation and include these as part of the data of $\Phi_{\red}$.

\begin{rmk} \label{rmk:conventions}
Our conventions in defining the unstable branched surface and flow graph is consistent with that in \cite{AT22}, but might be different from that of other authors. We provide here a dictionary between our convention and two other sets of conventions that we know of:

In work of Schleimer and Segerman, what we call the unstable branched surface is called the upper branched surface (in dual position).

In work of Landry, Minsky, and Taylor, what we call the unstable branched surface is called the stable branched surface, and the edges of the flow graph are oriented in the opposite direction.

\end{rmk}

\subsection{Pseudo-Anosov flows} \label{subsec:pAflow}

We recall the definition of a pseudo-Anosov flow.

\begin{defn} \label{defn:phorbit}
Consider the map $\begin{bmatrix} \lambda^{-1} & 0\\ 0 & \lambda \end{bmatrix}: \mathbb{R}^2 \to \mathbb{R}^2$, for $\lambda>1$. This preserves the foliations of $\mathbb{R}^2$ by horizontal and vertical lines respectively. Let $\phi_{n,0,\lambda}:\mathbb{R}^2 \to \mathbb{R}^2$ be the lift of this map over $z \mapsto z^{\frac{n}{2}}$ that preserves the lift of the quadrants. (When $n$ is odd, one has to choose a branch of $z \mapsto z^{\frac{n}{2}}$ but it is easy to see that the result is independent of the choice.) Let $\phi_{n,k, \lambda}: \mathbb{R}^2 \to \mathbb{R}^2$ be the composition of $\phi_{n,0,\lambda}$ and rotation by $\frac{2\pi k}{n}$ anticlockwise. Meanwhile, let $l^s, l^u$ be the singular foliations of $\mathbb{R}^2$ obtained by pulling back the foliations by horizontal and vertical lines under $z \mapsto z^{\frac{n}{2}}$, respectively. These are preserved by $\phi_{n,k,\lambda}$. Let $\Phi_{n,k,\lambda}$ be the mapping torus of $\phi_{n,k, \lambda}$, let $\Lambda^s, \Lambda^u$ be the suspensions of $l^s, l^u$ respectively, and consider the suspension flow on $\Phi_{n,k, \lambda}$. Call the suspension of the origin the \textit{pseudo-hyperbolic orbit} of $\Phi_{n,k,\lambda}$.
\end{defn}

\begin{defn} \label{defn:pAflow}
A \textit{pseudo-Anosov flow} on a closed 3-manifold $N$ is a $C^1$-flow $\phi_t$ satisfying:
\begin{itemize}
    \item There is a finite collection of closed orbits $\{\gamma_1, ..., \gamma_s \}$, called the \textit{singular orbits}, such that $\phi_t$ is smooth away from the singular orbits.
    \item There is a path metric $d$ on $N$, which is induced by a Riemannian metric $g$ away from the singular orbits.
    \item Away from the singular orbits, there is a splitting of the tangent bundle into three $\phi_t$-invariant line bundles $TM=E^s \oplus E^u \oplus T\phi_t$, such that $$|d\phi_t(v)| < C \lambda^{-t} |v|$$ for every $v \in E^s, t>0$, and $$|d\phi_t(v)| < C \lambda^t |v|$$ for every $v \in E^u, t<0$, for some $C, \lambda>1$.
    \item Each singular orbit $\gamma_i$ has a neighborhood $N_i$ and a map $f_i$ sending $N_i$ to a neighborhood of the pseudo-hyperbolic orbit in $\Phi_{n_i, k_i, \lambda}$, for some $n_i \geq 3$, such that $f_i$ is bi-Lipschitz on $N_i$ and smooth away from $\gamma_i$, preserves the orbits, and sends $E^s, E^u$ to line bundles tangent to $\Lambda^s, \Lambda^u$ respectively. In this case, we say that $\gamma_i$ is \textit{$n_i$-pronged}. By extension, we also say that a non-singular orbit is \textit{$2$-pronged}.
\end{itemize}

We call the (possibly singular) foliation which is tangent to $E^s \oplus T\phi_t$ away from the singular orbits and given by the image of $\Lambda^s \subset \Phi_{n_i,k_i,\lambda}$ under $f_i$ near the singular orbits the \textit{stable foliation} $\Lambda^s$. We define the \textit{unstable foliation} $\Lambda^u$ similarly.

A pseudo-Anosov flow without singular orbits is called an \textit{Anosov flow}.
\end{defn}

\begin{defn} \label{defn:perfectfit}
Let $\phi$ be a pseudo-Anosov flow on a closed 3-manifold $N$, and let $\mathcal{C}$ be a collection of closed orbits of $\phi$. Lift these up to a flow $\widetilde{\phi}$ on the universal cover $\widetilde{N}$ with a collection of orbits $\widetilde{\mathcal{C}}$ which is the preimage of $\mathcal{C}$.

It is shown in \cite[Proposition 4.2]{FM01} that the orbit space $\mathcal{O}$ of $\widetilde{\phi}$ is homeomorphic to $\mathbb{R}^2$, and the images of $\Lambda^s, \Lambda^u$ are two (possibly singular) 1-dimensional foliations $\mathcal{O}^s, \mathcal{O}^u$, respectively.

A \textit{perfect fit rectangle} is a rectangle-with-one-ideal-vertex properly embedded in $\mathcal{O}$ such that 2 opposite sides of the rectangle lie along leaves of $\mathcal{O}^s$ and the remaining 2 opposite sides lie along leaves of $\mathcal{O}^u$, and such that the restrictions of $\mathcal{O}^s$ and $\mathcal{O}^u$ to the rectangle foliate it as a product, i.e. conjugate to the foliations of $[0,1]^2 \backslash \{(1,1)\}$ by vertical and horizontal lines. See \Cref{fig:perfectfitdefn}. 

\begin{figure}
    \centering
    \fontsize{18pt}{18pt}\selectfont
    \resizebox{!}{3.5cm}{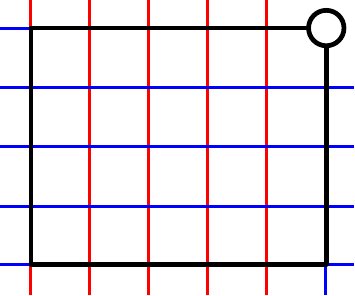}
    \caption{A perfect fit rectangle.}
    \label{fig:perfectfitdefn}
\end{figure}

The collection of orbits $\widetilde{\mathcal{C}}$ can be regarded as a set of points in $\mathcal{O}$. We will say that \textit{$\phi$ has no perfect fits relative to $\mathcal{C}$} if there are no perfect fit rectangles in $\mathcal{O}$ disjoint from $\widetilde{\mathcal{C}}$.
\end{defn}

We recall some definitions in the study of pseudo-Anosov flows.

\begin{defn} \label{defn:Markovpartition}
Given a pseudo-Anosov flow on a closed 3-manifold $N$, a \textit{flow box} is a set of the form $I_s \times I_u \times [0,1]_t \subset N$, where $I_s, I_u$ are intervals, such that:
\begin{itemize}
    \item Every $\{s\} \times \{u\} \times [0,1]_t$ lies along a flow line, with $t$ decreasing being the flow direction
    \item Every $I_s \times \{u\} \times [0,1]_t$ lies along a leaf of the stable foliation $\Lambda^s$
    \item Every $\{s\} \times I_u \times [0,1]_t$ lies along a leaf of the unstable foliation $\Lambda^u$
\end{itemize}

A \textit{Markov partition} is a collection of flow boxes $\{ I^{(i)}_s \times I^{(i)}_u \times [0,1]_t \}_i$ covering $N$ with disjoint interiors, such that 
$$(I^{(i)}_s \times I^{(i)}_u \times \{1\}) \cap (I^{(j)}_s \times I^{(j)}_u \times \{0\}) = \bigcup_k J^{(ij,k)}_s \times I^{(i)}_u \times \{1\} = \bigcup_k I^{(j)}_s \times J^{(ji,k)}_u \times \{0\}$$
for some finite collection of subintervals $J^{(ij,k)}_s \subset I^{(i)}_s$ and $J^{(ji,k)}_u \subset I^{(j)}_u$. Intuitively, when flowing downwards, the flow boxes stretch over multiple flow boxes in the unstable direction and contract to only cover a portion of a flow box in the stable direction. 

Define a directed graph $G$ by letting the set of vertices be the flow boxes, and putting an edge from $(I^{(j)}_s \times I^{(j)}_u \times [0,1]_t)$ to $(I^{(i)}_s \times I^{(i)}_u \times [0,1]_t)$ for every $J^{(ij,k)}_s$. 

Notice that $G$ has a natural embedding in $N$ by placing the vertices in the interior of the corresponding flow box and placing the edges through the corresponding intersections $J^{(ij,k)}_s \times I^{(i)}_u \times \{1\}$. The product structure of the flow boxes determines a framing on the edges in $N$. Also, the sets of incoming and outgoing edges at each vertex of $G$ have natural planar orderings given by the positions of $J^{(ij,k)}_{s/u}$ in $I^{(i)}_{s/u}$.

$G$ together with the information of its embedding in $N$ is said to \textit{encode the Markov partition}. Sometimes we will abuse notation and consider the framing of the edges of $G$ in $N$ and planar orderings of the incoming and outgoing edges at each vertex as part of the data of $G$.
\end{defn}

Markov partitions allow one to study pseudo-Anosov flows using symbolic dynamics. For example, it is a standard fact that if one has a Markov partition of a pseudo-Anosov flow $\phi$ which is encoded by $G \subset N$, then for every closed loop carried by $G$, there is a closed orbit of $\phi$ homotopic to it; conversely, for every closed orbit of $\phi$, there is a closed loop carried by $G$ homotopic to some multiple of it, see for example \cite[Corollary 5.16]{AT22}.

The additional data of the framing of the edges and planar orderings of the incoming and outgoing edges at each vertex allows one to upgrade this statement from `homotopic' to `isotopic', at least for primitive loops and orbits, since one can now decide the relative positions of a loop carried by $G$ when it passes through edges and vertices multiple times. 

In fact, one can essentially recover a Markov partition from the graph that encodes it along with this additional data: Place a flow box at each vertex of the graph and connect up the corresponding flow boxes along their top and bottom faces for each edge, according to the framing and the planar orderings. The side faces of the union of flow boxes can then be glued up, since the semiflow on each component of these must consist of a closed orbit and orbits spiralling into or out of the closed orbit. For more information, see \cite[Sections 3.1-3.4]{Mos96}.

\begin{defn} \label{defn:orbitequiv}
Two flows $\phi_1$ and $\phi_2$ on a 3-manifold $N$ are said to be \textit{orbit equivalent} if there is a homeomorphism of $N$ taking the flow lines of $\phi_1$ to those of $\phi_2$ in an orientation preserving way (but not necessarily respecting the parametrization of the flow lines).

We will often abuse notation and consider two pseudo-Anosov flows as the same if they are orbit equivalent.
\end{defn}

To state \Cref{thm:vtpAcorr} below, we introduce some shorthand notations. Let $\overline{M}$ be a compact 3-manifold with torus boundary components, and let $M$ be its interior. Let $C$ be the set of boundary components of $\overline{M}$, which can be canonically identified with the set of ends of $M$. Given a collection of curves on each boundary component of $\overline{M}$, $s=(s_i)_{i \in C}$, we write $M(s)$ for the closed 3-manifold obtained by Dehn filling $\overline{M}$ along $s_i$. Also, given two collections of (multi-)curves on each boundary component of $\overline{M}$, $s=(s_i)_{i \in C}$ and $t=(t_i)_{i \in C}$, we write $| \langle s, t \rangle | \geq k$ to mean that the geometric intersection number between $s_i$ and $t_i$ is greater or equal to $k$, for every $i \in C$.

\begin{thm} \label{thm:vtpAcorr}
Let $M$ be the interior of a compact 3-manifold with torus boundary components, and let $C$ be the set of ends of $M$. Given a veering triangulation $\Delta$ on $M$, let $l$ be the collection of ladderpole classes on the ends of $M$. Then for every collection of slopes $s$ on the ends of $M$ such that $| \langle s, l \rangle | \geq 2$, $M(s)$ carries a pseudo-Anosov flow $\phi$. Moreover, we have the following properties of $\phi$:
\begin{enumerate}[label=(\alph*)]
    \item There exist closed orbits $c_i$ isotopic to cores of the filling solid tori. Each $c_i$ is $|\langle s_i, l_i \rangle |$-pronged, and $\phi$ has no perfect fits relative to the collection $\{c_i\}$.
    \item The unstable branched surface $B$ carries the unstable lamination of $\phi$ (which is obtained by blowing air into the singular leaves of the unstable foliation).
    \item The reduced flow graph $\Phi_{\red}$ of $\Delta$ encodes a Markov partition of $\phi$. This includes the data of the framing of its edges and the planar orderings of the incoming and outgoing edges at each vertex.
\end{enumerate}
\end{thm}

\begin{proof}
The existence of a pseudo-Anosov flow on $M(s)$ was first proven by Schleimer and Segerman. Their construction appears in \cite{SS23} and will be further elaborated on in \cite{SS22c}. Additional properties (a) and (b) are satisfied by their construction. Meanwhile, an alternate construction has been written up in \cite[Section 5]{AT22}. Additional properties (a)-(c) are satisfied by this construction, see \cite[Theorem 5.1, Proposition 5.13, Proposition 5.15]{AT22} respectively. 
\end{proof}

As remarked in the introduction, Schleimer and Segerman's construction provides a correspondence between veering triangulations and pseudo-Anosov flows, in a suitable sense. 
At the time of writing, the complete proof of this fact is not yet available, but see the introduction of \cite{SS19} for an outline of the proof.

\subsection{Geodesic flows} \label{subsec:geodflow}

We recall some basic facts about geodesic flows. For a more detailed introduction, we refer the reader to \cite[Section 2]{Deh15}. 

\begin{defn} \label{defn:geodflow}
Let $\Sigma$ be a closed 2-dimensional orbifold with a Riemannian metric $g$. Consider its unit tangent bundle $T^1 \Sigma = \{v \in T \Sigma : ||v||_g=1 \}$. The \textit{geodesic flow} $\phi_t$ on this 3-manifold is defined by $\phi_t(\gamma'(0))=\gamma'(t)$ for every unit speed geodesic $\gamma$.

Let $c$ be a collection of oriented geodesics on $\Sigma$. Suppose elements of $c$ are parametrized with unit speed. The \textit{lift} of $c$ in $T^1 \Sigma$ is defined to be $\overset{\rightarrow}{c}:=\{c'_i(t):\text{$c_i$ is an element of $c$} \}$. This is a collection of orbits of the geodesic flow.

Similarly, let $c$ be a collection of unoriented geodesics on $\Sigma$. Suppose elements of $c$ are parametrized with unit speed in some orientation. Then the \textit{full lift} of $c$ in $T^1 \Sigma$ is defined to be $\overset{\leftrightarrow}{c}:=\{\pm c_i'(t):\text{$c_i$ is an element of $c$} \}$. This is again a collection of orbits of the geodesic flow. 
\end{defn}

When $g$ has negative curvature everywhere, it is a classical fact that the geodesic flow $\phi$ is an Anosov flow. This is first proven in \cite{Hop39}, but see the appendix in \cite{Bal95} for a more modern exposition. It is also well known that Anosov flows are structurally stable, i.e. the orbit equivalence class is preserved under any $C^1$ perturbation to the underlying vector field, see \cite{Rob74} for a proof. Together with the fact that the space of negatively curved Riemannian metrics on a 2-dimensional orbifold with negative Euler characteristic is nonempty and connected (\cite{Ham88}), this means that we can talk about \textit{the geodesic flow} on the unit tangent bundle of an orbifold $S$ with negative Euler characteristic when we mean the geodesic flow under some negatively curved Riemannian metric. Also, given a collection of homotopically nontrivial and mutually nonparallel (un)oriented curves $c$, we can talk about \textit{the (full) lift} of $c$ when we mean the (full) lift of the geodesic representative of $c$ under some negatively curved Riemannian metric.



One reason why geodesic flows are important to the study of pseudo-Anosov flows is that they account for virtually all pseudo-Anosov flows on Seifert fibered 3-manifolds. More precisely,

\begin{thm} \label{thm:sfspageod}
Let $N$ be a Seifert fibered space carrying a pseudo-Anosov flow $\phi$. Then $N$ is a finite cover of the unit tangent bundle of some hyperbolic orbifold $T^1 \Sigma$, and $\phi$ is orbit equivalent to the lift of the geodesic flow on $T^1 \Sigma$. 
\end{thm}

\begin{proof}
This is essentially proved in \cite{BF13} and \cite{Bar95}. We provide a sketch of the argument found across the two papers.

Let $S$ be the base orbifold of $N$ and let $h \in \pi_1 N$ be the class of a regular fiber. The proof of \cite[Theorem 4.1]{BF13} starts by analyzing the action of $h$ on the leaf space $\mathcal{H}^s$ of the stable foliation on $\widetilde{N}$, showing that it is homeomorphic to $\mathbb{R}$. This means that $\phi$ is an $\mathbb{R}$-covered Anosov flow. In fact, by \cite[Theorem 2.8]{Bar95}, $\phi$ must be a skew $\mathbb{R}$-covered Anosov flow or else $N$ would have $\mathrm{Solv}$ geometry.

Now in general from a skew $\mathbb{R}$-covered Anosov flow, one can construct a step map $\tau_s:\mathcal{H}^s \to \mathcal{H}^s$ such that $\mathcal{H}^s/\tau_s$ is a circle, and such that the action of $\pi_1 N$ on $\mathcal{H}^s$ descends to an action on this circle. In the proof of \cite[Theorem 4.1]{BF13}, by further analyzing the action of $\pi_1 N$ on $\mathcal{H}^s$, it is shown that $h$ acts trivially on $\mathcal{H}^s/\tau_s$, and the quotiented action of $\pi_1 N/ \langle h \rangle = \pi^{orb}_1 S$ on $\mathcal{H}^s/\tau_s$ is a convergence group action. Hence by \cite{Gab92} or \cite{CJ94}, $\pi^{orb}_1 S$ can be conjugated to a Fuchsian group. Let $\rho: \pi^{orb}_1 S \to \mathrm{PSL}_2 \mathbb{R}$ be this Fuchsian representation, which bestows $S$ with a hyperbolic structure $\Sigma$. Also, let $\widetilde{\rho}: \pi_1 N \to \widetilde{\mathrm{PSL}_2 \mathbb{R}}$ be the lift of $\rho$.

Denoting the element of $\widetilde{\mathrm{PSL}_2 \mathbb{R}}$ $x$ to $x+l$ by $sh(l)$ and conjugating such that $\tau_s=sh(1)$, $\widetilde{\rho}(h)$ must be of the form $sh(r)$ for some $r \in \mathbb{Z}$. Since $\rho$ is a Fuschian representation, for $H=\langle \widetilde{\rho}(\pi_1 N), sh(1) \rangle$, $\widetilde{\mathrm{PSL}_2 \mathbb{R}}/H \cong \mathrm{PSL}_2 \mathbb{R}/\rho(\pi^{orb}_1 S) \cong T^1 \Sigma$. Hence $\widetilde{\mathrm{PSL}_2 \mathbb{R}} / \widetilde{\rho}(\pi_1 N) \cong N$ is the $|r|^{th}$ fiberwise cyclic cover of $T^1 \Sigma$. 

To prove the second part of the statement, note that $\widetilde{\mathrm{PSL}_2 \mathbb{R}}/H \cong T^1 \Sigma$ carries the geodesic flow, which is a skew $\mathbb{R}$-covered Anosov flow. Its cover $\widetilde{\mathrm{PSL}_2 \mathbb{R}}/\widetilde{\rho}(\pi_1 N) \cong N$ thus carries the lifted flow, which is also skew $\mathbb{R}$-covered Anosov. It is easy to see that the action of $\pi_1 N$ on the leaf space of the stable foliation of this Anosov flow is exactly $\widetilde{\rho}$. Hence by \cite[Theorem 4.6]{Bar95}, the original flow is orbit-equivalent to this lift of the geodesic flow on $T^1 \Sigma$. 
\end{proof}




We also recall the following classical theorem. This follows from the more general statement of \cite[Theorem 1.3]{CR20}.

\begin{thm} \label{thm:fulllifthyp}
Let $S$ be a 2-dimensional closed orbifold with negative Euler characteristic, and let $c$ be a filling collection of homotopically nontrivial and mutually nonparallel curves. Then $T^1 S \backslash \overset{\leftrightarrow}{c}$ is a hyperbolic 3-manifold.
\end{thm}

\subsection{Montesinos links} \label{subsec:montesinos}

We recall some notation and basic facts about Montesinos links. We refer to \cite{BS09} for more detailed explanations.

\begin{defn}
Let $S$ be a 2-sphere with 4 marked points. Label the 4 marked points as NE, NW, SW, SE. Suppose $S$ bounds a 3-ball $B$, fix a projection of $B$ to a disc, so that the 4 marked points are mapped to the NE, NW, SW, SE corners of the disc respectively. For us, a \textit{tangle} will mean the projection of a (tame) embedding of two arcs in $B$ where the endpoints of the arcs lie on the 4 marked points.

For each finite sequence of integers $(a_0,...,a_k)$, define a tangle in the following way. If $k$ is even, start with two disjoint arcs connecting NE with NW and SW with SE, then add $a_k$ half twists around the NW and SW corners, then $-a_{k-1}$ half twists around the NE and NW corners, and so on. If $k$ is odd, start with two disjoint arcs connecting NE with SE and NW with SW, then add $-a_k$ half twists around the NE and NW corners, then $a_{k-1}$ half twists around the NW and SW corners, and so on.

We call such a tangle a \textit{rational tangle}, and associate to it the rational number $a_0 + \frac{1}{a_1+\frac{1}{...}}$. In \Cref{fig:montesinos} top left we illustrate an example of a rational tangle which corresponds to the rational number $2 + \frac{1}{3+\frac{1}{4}}$.

A \textit{Montesinos link} is a knot or link obtained by inserting rational tangles into the empty regions of the knot diagram illustrated in \Cref{fig:montesinos} bottom. If the rational numbers associated to the rational tangles we inserted are $\frac{q_1}{p_1}, ..., \frac{q_n}{p_n}$, we denote the corresponding Montesinos link by $M(\frac{q_1}{p_1},...,\frac{q_n}{p_n})$. We remark that the same knot or link may be representable as a Montesinos link for various different choices of $\frac{q_i}{p_i}$.
\end{defn}

\begin{figure}
    \centering
    \fontsize{14pt}{14pt}\selectfont
    \resizebox{!}{8cm}{
\begingroup%
  \makeatletter%
  \providecommand\color[2][]{%
    \errmessage{(Inkscape) Color is used for the text in Inkscape, but the package 'color.sty' is not loaded}%
    \renewcommand\color[2][]{}%
  }%
  \providecommand\transparent[1]{%
    \errmessage{(Inkscape) Transparency is used (non-zero) for the text in Inkscape, but the package 'transparent.sty' is not loaded}%
    \renewcommand\transparent[1]{}%
  }%
  \providecommand\rotatebox[2]{#2}%
  \newcommand*\fsize{\dimexpr\f@size pt\relax}%
  \newcommand*\lineheight[1]{\fontsize{\fsize}{#1\fsize}\selectfont}%
  \ifx\svgwidth\undefined%
    \setlength{\unitlength}{380.38910602bp}%
    \ifx\svgscale\undefined%
      \relax%
    \else%
      \setlength{\unitlength}{\unitlength * \real{\svgscale}}%
    \fi%
  \else%
    \setlength{\unitlength}{\svgwidth}%
  \fi%
  \global\let\svgwidth\undefined%
  \global\let\svgscale\undefined%
  \makeatother%
  \begin{picture}(1,0.77212015)%
    \lineheight{1}%
    \setlength\tabcolsep{0pt}%
    \put(0.79498788,0.64563881){\color[rgb]{0,0,0}\makebox(0,0)[lt]{\lineheight{1.25}\smash{\begin{tabular}[t]{l}$a_0+\frac{1}{a_1+\frac{1}{a_2}}$\end{tabular}}}}%
    \put(0,0){\includegraphics[width=\unitlength,page=1]{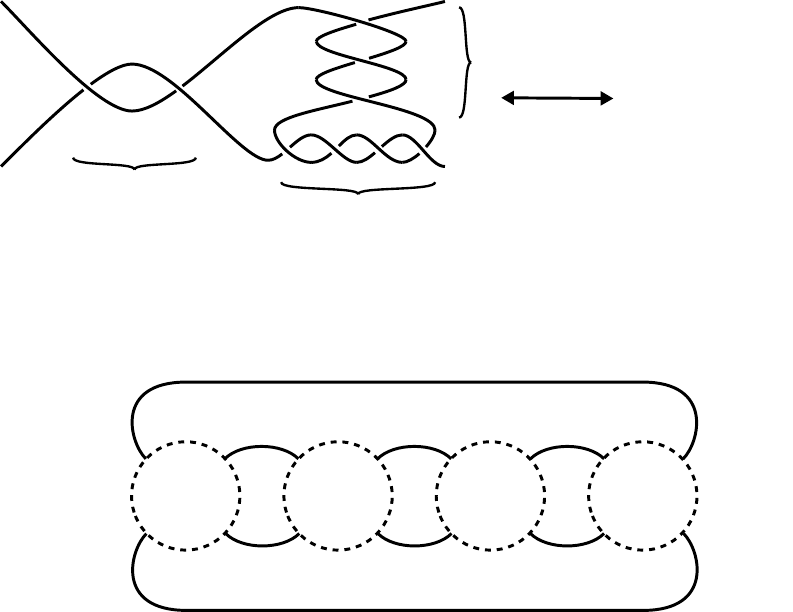}}%
    \put(0.60377337,0.68791883){\color[rgb]{0,0,0}\makebox(0,0)[lt]{\lineheight{1.25}\smash{\begin{tabular}[t]{l}$a_1$\end{tabular}}}}%
    \put(0.43713559,0.49350626){\color[rgb]{0,0,0}\makebox(0,0)[lt]{\lineheight{1.25}\smash{\begin{tabular}[t]{l}$a_2$\end{tabular}}}}%
    \put(0.15089212,0.52413747){\color[rgb]{0,0,0}\makebox(0,0)[lt]{\lineheight{1.25}\smash{\begin{tabular}[t]{l}$a_0$\end{tabular}}}}%
    \put(0,0){\includegraphics[width=\unitlength,page=2]{montesinos.pdf}}%
  \end{picture}%
\endgroup%
}
    \caption{Definition of a Montesinos link.}
    \label{fig:montesinos}
\end{figure}

A key fact about Montesinos links is that they are exactly those knots and links whose double branched cover is a Seifert fibered space. In fact, the double branched cover of $M(\frac{q_1}{p_1},...,\frac{q_n}{p_n})$ is the Seifert fibered space with base orbifold $S^2$ with cone points of index $p_1,...,p_n$ and singular fibers above those cone points having parameters $(p_i,q_i)$. In the sequel, we will denote such an orbifold by $S^2(p_1,..,p_n)$ and such a Seifert fibered space as $(S^2, (p_1,q_1),...,(p_n,q_n))$.

There is a neat way of seeing how this double branched cover works. Take a $n$-gon $Q$ with edges $l_1,...,l_n$ and consider a trivial circle bundle $T$ over $Q$. Choose parametrizations $l_i \times \mathbb{R}/\mathbb{Z}$ of the bundle over $l_i$, in a way such that the second coordinates in the parametrizations shift down by $\frac{q_i}{2p_i}$ going from $l_i \times \mathbb{R}/\mathbb{Z}$ to $l_{i+1} \times \mathbb{R}/\mathbb{Z}$ (here indices should be taken mod $n$). Now this shift can only be well-defined mod $\mathbb{Z}$, so there is still the ambiguity of how the parametrizations fit together when going around $\partial T$. We fix this by requiring that in the universal cover of $T$, if we start at $l_1 \times \{0 \}$ and follow horizontal lines $l_i \times \{t \}$ around $\partial Q \times \{t\}$, we will return to $l_1 \times \{\frac{e}{2}\}$ for $e=\sum \frac{q_i}{p_i}$. 

Now define an involution $\iota_i$ on each $l_i \times \mathbb{R}/\mathbb{Z}$ by reflecting across the horizontal lines $l_i \times \{0 \}$ and $l_i \times \{ \frac{1}{2} \}$. Then $T$ quotiented by the $\iota_i$ on its faces is homeomorphic to $S^3$, and the lines of reflection form the Montesinos link $M(\frac{q_1}{p_1},...,\frac{q_n}{p_n})$. Intuitively, we are folding up each face of $T$, but on the fibers above the vertices of $Q$, the two foldings differ by a shift, so those fibers are folded up in $2p_i$-ply fashion, from the action of a dihedral group. 

From this picture, we can construct the double branched cover of $M(\frac{q_1}{p_1},...,\frac{q_n}{p_n})$ by taking two copies of $T$ and gluing their faces together via $\iota_i$. It can be seen that the result will be the Seifert fibered space $(S^2, (p_1,q_1),...,(p_n,q_n))$. 

In this paper, we will only be interested in the cases when the Seifert fibered space $(S^2, (p_1,q_1),...,(p_n,q_n))$ is the unit tangent bundle over an orbifold. This is the case when $\frac{q_1}{p_1}=\frac{1}{p_1}+1$ and $\frac{q_i}{p_i}=\frac{1}{p_i}-1$ for $i \neq 1$. In this case, the picture above simplifies, and can be illustrated as in \Cref{fig:monteblock}.

\begin{figure}
    \centering
    \resizebox{!}{5.5cm}{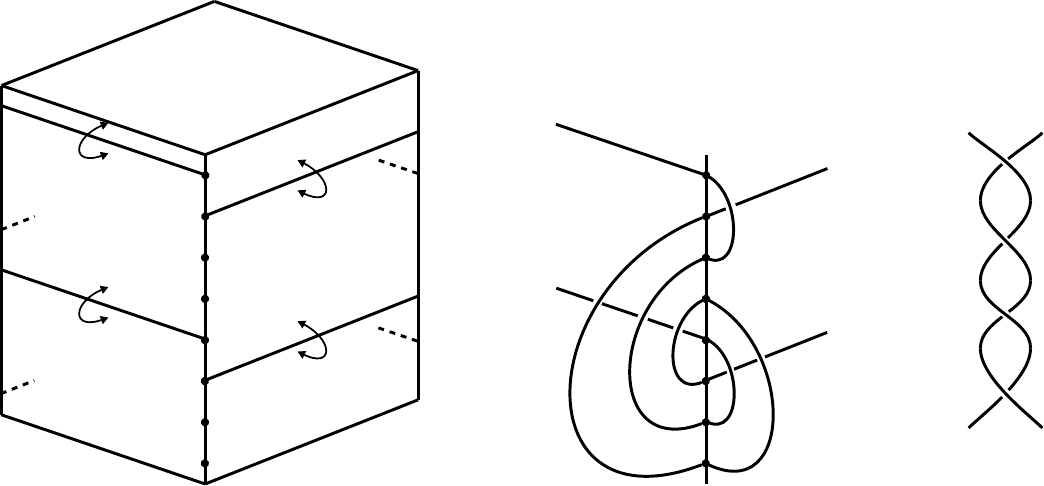}
    \caption{Understanding the double branched cover of $M(\frac{1}{p_1}+1,\frac{1}{p_2}-1,...,\frac{1}{p_n}-1)$ near one of the rational tangles with $p_i=4$.}
    \label{fig:monteblock}
\end{figure}

The reason why we mentioned this perspective is that it allows one to see why these Montesinos links are related to geodesic flows. Give $Q$ a Riemannian metric where the sides $l_i$ are geodesics and the angle between $l_i$ and $l_{i+1}$ is $\frac{\pi}{p_i}$. Then the unit tangent bundle of $Q$ is a trivial circle bundle over $Q$. If we orient the $l_i$ coherently, say, to all go from $l_{i-1}$ to $l_{i+1}$, we can then choose parametrizations of the bundle over $l_i$ using parallel transport, where, say, $l_i \times \{0 \}$ is the lift of $l_i$ and $l_i \times \{ \frac{1}{2} \}$ is the lift of $-l_i$. It is straightforward to see that these parametrizations satisfy the conditions imposed above, and the involutions $\iota_i$ we defined are induced by reflections across $l_i$.

Now if we construct the double branched cover of $M(\frac{1}{p_1}+1,\frac{1}{p_2}-1,...,\frac{1}{p_n}-1)$ by gluing together two copies of $T^1 Q$ as above, we get the unit tangent bundle over the orbifold obtained by doubling $Q$, which is exactly $S^2(p_1,...,p_n)$. The boundary of $Q$ is a geodesic $c$ in this orbifold, and the double branched covering is induced by reflection across $c$.  In particular, the Montesinos link $M(\frac{1}{p_1}+1,\frac{1}{p_2}-1,...,\frac{1}{p_n}-1)$ is the image of $\overset{\leftrightarrow}{c}$.

We record this fact as a proposition.

\begin{prop} \label{prop:montelink}
Let $n \geq 1$, $p_1,...,p_n \geq 2$. Let $c$ be a curve that passes through the cones points of order $p_1,...,p_n$ in $S^2(p_1,...,p_n)$ in that order. Reflection across $c$ induces a double branched cover $T^1S^2(p_1,...,p_n) \to S^3$. The branch locus of the covering in $T^1S^2(p_1,...,p_n)$ is the full lift $\overset{\leftrightarrow}{c}$. The branch locus of the covering in $S^3$ is the Montesinos link $M(\frac{1}{p_1}+1,\frac{1}{p_2}-1,...,\frac{1}{p_n}-1)$.
\end{prop}

\section{Veering branched surfaces} \label{sec:vbs}

In this section, we will introduce the notion of veering branched surfaces. Their definition is modeled after the properties of the unstable branched surface listed in \Cref{prop:branchsurf}, and they end up being essentially equivalent to veering triangulations, hence the name. 

\begin{defn} \label{defn:vbs}
Let $M$ be the interior of a compact 3-manifold with torus boundary components, and let $B$ be a branched surface in $M$. $B$ along with a choice of orientations on the components of its branch locus is \textit{veering} if:
\begin{enumerate}[label=(\roman*)]
    \item Each sector of $B$ is homeomorphic to a disc.
    \item Each component of $M \cut B$ is a cusped solid torus or a cusped torus shell.
    \item At each triple point, the orientation of each component of $\brloc(B)$ induces the maw coorientation on the other component.
\end{enumerate}

We will often abuse notation and consider the orientations on the components of the branch locus as part of the data of $B$. 
\end{defn}

\begin{prop} \label{prop:vbs}
Let $B$ be a veering branched surface in an oriented 3-manifold $M$ whose complementary regions are all cusped torus shells. Then the dual ideal triangulation of $B$ is a veering triangulation $\Delta$ of $M$, and $B$ can be identified with the unstable branched surface for $\Delta$.
\end{prop}

\begin{proof}
The orientations on the components of the branch locus induce orientations on the sides of all the sectors of $B$. With this in mind, we first show the following lemma:

\begin{lemma} \label{lemma:vbssector}
Let $B$ be a veering branched surface. Then each sector of $B$ has 4 corners. The orientations on the sides flip on two opposite corners. 
\end{lemma}

\begin{proof}
Because of (iii), at each corner of a sector, the orientation of one side must induce the maw coorientation on the other side. This implies that the orientations on the sides flip on every other corner, hence the number of corners is divisible by $4$. Together with (i), we deduce that each sector is a $4n$-gon, which has index $1-n$.

Fix a component $C$ of $M \cut B$. By (ii), $\partial C$ is a torus which is divided into annuli by the cusp circles. Meanwhile, the intersection of $\partial C$ with $\brloc(B)$ is a graph, the complementary regions of which in $\partial C$ are among the sectors of $B$. Suppose one of these complementary regions is a disc with no corners. Then the disc must lie in the interior of one of the annuli. But if the boundary of the disc has maw coorientation pointing outwards, then the complementary region of $B$ on the other side of the disc cannot be a cusped torus shell, and if the maw coorientation is pointing inwards, then the sector on $\partial C$ surrounding the disc cannot be homeomorphic to a disc. Hence we deduce that all sectors which appear on $\partial C$ have nonpositive index. But their indices must add up to give the index of a torus, which is $0$, so their indices have to be all $0$. Together with the fact that each sector must lie on the boundary of some complementary region of $B$, this proves the lemma.
\end{proof}

In particular, we now know that every component of $\brloc(B)$ must meet at least one triple point. Hence we can define a cellular structure on $B$ as in \Cref{defn:branchsurf} and talk about its dual ideal triangulation $\Delta$.

We define a veering structure on $\Delta$. Each face of $\Delta$ is dual to a 1-cell in $\brloc(B)$. We coorient the face with the coorientation opposite to the one induced by the orientation on the 1-cell. 

Each tetrahedron of $\Delta$ is dual to a triple point of $B$, which is adjacent to four 1-cells in $\brloc(B)$, two of which are oriented inwards and two of which are oriented outwards. This implies that among the four faces of each tetrahedron, two are cooriented inwards and two are cooriented outwards. This induces a natural choice of dihedral angles among $\{0, \pi \}$ on each edge of a tetrahedron. \Cref{lemma:vbssector} implies that the angle sum around each edge is $2\pi$. Hence we have defined a transverse taut structure on $\Delta$.

Finally, each triple point of $B$ is of one of the two forms illustrated in \Cref{fig:vbslocal}. (Note that this uses the assumption that $M$ is oriented.) We color a triple point blue if it is of the form on the left, and color it red if it is of the form on the right. Each edge of $\Delta$ is the top edge of a unique tetrahedron, and we color the edge with the same color as the triple point dual to this tetrahedron. A good mnemonic for this is that an edge of $\Delta$ is b\textbf{L}ue or \textbf{R}ed if the sector of $B$ dual to it has `fins' on the bottom spiralling in the left-handed or right-handed direction respectively. 

\begin{figure}
    \centering
    \fontsize{30pt}{30pt}\selectfont
    \resizebox{!}{5cm}{
\begingroup%
  \makeatletter%
  \providecommand\color[2][]{%
    \errmessage{(Inkscape) Color is used for the text in Inkscape, but the package 'color.sty' is not loaded}%
    \renewcommand\color[2][]{}%
  }%
  \providecommand\transparent[1]{%
    \errmessage{(Inkscape) Transparency is used (non-zero) for the text in Inkscape, but the package 'transparent.sty' is not loaded}%
    \renewcommand\transparent[1]{}%
  }%
  \providecommand\rotatebox[2]{#2}%
  \newcommand*\fsize{\dimexpr\f@size pt\relax}%
  \newcommand*\lineheight[1]{\fontsize{\fsize}{#1\fsize}\selectfont}%
  \ifx\svgwidth\undefined%
    \setlength{\unitlength}{440.71362541bp}%
    \ifx\svgscale\undefined%
      \relax%
    \else%
      \setlength{\unitlength}{\unitlength * \real{\svgscale}}%
    \fi%
  \else%
    \setlength{\unitlength}{\svgwidth}%
  \fi%
  \global\let\svgwidth\undefined%
  \global\let\svgscale\undefined%
  \makeatother%
  \begin{picture}(1,0.49482951)%
    \lineheight{1}%
    \setlength\tabcolsep{0pt}%
    \put(0,0){\includegraphics[width=\unitlength,page=1]{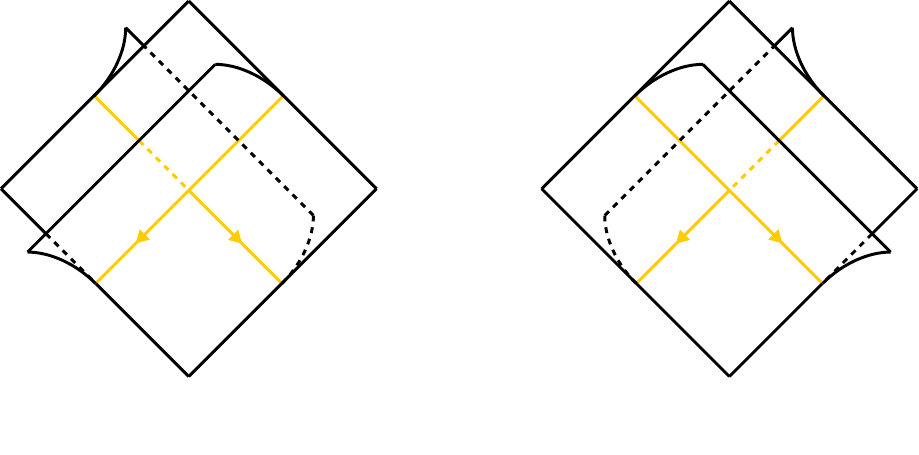}}%
    \put(0.19216899,0.00719935){\color[rgb]{0,0,1}\makebox(0,0)[lt]{\lineheight{1.25}\smash{\begin{tabular}[t]{l}$L$\end{tabular}}}}%
    \put(0.78071764,0.00719935){\color[rgb]{1,0,0}\makebox(0,0)[lt]{\lineheight{1.25}\smash{\begin{tabular}[t]{l}$R$\end{tabular}}}}%
  \end{picture}%
\endgroup%
}
    \caption{Defining the color of a triple point of a veering branched surface in an oriented 3-manifold: The triple point on the left is blue while that on the right is red.}
    \label{fig:vbslocal}
\end{figure}

To verify that this makes $\Delta$ into a veering triangulation, we look at how the edges are colored in each tetrahedron. There are two cases here depending on whether the dual triple point is colored blue or red, and in either case we see that the coloring satisfies \Cref{defn:vt}. It is also straightforward to see that $B$ is then the unstable branched surface of $\Delta$.
\end{proof}

In other words, given a veering branched surface $B$ on an orientable 3-manifold $N$, one can drill out the cores of all the cusped solid torus components of $N \cut B$ to get a 3-manifold $M$ which carries a veering triangulation $\Delta$, obtained by taking the dual ideal triangulation to $B$ in $M$. From now on, for simplicity, we will say that $\Delta$ is the \textit{dual veering triangulation} to $B$, and conversely, $B$ is the \textit{dual veering branched surface} to $\Delta$.

\begin{rmk} \label{rmk:vbsnonori}
Notice that the definition of a veering triangulation only makes sense on an oriented 3-manifold, while the definition of a veering branched surface makes sense for non-orientable 3-manifolds as well. Hence veering branched surfaces serve as a generalization of veering triangulations to the non-orientable case. In the language of \cite{SS21}, the dual ideal triangulations to veering branched surfaces in general are what might be called transverse locally veering triangulations. 

\end{rmk}

We take the time to explain how to recover the flow graph $\Phi$ of a veering triangulation $\Delta$ using just the data of its dual veering branched surface $B$: The set of vertices of $\Phi$ is equal to the set of sectors of $B$, and there are 3 edges for each triple point, going from the 3 sectors which the maw coorientations on the two components of the branch locus passing through the triple point are pointing away from, into the sector which the maw coorientations are pointing into. The embedding of $\Phi$ in $B$ can also be recovered by placing each vertex at the corner of the corresponding sector which the sides are oriented away from, and placing the edges that exit that vertex within that sector. 

We will also slightly generalize the notion of veering branched surfaces to what we call almost veering branched surfaces. Despite not having as much relation to veering triangulations, these serve as good intermediate objects when trying to build veering branched surfaces. 

\begin{defn} \label{defn:avbs}
Let $M$ be the interior of a compact 3-manifold with torus boundary components, and let $B$ be a branched surface in $M$. $B$ along with a choice of orientations on the components of its branch locus is \textit{almost veering} if:
\begin{enumerate}[label=(\roman*)]
    \item No sector of $B$ is a disc without corners.
    \item Each component of $M \cut B$ is a cusped solid torus or a cusped torus shell.
    \item At each triple point, the orientation of each component of $\brloc(B)$ induces the maw coorientation on the other component.
\end{enumerate}

As in \Cref{defn:vbs}, we will often abuse notation and consider the orientations on the components of the branch locus as part of the data of $B$. 
\end{defn}

We have the following analogue of \Cref{lemma:vbssector}.

\begin{lemma} \label{lemma:avbssector}
Let $B$ be an almost veering branched surface. Each sector of $B$ is homeomorphic to a disc, an annulus, or a M\"obius band. If a sector of $B$ is homeomorphic to a disc, it has 4 corners and the orientations on the sides flips on two opposite corners. If a sector of $B$ is homeomorphic to an annulus or a M\"obius band, it has no corners.
\end{lemma}

\begin{proof}
One can prove this using the same argument as in \Cref{lemma:vbssector}: First observe that (iii) implies the number of corners in each sector is divisible by $4$. In this case, sectors that are discs with no corners are explicitly forbidden by (i), so the indices of all sectors are nonpositive. Finally, an additivity argument implies that all these indices are $0$.
\end{proof}

We also have the following observation.

\begin{prop} \label{prop:avbshyp}
If $B$ is an almost veering branched surface on a hyperbolic 3-manifold $M$ for which all components of $M \cut B$ are cusped torus shells, then $B$ is automatically a veering branched surface.
\end{prop}

\begin{proof}
If $M$ is orientable and there were any annulus sectors, then the complementary regions of $B$ on the two sides of the sector would contain parallel ends, which is impossible for hyperbolic $M$. If $M$ is non-orientable or if there are any M\"obius band sectors, this argument gives a contradiction in a finite cover of $M$.
\end{proof}

Before we end this section, we explain some notation that we will be using in the rest of this paper.

Given an almost veering branched surface $B$ in a 3-manifold $M$, recall that there are implicitly chosen orientations for the components of its branch locus. We will also implicitly coorient the components of the branch locus by the maw coorientation, unless otherwise stated. These orientations and coorientations will be denoted by arrows in the figures. 

Also, when $M$ is oriented, we will say that a triple point of $B$ is blue or red according to \Cref{fig:vbslocal}, as in the proof of \Cref{prop:vbs}. From that proof, we note that if $B$ is veering and dual to a veering triangulation $\Delta$, then the number of blue or red triple points of $B$ is equal to the number of blue or red edges in $\Delta$ respectively.

\section{Surgeries on veering branched surfaces} \label{sec:surgery}

In this section, we will describe some surgical operations one can perform on almost veering branched surfaces. The two basic types are horizontal and vertical surgery, and we will also describe a generalization of horizontal surgery. We will only be using horizontal surgery in our constructions in the rest of this paper, but since vertical surgery admits a very similar description, we introduce it here as well. It is also an interesting question how vertical surgery can interact with the constructions we present in this paper, see \Cref{sec:questions}.

The surgeries will be done along certain types of curves carried by almost veering branched surfaces, hence we make the following preliminary definition. 

\begin{defn} \label{defn:orisurcurve}
Let $B$ be a branched surface in a 3-manifold $M$. Let $\alpha \subset B$ be a smoothly embedded curve which avoids the triple points of $B$.

$\alpha$ is said to be \textit{orientation preserving} if the tangent planes of $B$ along $\alpha$ can be oriented in a coherent way. Otherwise $\alpha$ is said to be \textit{orientation reversing}. Similarly, $\alpha$ is said to be \textit{coorientation preserving} if the tangent planes of $B$ along $\alpha$ can be cooriented in a coherent way, otherwise $\alpha$ is \textit{coorientation reversing}. Note that we do not assume that $M$ is orientable, hence orientation preserving does not imply coorientation preserving (and vice versa).

Let $N$ be a small tubular neighborhood of $\alpha$ in $M$. For $N$ small enough, $N \cap B$ is an annulus or M\"obius band $A$ with sectors attached along disjoint arcs, with each arc corresponding to a point of intersection between $\alpha$ and the branch locus of $B$. We call $A$ a \textit{smooth neighborhood} of $\alpha$ in $B$. $N \cut A$ has one or two components, depending on whether $\alpha$ is coorientation preserving. We refer to the components of $N \cut A$ as the \textit{regular half-neighborhoods} of $\alpha$ in $B$. Note that the restriction of $B$ to each regular half-neighborhood is a branched surface, specifically it is an annulus with sectors attached on one side.
\end{defn}

\subsection{Horizontal surgery} \label{subsec:hsur}

\begin{defn} \label{defn:hsurcurve}
Let $B$ be an almost veering branched surface in a 3-manifold $M$ and let $\alpha \subset B$ be a smoothly embedded curve which avoids the triple points of $B$. Let $N$ be a tubular neighborhood of $\alpha$ in $M$ and let $A$ be a smooth neighborhood of $\alpha$ in $B$.
We say that $\alpha$ is a \textit{horizontal surgery curve} if:
\begin{enumerate}
    \item $\alpha$ is orientation and coorientation preserving; hence $A$ is an annulus and there are two regular half-neighborhoods of $\alpha$ in $B$, which we label as $N_1$ and $N_2$.
    \item The arcs in the branch locus of $N \cap B$ are all oriented from the same boundary component of $A$ to the other.
    \item The arcs in the branch locus of each $N_i  \cap B$ are cooriented in the same direction, but the two directions for the two regular half-neighborhoods are opposite to each other. 
\end{enumerate}

See \Cref{fig:hsur} top for an illustration of a horizontal surgery curve.
\end{defn}

Let $\alpha$ be a horizontal surgery curve on an almost veering branched surface $B$. As the name suggests, we will explain how to do surgery along $\alpha$. First we set up some orientation conventions: Orient $\alpha$ so that the arcs in $\brloc(N_1 \cap B)$ are cooriented coherently as $\alpha$. Note that $(\text{arcs in }\brloc(N_1 \cap B), \alpha)$ determines an orientation of $A$. Orient the meridian $\mu$ of $N_1$ such that the basis $(\mu, \alpha)$ on $\partial N_1$ agrees with this orientation on $A$.

For each $k \geq 0$, one can cut $N_1$ out of $M$, and glue it back with a map that is identity on $\partial N_1 \cut A$ and sends the meridian to a curve of isotopy class $\mu-k \alpha$, such that the arcs in the branch locus of $N_1 \cap B$ intersect those in $N_2 \cap B$ minimally. See \Cref{fig:hsur}. This gives us a branched surface $B'$ in another 3-manifold $M'$, topologically obtained by doing $\frac{1}{-k}$ surgery along $\alpha$ on $M$ (with respect to the basis we chose above). We call this operation a \textit{$\frac{1}{-k}$ horizontal surgery along $\alpha$}.

\begin{rmk} \label{rmk:hsurrmk}
Alternatively, one can also perform the surgery by cutting out $N_2$ and gluing it back appropriately. The resulting branched surface will be isotopic to $B'$ above.
\end{rmk}

Notice that $\alpha \subset B'$ continues to be a horizontal surgery curve (as long as the gluing map on $\partial N_1$ is generic enough to avoid triple intersections of $\alpha$ and the branch loci of $N_1 \cap B$ and $N_2 \cap B$), and if we perform $\frac{1}{-l}$ horizontal surgery on it, the total effect is equivalent to doing $\frac{1}{-k-l}$ horizontal surgery along the original $\alpha \subset B$. For this reason, it often suffices to consider $\frac{1}{-1}$ horizontal surgeries, and for convenience we abbreviate doing $\frac{1}{-1}$ horizontal surgery along $\alpha$ as just doing \textit{horizontal surgery along $\alpha$}.

\begin{prop} \label{prop:hsur}
Let $\alpha$ be a horizontal surgery curve on an almost veering branched surface $B$. Let $B'$ be the branched surface obtained by doing $\frac{1}{-k}$ horizontal surgery along $\alpha$. Then $B'$ is almost veering.

Furthermore, if $B$ is veering then $B'$ is veering as well.
\end{prop}

\begin{proof} 
We need to check conditions (i)-(iii) in \Cref{defn:avbs}. 

Conditions (ii) and (iii) are straightforward. 
For (ii), the complementary regions of $B$ can be canonically identified with those of $B'$.
For (iii), the orientations on the components of $\brloc(B)$ induce orientations on those of $\brloc(B')$, which satisfy (iii) for $k \geq 0$. 

For (i), we separate the sectors of $B'$ into three types and inspect them one by one. The sectors of $B'$ that do not intersect $A$ in their interior are among the sectors of $B$, hence are discs with corners, annuli, or M\"obius bands. The sectors of $B'$ whose interiors lie in $A$ are discs with $4$ corners. Finally, for the sectors of $B'$ that meet $\partial A$ in their interior, these are homeomorphic to complementary regions of $\alpha$ in sectors of $B$ and each contain at least one corner.

In fact, once we know that $B'$ is almost veering, \Cref{lemma:avbssector} implies that the last type of sectors must be discs with corners. If $B$ is veering, the analysis above combines with this observation to show that every sector of $B'$ is a disc, hence $B'$ is veering.
\end{proof}

\begin{figure}
    \centering    
    \fontsize{14pt}{14pt}\selectfont
    \resizebox{!}{6cm}{
\begingroup%
  \makeatletter%
  \providecommand\color[2][]{%
    \errmessage{(Inkscape) Color is used for the text in Inkscape, but the package 'color.sty' is not loaded}%
    \renewcommand\color[2][]{}%
  }%
  \providecommand\transparent[1]{%
    \errmessage{(Inkscape) Transparency is used (non-zero) for the text in Inkscape, but the package 'transparent.sty' is not loaded}%
    \renewcommand\transparent[1]{}%
  }%
  \providecommand\rotatebox[2]{#2}%
  \newcommand*\fsize{\dimexpr\f@size pt\relax}%
  \newcommand*\lineheight[1]{\fontsize{\fsize}{#1\fsize}\selectfont}%
  \ifx\svgwidth\undefined%
    \setlength{\unitlength}{315.21932575bp}%
    \ifx\svgscale\undefined%
      \relax%
    \else%
      \setlength{\unitlength}{\unitlength * \real{\svgscale}}%
    \fi%
  \else%
    \setlength{\unitlength}{\svgwidth}%
  \fi%
  \global\let\svgwidth\undefined%
  \global\let\svgscale\undefined%
  \makeatother%
  \begin{picture}(1,0.96205428)%
    \lineheight{1}%
    \setlength\tabcolsep{0pt}%
    \put(0,0){\includegraphics[width=\unitlength,page=1]{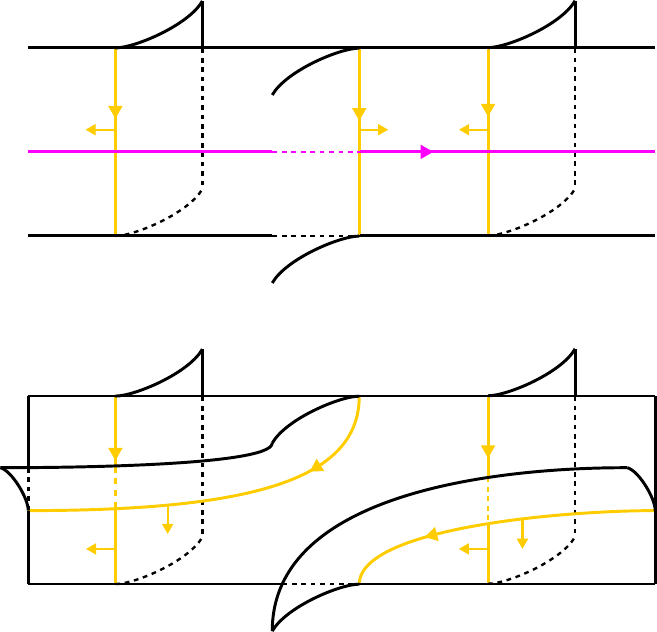}}%
    \put(0.63194487,0.75923268){\color[rgb]{1,0,1}\makebox(0,0)[lt]{\lineheight{1.25}\smash{\begin{tabular}[t]{l}$\alpha$\end{tabular}}}}%
    \put(0,0){\includegraphics[width=\unitlength,page=2]{hsur.pdf}}%
  \end{picture}%
\endgroup%
}
    \caption{Doing horizontal surgery on a horizontal surgery curve $\alpha$. Blue triple points are produced from the surgery in this example.}
    \label{fig:hsur}
\end{figure}

Let $n_i$ be the number of sectors attached along $N_i$. If $B'$ is obtained by doing $\frac{1}{-k}$ horizontal surgery along $\alpha$, then the number of triple points of $B'$ is $kn_1n_2$ more than that of $B$. If $M$ is oriented, then the added triple points are all of the same color. This color is red if the orientation of $N_1$ we chose above is coherent with the orientation of $M$, and blue otherwise.

\begin{rmk} \label{rmk:hsur}
If $\alpha$ is a horizontal surgery curve, the 1-cells in the branch locus which it passes through are dual to faces which join together as an annulus smoothly carried by the 2-skeleton of the dual veering triangulation. In fact, this is the reason why we call this surgery `horizontal'. When we do horizontal surgery along $\alpha$, we cut open the triangulation along this annulus, and insert some tetrahedra in-between.

In \cite{SS23}, Schleimer and Segerman show that any veering triangulation can be canonically decomposed as a union of \textit{veering solid tori}. The core of each veering solid torus is an example of a horizontal surgery curve. In this case, surgery along such a core is equivalent to inserting more veering solid tori along the core.
\end{rmk}

\subsection{Vertical surgery} \label{subsec:vsur}

\begin{defn} \label{defn:vsurcurve}
Let $B$ be an almost veering branched surface in a 3-manifold $M$ and let $\alpha \subset B$ be an oriented smoothly embedded curve which avoids the triple points of $B$. Let $N$ be a tubular neighborhood of $\alpha$ in $M$ and let $A$ be a smooth neighborhood of $\alpha$ in $B$.
We say that $\alpha$ is a \textit{vertical surgery curve} if:
\begin{enumerate}
    \item $\alpha$ is orientation and coorientation preserving; hence $A$ is an annulus and there are two regular half-neighborhoods of $\alpha$ in $B$, which we label as $N_1$ and $N_2$.
    \item The arcs in the branch locus of each $N_i \cap B$ are oriented from the same boundary component of $A$ to the other, but the two orientations for the two regular half-neighborhoods are opposite to each other.
    \item The arcs in the branch locus of $N \cap B$ are all cooriented coherently as $\alpha$.
\end{enumerate}

See \Cref{fig:vsur} left for an illustration of a vertical surgery curve.
\end{defn}

We explain how to do surgery along a vertical surgery curve $\alpha$. Again, we first set up the orientations: Note that $(\text{arcs in }\brloc(N_1 \cap B), \alpha)$ determines an orientation of $A$. Orient the meridian $\mu$ of $N_1$ such that the basis $(\mu, \alpha)$ on $\partial N_1$ agrees with this orientation on $A$.

For each $k \geq 0$, cut $N_1$ out of $M$ and glue it back with a map that is identity on $\partial N_1 \cut A$ and sends the meridian to a curve of isotopy class $\mu+ k \alpha$, such that the arcs in the branch locus of $N_1 \cap B$ intersect those in $N_2 \cap B$ minimally. See \Cref{fig:vsur}. This gives us a branched surface $B'$ in another 3-manifold $M'$, topologically obtained by doing $\frac{1}{k}$ surgery along $\alpha$ in $M$ (with respect to the basis we chose above). We call this operation a \textit{$\frac{1}{k}$ vertical surgery along $\alpha$}.

\begin{rmk} \label{rmk:vsurrmk}
As in \Cref{rmk:hsurrmk}, one can perform the surgery as cutting out $N_2$ and gluing it back appropriately. 
\end{rmk}

As for horizontal surgery, notice that $\alpha \subset B'$ continues to be a vertical surgery curve (for generic gluing maps), and if we perform $\frac{1}{l}$ vertical surgery on it, the total effect is equivalent to doing $\frac{1}{k+l}$ vertical surgery along the original $\alpha \subset B$. We abbreviate doing $\frac{1}{1}$ vertical surgery along $\alpha$ as just doing \textit{vertical surgery along $\alpha$}.

The same argument as in \Cref{prop:hsur} can be used to show the following proposition.

\begin{prop} \label{prop:vsur}
Let $\alpha$ be a vertical surgery curve on an almost veering branched surface $B$. Let $B'$ be the branched surface obtained by doing $\frac{1}{k}$ vertical surgery along $\alpha$. Then $B'$ is almost veering.

Furthermore, if $B$ is veering then $B'$ is veering as well.
\end{prop}

\begin{figure}
    \centering
    \fontsize{14pt}{14pt}\selectfont
    \resizebox{!}{5cm}{
\begingroup%
  \makeatletter%
  \providecommand\color[2][]{%
    \errmessage{(Inkscape) Color is used for the text in Inkscape, but the package 'color.sty' is not loaded}%
    \renewcommand\color[2][]{}%
  }%
  \providecommand\transparent[1]{%
    \errmessage{(Inkscape) Transparency is used (non-zero) for the text in Inkscape, but the package 'transparent.sty' is not loaded}%
    \renewcommand\transparent[1]{}%
  }%
  \providecommand\rotatebox[2]{#2}%
  \newcommand*\fsize{\dimexpr\f@size pt\relax}%
  \newcommand*\lineheight[1]{\fontsize{\fsize}{#1\fsize}\selectfont}%
  \ifx\svgwidth\undefined%
    \setlength{\unitlength}{426.39610559bp}%
    \ifx\svgscale\undefined%
      \relax%
    \else%
      \setlength{\unitlength}{\unitlength * \real{\svgscale}}%
    \fi%
  \else%
    \setlength{\unitlength}{\svgwidth}%
  \fi%
  \global\let\svgwidth\undefined%
  \global\let\svgscale\undefined%
  \makeatother%
  \begin{picture}(1,0.58371548)%
    \lineheight{1}%
    \setlength\tabcolsep{0pt}%
    \put(0,0){\includegraphics[width=\unitlength,page=1]{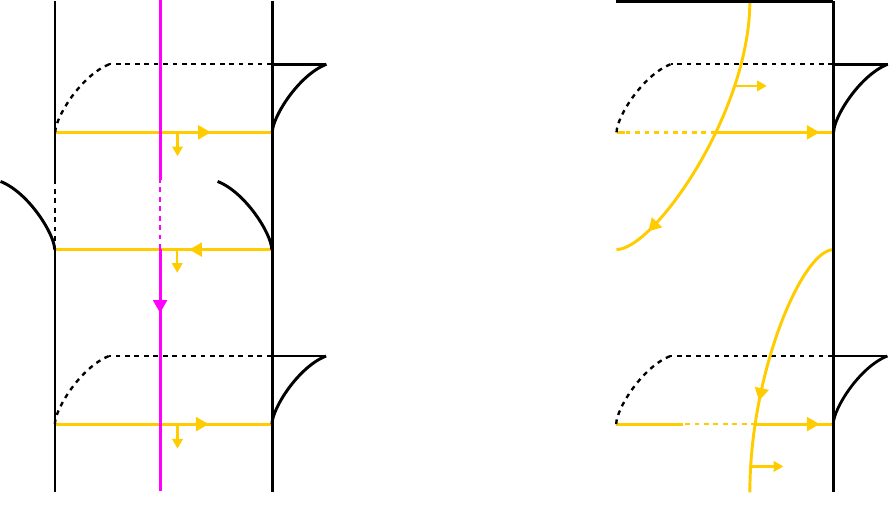}}%
    \put(0.19380052,0.23532311){\color[rgb]{1,0,1}\makebox(0,0)[lt]{\lineheight{1.25}\smash{\begin{tabular}[t]{l}$\alpha$\end{tabular}}}}%
    \put(0,0){\includegraphics[width=\unitlength,page=2]{vsur.pdf}}%
  \end{picture}%
\endgroup%
}
    \caption{Doing vertical surgery on a vertical surgery curve $\alpha$. Blue triple points are produced from the surgery in this example.}
    \label{fig:vsur}
\end{figure}

Let $n_i$ be the number of sectors attached along $N_i$. If $B'$ is obtained by doing $\frac{1}{k}$ vertical surgery along $\alpha$, then the number of triple points of $B'$ is $kn_1n_2$ more than that of $B$. If $M$ is oriented, then the added triple points are all of the same color. This color is red if the orientation of $N_1$ we chose above is coherent with the orientation of $M$, and blue otherwise.

\begin{rmk} \label{rmk:vsur}
If $\alpha$ is a vertical surgery curve, the 1-cells in the branch locus which $\alpha$ passes through determines an annulus as in \Cref{rmk:hsur}, but here the annulus is not smoothly carried by the 2-skeleton of the dual veering triangulation. Regardless, when one performs vertical surgery along $\alpha$, this annulus is cut open and some tetrahedra are inserted within. 

One example of vertical surgery curves are the infinitesimal cycles of the flow graph, as described in \cite{AT22} (and perturbed to avoid triple points). In this case, the collection of the added tetrahedra forms a wall in the terminology of \cite{AT22}.

We also explain why we call this surgery `vertical'. Suppose $B$ is veering, then by pushing $\alpha$ to the top of each 1-cell in the branch locus it intersects, one can see that $\alpha$ is carried by the 1-skeleton of $B$. This 1-skeleton is called the dual graph of the dual veering triangulation $\Delta$ in work of Landry-Minsky-Taylor, and by combining \cite[Proposition 5.7]{LMT20} and \cite[Corollary 5.16]{AT22}, we know that $\alpha$ is homotopic to a periodic orbit of the pseudo-Anosov flow associated to $\Delta$ under \Cref{thm:vtpAcorr}. We conjecture that the associated pseudo-Anosov flow after vertical surgery along $\alpha$ is related to the original flow via Goodman-Fried surgery on this periodic orbit. See \cite{Sha21} for more information on Goodman-Fried surgery.
\end{rmk}

\subsection{Variants of horizontal surgery: halved and concurrent} \label{subsec:hsurvar}

The idea behind horizontal and vertical surgeries can be easily modified to produce variants. Here we introduce a halved variant of horizontal surgery, which has the property of being equivariant under an involution. This property will allow us to take veering branched surfaces constructed on unit tangent bundles of genus zero orbifolds and quotient them down to Montesinos link complements under a branched double cover. It turns out that this halved horizontal surgery can be performed along multiple interacting sites concurrently. We will utilize this to construct veering branched surfaces on unit tangent bundles when we have to modify the order of more than one cone point of the base orbifold.

\begin{defn}
Let $B \subset M$ be an almost veering branched surface. Let $A \subset M$ be an annulus such that:
\begin{enumerate}
    \item $A$ is transverse to $B$ and to $\brloc(B)$, and $\partial A \subset B$
    \item All intersections of $A$ with $\brloc(B)$ lie in $\partial A$ and induce the same coorientation on $A$
    \item The boundary components of $A$ are horizontal surgery curves on $B$.
    \item The train track $A \cap B$ on $A$ consists of the two boundary components of $A$ and some branches in the interior going from one boundary component to the other, combed in different directions on the two boundary components.
\end{enumerate}

Then we say that the boundary components of $A$ form a \textit{pair of parallel horizontal surgery curves}, and that $A$ is a \textit{connecting annulus} between them.
\end{defn}

We explain a type of horizontal surgery one can perform along such a connecting annulus $A$. Let $\alpha$ and $\beta$ be the two boundary components of $A$. Take a neighborhood of $A$, $N \cong A \times [0,1]$, where the coorientation on $A$ in (2) is from $A \times \{1 \}$ to $A \times \{0\}$, and such that $\partial A \times [0,1]$ are smooth neighborhoods of $\alpha$ and $\beta$ in $B$. Orient $\alpha$ so that the arcs in $\brloc(N \cap B)$ which meet $\alpha$ are all cooriented coherently as $\alpha$, and label the branches of $A \cap B$ in the interior of $A$ as $b_1,...,b_n$ in the order of transversal by $\alpha$ (indices taken mod $n$ here). As in \Cref{subsec:hsur}, $(\text{arcs in }\brloc(N \cap B)\text{ which meet }\alpha, \alpha)$ determines an orientation on the smooth neighborhood of $\alpha$. Orient the meridian $\mu$ of $N$ such that the basis $(\mu,\alpha)$ on $\partial N$ agrees with this orientation.

Now for each $0 \leq j \leq n$, one can cut $N$ out of $M$, and glue it back with a map that is identity on $A \times \{ 0 \}$, takes $b_i \times \{1 \}$ to $b_{i+j} \times \{1\}$ for each $i$ (indices taken mod $n$ here) on $A \times \{1 \}$, and sends the meridian to a curve of isotopy class $\mu-\alpha$, such that the arcs in the branch locus of $N \cap B$ intersect those in the other half regular neighborhoods of $\alpha$ and $\beta$ minimally. 

See \Cref{fig:hsurinv} (ignoring the gray lines for now) top for an illustration of a connecting annulus. In \Cref{fig:hsurinv} bottom left we illustrate the result of the surgery operation from the perspective of the neighborhood $N$. In \Cref{fig:hsurinv} bottom right we illustrate the result of the surgery operation from the perspective of the exterior $M \cut N$. The orange curves in the bottom left and right are to be identified, and is the meridian in the surgered manifold.

This gives us a branched surface $B'$ in another 3-manifold $M'$, topologically obtained by doing $\frac{1}{-1}$ surgery along $\alpha$ on $M$ (with respect to the basis we chose above). 
We call this operation a \textit{$(\frac{-j}{n},\frac{j-n}{n})$ horizontal surgery along $A$}.

Intuitively what we have done is perform $\frac{j}{n}$ of a horizontal surgery along $\alpha$ and $\frac{n-j}{n}$ of a horizontal surgery along $\beta$, and together these `add up' to a complete horizontal surgery. Similar to our remark for basic horizontal surgery, observe that $A \subset M'$ continues to be a connecting annulus between a pair of parallel horizontal surgery curves (for generic gluing maps), and one can repeat this procedure, in particular producing many branched surfaces on the 3-manifolds obtained by $\frac{1}{-k}$ surgery along $\alpha$, for every $k \geq 0$.

\begin{prop} \label{prop:hsurinv}
Let $A$ be a connecting annulus between a pair of parallel horizontal surgery curves on an almost veering branched surface $B$. Let $B'$ be the branched surface obtained by doing $(\frac{-j}{n},\frac{j-n}{n})$ horizontal surgery along $A$. Then $B'$ is almost veering. 

Furthermore, if $B$ is veering then $B'$ is veering as well.
\end{prop}

\begin{proof}
The same arguments as in \Cref{prop:hsur} work here with one necessary modification for condition (ii): the complementary regions of $B$ may not be in one-to-one correspondence with those of $B'$. But observe that the complementary regions of $B$ which meet the interior of $A$ must be 2-cusped solid tori, and $A$ must meet these components in meridional discs, by (2) and (4). The complementary regions of $B'$ can be obtained by cutting and gluing these 2-cusped solid tori along meridional discs, the results of which will be 2-cusped solid tori again.
\end{proof}

\begin{figure}
    \centering
    \fontsize{14pt}{14pt}\selectfont
    \resizebox{!}{8cm}{
\begingroup%
  \makeatletter%
  \providecommand\color[2][]{%
    \errmessage{(Inkscape) Color is used for the text in Inkscape, but the package 'color.sty' is not loaded}%
    \renewcommand\color[2][]{}%
  }%
  \providecommand\transparent[1]{%
    \errmessage{(Inkscape) Transparency is used (non-zero) for the text in Inkscape, but the package 'transparent.sty' is not loaded}%
    \renewcommand\transparent[1]{}%
  }%
  \providecommand\rotatebox[2]{#2}%
  \newcommand*\fsize{\dimexpr\f@size pt\relax}%
  \newcommand*\lineheight[1]{\fontsize{\fsize}{#1\fsize}\selectfont}%
  \ifx\svgwidth\undefined%
    \setlength{\unitlength}{467.27786087bp}%
    \ifx\svgscale\undefined%
      \relax%
    \else%
      \setlength{\unitlength}{\unitlength * \real{\svgscale}}%
    \fi%
  \else%
    \setlength{\unitlength}{\svgwidth}%
  \fi%
  \global\let\svgwidth\undefined%
  \global\let\svgscale\undefined%
  \makeatother%
  \begin{picture}(1,0.71054182)%
    \lineheight{1}%
    \setlength\tabcolsep{0pt}%
    \put(0,0){\includegraphics[width=\unitlength,page=1]{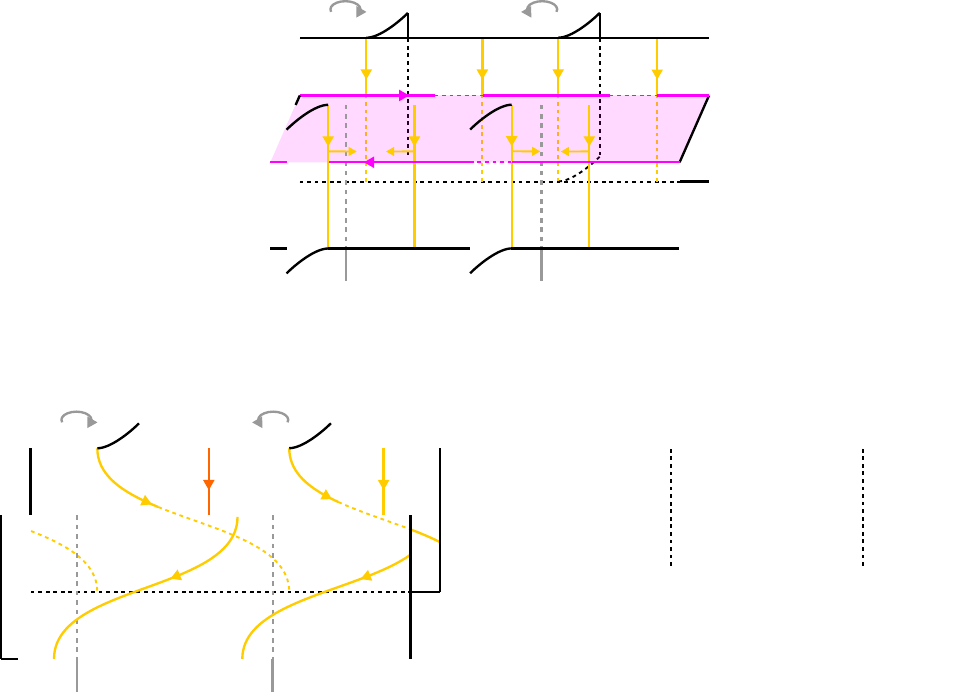}}%
    \put(0.37306532,0.55652717){\color[rgb]{1,0,1}\makebox(0,0)[lt]{\lineheight{1.25}\smash{\begin{tabular}[t]{l}$\beta$\end{tabular}}}}%
    \put(0,0){\includegraphics[width=\unitlength,page=2]{hsurinv2.pdf}}%
    \put(0.40652987,0.6246024){\color[rgb]{1,0,1}\makebox(0,0)[lt]{\lineheight{1.25}\smash{\begin{tabular}[t]{l}$\alpha$\end{tabular}}}}%
    \put(0,0){\includegraphics[width=\unitlength,page=3]{hsurinv2.pdf}}%
  \end{picture}%
\endgroup%
}
    \caption{Doing halved horizontal surgery along an equivariant connecting annulus (top, in pink). Bottom left: Result from the perspective of the neighborhood $N$. Bottom right: Result from the perspective of the exterior $M \cut N$. The orange curves are to be identified.}
    \label{fig:hsurinv}
\end{figure}

The ability to do fractions of horizontal surgeries is of interest in the setting where we have an involution on the 3-manifold. More precisely:

\begin{lemma} \label{lemma:inv}
Let $\iota: M \to M$ be an involution with fixed point set being a link in $M$, such that $M \to M/\langle \iota \rangle$ is a branched double cover. If $B$ is an (almost) veering branched surface on $M$ which is preserved by $\iota$, and for which the fixed point set of $\iota$ does not intersect $B$, then $B/\langle \iota \rangle$ is an (almost) veering branched surface in $M/\langle \iota \rangle$.

Furthermore, if $A$ is a connecting annulus between a pair of parallel horizontal surgery curves, and if $\iota$ preserves $A$ and its fixed point set intersects $A$, then $A/\langle \iota \rangle$ is a disc whose boundary curve is a horizontal surgery curve on $B/\langle \iota \rangle$.
\end{lemma}

\begin{proof}
Note that our hypothesis includes the assumption that $\iota$ preserves the orientations on the components of $\brloc(B)$. Thus (iii) of \Cref{defn:vbs} (or \Cref{defn:avbs}) is clear for $B/\langle \iota \rangle$ with the induced orientation on its branch locus. 
(i) is also clear from the observation that sectors of $B$ double cover sectors of $B/\langle \iota \rangle$. 

For (ii), the complementary regions of $B$ double cover or double branched cover those of $B/\langle \iota \rangle$, where the branch locus of the cover does not intersect the boundary of the complementary regions. But by standard 3-manifold topology and the Smith conjecture (proved in this setting by Waldhausen (\cite{Wal69})), a cusped solid torus or cusped torus shell can only double cover or double branched cover another cusped solid torus or cusped torus shell respectively, if the branch locus of the cover lies away from the boundary, so $B/\langle \iota \rangle$ satisfies (ii) as well.

For the second part of the lemma, $A$ double branched covers $A/\langle \iota \rangle$ hence $A/\langle \iota \rangle$ is a disc. Either boundary component of $A$ maps to the boundary curve of $A/\langle \iota \rangle$ hence the latter is a horizontal surgery curve.
\end{proof}

In the setting of \Cref{lemma:inv}, we call $A$ an \textit{equivariant connecting annulus}. 
Let $\alpha$ and $\beta$ be the boundary components of $A$, and let $N$ be a neighborhood of $A$ as described above which is in addition preserved by $\iota$. Let $n$ be the number of arcs in the branch locus of $N \cap B$ which meet $\alpha$; this is the same as the number of arcs which meet $\beta$, because of the symmetry from $\iota$. Let $m$ be the number of arcs in the branch locus of the intersection between $B$ and the half regular neighborhood of $\alpha$ lying on the opposite side as $N$; again this is the same as that for $\beta$ by symmetry. Moreover, the fact that no component of the fixed point set of $\iota$ lies in $B$ implies that $n$ is even. 

We perform $(-2,-2)$ horizontal surgery along $A$ as described above (i.e. we take $j=\frac{n}{2}$) to get $B' \subset M'$. Note that the gluing map can be chosen to be $\iota$-equivariant, hence we can define an involution $\iota'$ on $M'$ by gluing together the involutions on $M \cut N$ and $N$. The objects $M'$, $\iota'$, and $B'$ then satisfy the assumptions of \Cref{lemma:inv}. See \Cref{fig:hsurinv}, now with the gray lines. In particular, we have an (almost) veering branched surface $B'/\langle \iota' \rangle$ on $M'/\langle \iota' \rangle$. 

For every $k \geq 0$, we can repeat this procedure $k$ times to get an (almost) veering branched surface $B'$ on 3-manifold $M'$, topologically obtained by doing $\frac{1}{-k}$ surgery along $\alpha$ in $M$, and which has an involution $\iota'$. Moreover, $B' \subset M'$ descends to an (almost) veering branched surface $B'/\langle \iota' \rangle \subset M'/\langle \iota' \rangle$. We will refer to the total operation going from $B$ to $B'$ as \textit{halved $\frac{1}{-k}$ horizontal surgery} along the equivariant connecting annulus $A$.

\begin{rmk} \label{rmk:hsurinvdownstairs}
The operation of going from $M/\langle \iota \rangle$ to $M'/\langle \iota' \rangle$ can be described directly as follows. Take the disc $A/\langle \iota \rangle \subset M/\langle \iota \rangle$, for which the branch locus of $\iota$ on $M/\langle \iota \rangle$ intersects in two points, cut it open, then reglue it with $k$ half twists. The image of the branch locus of $\iota$ after cutting and gluing will be the branch locus of $\iota'$. The direction of the half twists can be determined by orienting the boundary curve of $A/\langle \iota \rangle$ such that the arcs in the branch locus of $(N \cap B)/\langle \iota \rangle$ are cooriented coherently with it, and twisting the bottom half $(A \times \{0 \})/\langle \iota \rangle$ in the direction of its oriented boundary. 
\end{rmk}

The number of triple points of $B'$ is $kmn$ more than that of $B$, so the number of triple points of $B'/\langle \iota' \rangle$ is $\frac{kmn}{2}$ more than that of $B/\langle \iota \rangle$. If $M$ is oriented, then $M/\langle \iota \rangle$ is also oriented by our assumptions on $\iota$. In this case the $\frac{knm}{2}$ triple points added are all of the same color: red if the orientation we chose on $N$ matches that of $M$, blue otherwise.

Next, we explain a setting where we can do several of such halved horizontal surgeries concurrently. 

\begin{defn} \label{defn:hsursim}
Suppose $M, \iota,$ and $B$ satisfy the assumptions of \Cref{lemma:inv}. Let $A_1,...,A_s$ be equivariant connecting annuli, and let $N_i \cong A_i \times [0,1]$ be neighborhoods of $A_i$ as above. Let the boundary components of $A_i$ be $\alpha_i$ and $\beta_i$. Orient $\alpha_i$ and $\beta_i$ so that the arcs in the branch locus of $N_i \cap B$ are cooriented coherently. Meanwhile the orientations on those arcs also induce coorientations of $\alpha_i, \beta_i$ on $B$ (which are well-defined by (2) of \Cref{defn:hsurcurve}). 

Suppose that: 
\begin{enumerate}
    \item $\alpha_i$ and $\alpha_j$ are disjoint for all $i,j$, and $\beta_i$ and $\beta_j$ are disjoint for all $i,j$
    \item $\alpha_i$ and $\beta_j$ intersect transversely in $B$ for all $i,j$.
    \item At each intersection point between $\alpha_i$ and $\beta_j$, $A_i$ and $A_j$ locally lie on different sides of $B$.
    \item At each intersection point between $\alpha_i$ and $\beta_j$, the orientation of $\alpha_i$ is incoherent with the coorientation of $\beta_j$, and the orientation of $\beta_j$ is incoherent with the coorientation of $\alpha_i$
\end{enumerate}

Then we call $\{A_i\}$ a \textit{system of equivariant connecting annuli}.
\end{defn}

Orient each $N_i$ as above. (3) and (4) of \Cref{defn:hsursim} implies that these orientations match up where they overlap. For every $k_i \geq 0, i=1,...,s$, we can perform halved $\frac{1}{-k_i}$ horizontal surgeries along each $A_i$ simultaneously by cutting out each $N_i$ and gluing it back in with a map that is identity on $A_i \times \{0\}$, shifts the branches on $A_i \times \{1\}$ by $\frac{k_i}{2}$ cycles, and sends the meridian $\mu_i$ to a curve of isotopy class $\mu_i - k_i \alpha_i$, such that the arcs in the branch locus of $N_i \cap B$ intersect minimally on smooth neighborhoods of $\alpha_i, \beta_i$. See \Cref{fig:hsursim} for an illustration near an intersection point of $\alpha_i$ and $\beta_j$. This gives us a branched surface $B'$ in another 3-manifold $M'$, topologically obtained by doing $\frac{1}{-k_i}$ surgeries along $\alpha_i$ on $M$. We call this operation a \textit{concurrent halved $\frac{1}{-k_i}$ horizontal surgery on the system $\{A_i\}$}.

\begin{figure}
    \centering
    \resizebox{!}{10cm}{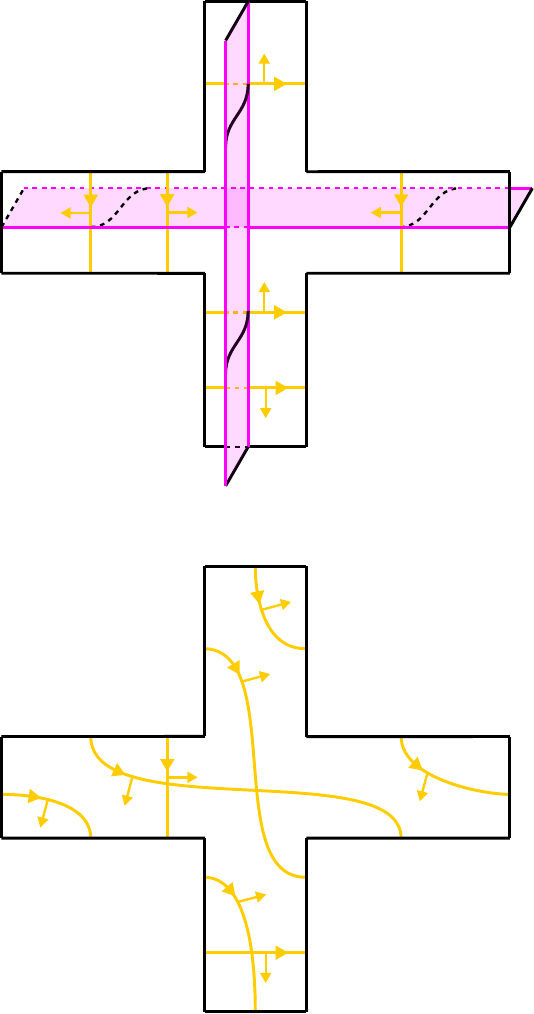}
    \caption{Doing concurrent halved horizontal surgery on a system of equivariant connecting annuli (in pink). In the bottom figure, to avoid clutter, we only draw the restriction of the branch locus to part of a surface carried by the branched surface.}
    \label{fig:hsursim}
\end{figure}

\begin{rmk}
It is possible to perform the surgeries along each $A_i$ in sequence, as opposed to concurrently, in the following sense: After performing surgery on, say, $A_1$ by cutting and regluing $N_1$, one can check that the collection of the remaining $A_i$ still forms a system of equivariant connecting annuli. However, the surgery along $A_1$ would introduce intersections of the branch locus with $\partial A_i$, and so for example if one wishes to keep track of the number of triple points as in \Cref{prop:hsursumtriplepoints} below, one would have to keep track of these additional intersections, which we believe would be a less clean approach.
\end{rmk}

Using the same argument as \Cref{prop:hsurinv}, one can show the following proposition.

\begin{prop} \label{prop:hsursim}
Let $\{A_i\}$ be a system of equivariant connecting annuli on an almost veering branched surface $B$. Let $B'$ be the branched surface obtained by doing concurrent halved $\frac{1}{-k_i}$ horizontal surgery on $\{A_i\}$. Then $B'$ is almost veering.

Furthermore, if $B$ is veering then $B'$ is veering as well.
\end{prop}

Also, we can again define an involution $\iota'$ on $M'$ by gluing together the involutions on $N_i$ and $M \cut N_i$, under which $M'$, $\iota'$, and $B'$ satisfy the assumptions of \Cref{lemma:inv}. To directly obtain $M'/\langle \iota' \rangle$ from $M/\langle \iota \rangle$, one performs $k_i$ half twists along the discs $A_i/\langle \iota \rangle$.

In terms of the 3-manifolds, concurrent halved horizontal surgery is just performing surgery along the disjoint curves $\alpha_i$, but on the level of the branched surfaces, the various surgeries interact with each other.

In particular, the fact that the surgery curves intersect each other adds extra terms to the number of triple points produced by the operation: Let $n_i$ be the number of arcs in the branch locus of $N_i \cap B$ which meet $\alpha_i$ (equivalently, that of $\beta_i$), let $m_i$ be the number of arcs in the branch locus of the intersection between $B$ and the half regular neighborhood of $\alpha_i$ on the opposite side as $N_i$ (equivalently, that of $\beta_i$), and let $q_{ij}$ be the number of intersection points between $\alpha_i$ and $\beta_j$ (equivalently, between $\alpha_j$ and $\beta_i$). 

Performing the surgery along $A_i$ on its own produces $k_i m_i n_i$ triple points. If we perform all the surgeries concurrently, then near each intersection point between $\alpha_i$ and $\beta_j$, $\frac{1}{4} n_i n_j k_i k_j$ additional triple points are produced. See \Cref{fig:hsursim} for an example, where $n_i=n_j=2, k_i=k_j=1$ and one additional triple point is produced. 

Adding these terms together gives the following formula.

\begin{prop} \label{prop:hsursumtriplepoints}
Let $C=[\frac{1}{4} q_{ij}n_in_j]$, $d=[n_im_i]$, and $k=[k_i]$. The number of triple points of $B'$ is $k^T (Ck+d)$ more than that of $B$, and the number of triple points of $B'/\langle \iota' \rangle$ is $\frac{1}{2} k^T (Ck+d)$ more than that of $B/\langle \iota \rangle$.
\end{prop}

As before, if $M$ is oriented, then the triple points added are all of the same color, red if the chosen orientations on $N_i$ match that of $M$, and blue otherwise. 

\section{Geodesic flows I: Genus zero orbifolds, Montesinos knots and links} \label{sec:genus0}

In this section we explain \Cref{constr:genus0}. We remind the reader of the setup. Let $N$ be the unit tangent bundle of the orbifold $S^2(p_1,...,p_n)$. Suppose that $e:=\sum \frac{1}{p_i}-n+2<0$, so that the geodesic flow on $N$ is Anosov. Let $c$ be the simple closed curve passing through the cone points of $S^2(p_1,...,p_n)$ in the order of their indices (taken mod $n$), and let $\overset{\leftrightarrow}{c}$ be the full lift of $c$. Let $M= N \backslash \overset{\leftrightarrow}{c}$. There is an involution $\iota$ of $N$ such that $N \to N/\langle \iota \rangle \cong S^3$ is a branched double cover, with branch locus of the cover equal to the Montesinos link $M(\frac{1}{p_1}+1, \frac{1}{p_2}-1, ...,\frac{1}{p_n}-1) \subset S^3$. 

Our goal in this section will be to construct a veering branched surface on $N$, for which the cores of the complementary regions are given by $\overset{\leftrightarrow}{c}$. This veering branched surface will be invariant under $\iota$ in the way described by \Cref{lemma:inv}, hence will descend to a veering branched surface on the Montesinos link complement $S^3 \backslash M(\frac{1}{p_1}+1, \frac{1}{p_2}-1, ...,\frac{1}{p_n}-1)$ where all complementary regions are cusped torus shells, thus giving us a dual veering triangulation. We will state this last result as a theorem.

\begin{thm} \label{thm:genus0}
For every $n, p_1,...,p_n$ such that $e:=\sum \frac{1}{p_i}-n+2<0$, there is a veering triangulation on the Montesinos link complement $S^3 \backslash M(\frac{1}{p_1}+1, \frac{1}{p_2}-1, ...,\frac{1}{p_n}-1)$.
\end{thm}

Now, this theorem by itself is not new. It is already known that these Montesinos link complements admit veering triangulations just from the fact that they are fibered with fully-punctured pseudo-Anosov monodromy (by \cite[Theorem E]{CD20}). Instead, the more significant point here is that our construction of the veering branched surfaces is entirely explicit, whereas it is not clear what the monodromies of these fiberings are, let alone periodic folding sequences of train tracks of the monodromies, which is what one needs in order to construct the veering triangulation as in \cite{Ago11}. We also point out that by \Cref{thm:sfspageod}, the Anosov flows on $T^1 S^2(p_1,...,p_n)$ that correspond to the double covers of the veering triangulations in \Cref{thm:genus0} must be the geodesic flow.

In \Cref{sec:table}, we will compile the IsoSig codes of those veering triangulations in \Cref{thm:genus0} which are present in the veering triangulation census \cite{VeeringCensus}.

We will divide the proof of \Cref{thm:genus0} into four cases. Notice that $(S^2(p_1,...,p_n),c) \cong (S^2(p_{\sigma(1)},...,p_{\sigma(n)}),c)$ for cyclic permutations and reversals $\sigma \in S_n$. So when $n=3$, we can always assume that $p_1 \leq p_2 \leq p_3$. With this arranged, all the possibilities of $p_1,p_2,p_3$ for which $e<0$ fall under three cases: (1) $p_1=2, p_2=3, p_3>6$, (2) $p_1=2, (p_2,p_3)>(4,4)$, and (3) $(p_1,p_2,p_3)>(3,3,3)$, where we use the lexicographic ordering on tuples of integers.

For $n=4$, we do not attempt to rearrange the $p_i$, and simply note that $e<0$ is equivalent to $(p_1,p_2,p_3,p_4) > (2,2,2,2)$. For $n \geq 5$, there are no restrictions on $p_i$. We group all the possibilities for $n \geq 4$ together as case (4).

Each case proceeds by first constructing an almost veering branched surface on some unit tangent bundle, then locating an appropriate system of equivariant connecting annulus, and performing concurrent halved horizontal surgery on the system to get veering branched surfaces on the family of unit tangent bundles that we want, which then descend, under the involution, to veering branched surfaces on the family of Montesinos link complements that we are considering. The construction for cases (1)-(3) are very similar to each other. Meanwhile case (4) will be done via a different approach for the purposes of \Cref{sec:highergenus}.

\subsection{Case 1: $n=3, p_1=2, p_2=3, p_3 > 6$} \label{subsec:g2}

We will first construct a branched surface on $T^1 \mathbb{R}^2$. Here we put the usual Euclidean coordinates $x,y$ on $\mathbb{R}^2$ and set $T^1 \mathbb{R}^2 \cong \mathbb{R}^2_{x,y} \times (\mathbb{R}/2 \pi \mathbb{Z})_\theta$ where $\tan \theta$ is the slope. 

Consider the $G_2$ grid on $\mathbb{R}^2$. For concreteness, say, this is given by taking the union of the lines
\begin{align*}
&\{ y=3n \}_{n \in \mathbb{Z}} \cup \{ y=\frac{1}{\sqrt{3}}x+2n \}_{n \in \mathbb{Z}} \cup \{ y=\sqrt{3}x+6n \}_{n \in \mathbb{Z}}\\
\cup &\{x=\sqrt{3} n \}_{n \in \mathbb{Z}} \cup \{ y=-\frac{1}{\sqrt{3}}x+2n \}_{n \in \mathbb{Z}} \cup \{ y=-\sqrt{3}x+6n \}_{n \in \mathbb{Z}}
\end{align*}

See the black lines in \Cref{fig:g2gridsur}.

\begin{figure}
    \centering
    \resizebox{!}{7.5cm}{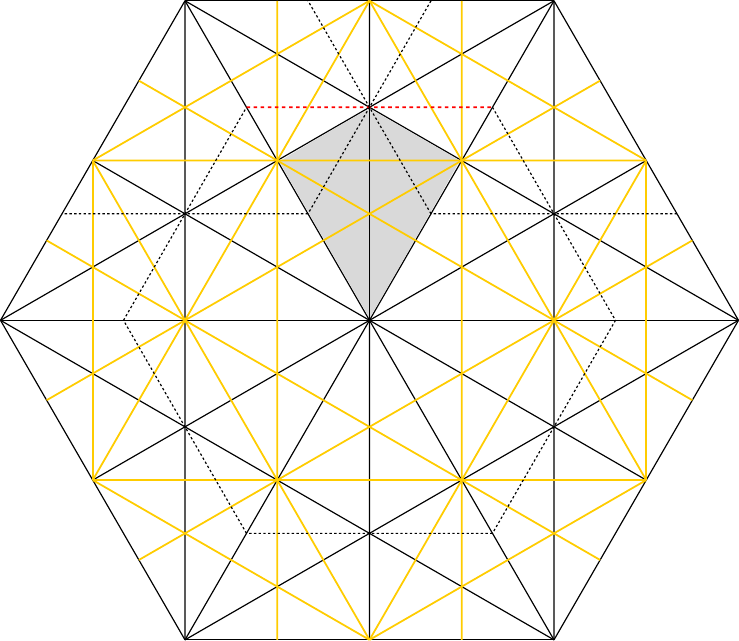}
    \caption{Using the $G_2$ grid to construct an almost veering branched surface on $T^1 S^2(2,3,6)$.}
    \label{fig:g2gridsur}
\end{figure}

Then consider the midway lines of this grid, i.e. lines that lie in the middle of each pair of adjacent parallel lines. Concretely, these are
\begin{align*}
&\{ y=3n+\frac{3}{2} \}_{n \in \mathbb{Z}} \cup \{ y=\frac{1}{\sqrt{3}}x+2n+1 \}_{n \in \mathbb{Z}} \cup \{ y=\sqrt{3}x+6n+3 \}_{n \in \mathbb{Z}}\\
\cup &\{x=\sqrt{3} n+\frac{\sqrt{3}}{2} \}_{n \in \mathbb{Z}} \cup \{ y=-\frac{1}{\sqrt{3}}x+2n+1 \}_{n \in \mathbb{Z}} \cup \{ y=-\sqrt{3}x+6n+3 \}_{n \in \mathbb{Z}}
\end{align*}

See the yellow lines in \Cref{fig:g2gridsur}.

Now construct a branched surface in $T^1 \mathbb{R}^2$ by taking the horizontal planes $\{\theta=\frac{(2i-1)\pi}{12} \}$, then attaching infinite strips of the form:
\begin{align*}
&\{ y=3n+\frac{3}{2} , \theta \in [-\frac{\pi}{12}, \frac{\pi}{12}] \cup [\frac{11\pi}{12}, \frac{13\pi}{12}] \}_{n \in \mathbb{Z}}\\
\cup &\{ y=\frac{1}{\sqrt{3}}x+2n+1 , \theta \in [\frac{\pi}{12}, \frac{3\pi}{12}] \cup [\frac{13\pi}{12}, \frac{15\pi}{12}] \}_{n \in \mathbb{Z}}\\
\cup &\{ y=\sqrt{3}x+6n+3, \theta \in [\frac{3\pi}{12}, \frac{5\pi}{12}] \cup [\frac{15\pi}{12}, \frac{17\pi}{12}] \}_{n \in \mathbb{Z}}\\ 
\cup &\{x=\sqrt{3} n+\frac{\sqrt{3}}{2}, \theta \in [\frac{5\pi}{12}, \frac{7\pi}{12}] \cup [\frac{17\pi}{12}, \frac{19\pi}{12}] \}_{n \in \mathbb{Z}}\\
\cup &\{ y=-\sqrt{3}x+6n+3, \theta \in [\frac{7\pi}{12}, \frac{9\pi}{12}] \cup [\frac{19\pi}{12}, \frac{21\pi}{12}] \}_{n \in \mathbb{Z}}\\
\cup &\{ y=-\frac{1}{\sqrt{3}}x+2n+1, \theta \in [\frac{9\pi}{12}, \frac{11\pi}{12}] \cup [\frac{21\pi}{12}, \frac{23\pi}{12}] \}_{n \in \mathbb{Z}}\\ 
\end{align*}
and combing the lines of attachment such that the maw coorientation of $\{ y=(\tan \theta_0) x+c_0, \theta=\theta_0+\frac{\pi}{12}\}$ is given by $-\sin \theta_0 \frac{\partial}{\partial x}+\cos \theta_0 \frac{\partial}{\partial y}$ and that of $\{ y=(\tan \theta_0) x+c_0, \theta=\theta_0-\frac{\pi}{12}\}$ is given by $\sin \theta_0 \frac{\partial}{\partial x}-\cos \theta_0 \frac{\partial}{\partial y}$. We also orient $\{ y=(\tan \theta_0) x+c_0, \theta=\theta_0 \pm \frac{\pi}{12}\}$ by $\cos \theta_0 \frac{\partial}{\partial x}+\sin \theta_0 \frac{\partial}{\partial y}$. 

Up to a small perturbation, we can arrange for these infinite strips to be transverse to the fiber direction $\frac{\partial}{\partial \theta}$. We call the resulting branched surface $\widetilde{B_0}$. Intuitively, we are attaching strips that lift each oriented midway line and which lie between the horizontal layers. This makes $\widetilde{B_0}$ into a `wasp nest' with holes where the full lifts of the grid lines live. See \Cref{fig:genus0}.

\begin{figure}
    \centering
    \resizebox{!}{5cm}{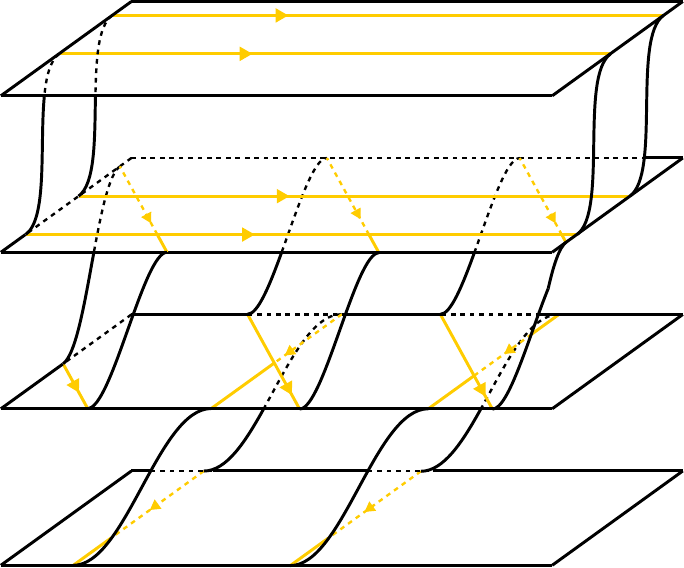}
    \caption{Rough picture of $\widetilde{B_0}$.}
    \label{fig:genus0}
\end{figure}

Now consider the subgroup $G^+$ in $\Isom^+(\mathbb{R}^2) \cong \mathbb{R}^2 \rtimes SO(2)$ generated by products of even numbers of reflections across the grid lines. $G^+$ quotients $\mathbb{R}^2$ down to the orbifold $S^2(2,3,6)$, hence its lifted action quotients $T^1 \mathbb{R}^2$ down to $T^1 S^2(2,3,6)$. In fact, it is easy to see that the quotient map $T^1 \mathbb{R}^2 \to T^1 S^2(2,3,6)$ is a covering. Meanwhile, our construction of $\widetilde{B_0}$ can be made to be equivariant under $G^+$; one only has to do the combings and perturbations equivariantly. Hence $\widetilde{B_0}$ descends to a branched surface $B_0$ on $T^1 S^2(2,3,6)$. The orientations we defined on the components of $\brloc(\widetilde{B_0})$ are also equivariant under $G^+$, hence they descend to orientations on the components of $\brloc(B_0)$.

\begin{claim} \label{lemma:genus0avbs}
$B_0$ is an almost veering branched surface. The cores of its complementary regions are given by $\overset{\leftrightarrow}{c}$.
\end{claim}

\begin{proof}
For the first statement, we check the conditions in \Cref{defn:avbs} one by one.

The sectors of $\widetilde{B_0}$ cover those of $B_0$. One can check that the former consists of quadrilaterals lying on the planes $\{\theta=\theta_0\}$ and the infinite strips. The quadrilaterals are topologically discs hence can only homeomorphically cover quadrilateral sectors in $B_0$. Similarly, the infinite strips can only cover annulus or M\"obius band sectors in $B_0$. Hence (i) of \Cref{defn:avbs} is satisfied by $B_0$. In fact, we do not get M\"obius band sectors in $B_0$ since $G^+$ preserves the fiber direction $\frac{\partial}{\partial \theta}$ hence maps the infinite strips to themselves in orientation preserving ways. 

Similarly, the complementary regions of $\widetilde{B_0}$ cover those of $B_0$. The former consists of 2-cusped infinite cylinders, i.e. $D^2 \times \mathbb{R}$ with two cusp lines running along the $\mathbb{R}$ direction. These can only cover 2-cusped solid tori. Hence (ii) of \Cref{defn:avbs} is satisfied by $B_0$. Similarly as above, since $G^+$ preserves the fiber direction, the 2-cusped solid tori must in fact be untwisted, i.e. there are two cusp circles on each solid torus.

Finally, the orientations on the components of $\brloc(\widetilde{B_0})$ satisfy (iii) of \Cref{defn:avbs} at each triple point, and so (iii) holds for $B_0$ as well.

For the second statement, notice that the grid lines in $\mathbb{R}^2$ descend to the curve $c$ in $S^2(2,3,6)$. Hence the cores of the complementary regions of $\widetilde{B_0}$, which are given by full lifts of the grid lines, descend to the full lift of $c$. From our argument showing (ii) above, these are the cores of the complementary regions of $B_0$.
\end{proof}

Now let $G$ be the subgroup in $\Isom(\mathbb{R}^2) \cong \mathbb{R}^2 \rtimes O(2)$ generated by all products of reflections across the grid lines. $G^+$ is a index $2$ subgroup of $G$, and $G/G^+=:\langle \iota \rangle$ acts on $S^2(2,3,6)$ by reflection across $c$. Hence $\iota$ acts on $T^1 S^2(2,3,6)$ by the lift of this reflection across $c$. In particular the fixed point set of $\iota$ is $\overset{\leftrightarrow}{c}$, which by \Cref{lemma:genus0avbs} above, is the core of the complementary regions of $B_0$. Thus by \Cref{prop:montelink}, \Cref{prop:hsurinv}, and \Cref{lemma:inv}, $B_0/\langle \iota \rangle$ is an almost veering branched surface on $S^3$, with the cores of its complementary regions given by the Montesinos link $M(\frac{3}{2}, -\frac{2}{3}, -\frac{5}{6})$.

Next we will locate a system of equivariant connecting annuli of $B_0$. To do so, we work in $T^1 \mathbb{R}^2$. For each vertex $(x_0,y_0)$ in the orbit of $(0,0)$ under $G$, consider the hexagon $H$ in $\mathbb{R}^2$ with vertices $(x_0 \pm \frac{4 \sqrt{3}}{3},y_0), (x_0 \pm \frac{2 \sqrt{3}}{3}, y_0 \pm 2)$, and consider the torus that is $T^1 \mathbb{R}^2|_H$. The hexagon $H$ for the point $(0,0)$ is drawn in dotted lines in the center of \Cref{fig:g2gridsur}.

Let $E$ be the edge of $H$ between $(x_0 - \frac{2 \sqrt{3}}{3}, y_0 + 2)$ and $(x_0 + \frac{2 \sqrt{3}}{3}, y_0 + 2)$, which is the red segment in \Cref{fig:g2gridsur}. We draw $T^1 \mathbb{R}^2|_E$ in \Cref{fig:g2section}, where the arrows on the left represent the vertical $\theta$ coordinate by the directions in \Cref{fig:g2gridsur}. Notice that $E$ is the fundamental domain of $H$ under the action of the stabilizer of $(x_0,y_0)$ in $G^+$, and that $H$ is the union of 6 translates of $E$. Hence the whole torus $T^1 \mathbb{R}^2|_H$ can be obtained by taking 6 copies of \Cref{fig:g2section}, and gluing them up cyclically with an upward shift of $\frac{\pi}{3}$ when moving to the copy on the right. Equivalently, when we quotient by $G^+$ to $T^1 S^2(2,3,6)$, the image of $T^1 \mathbb{R}^2|_H$ can be obtained by taking just \Cref{fig:g2section} and gluing the left and right sides with a shift of $\frac{\pi}{3}$. 

\begin{figure}
    \centering
    \resizebox{!}{8cm}{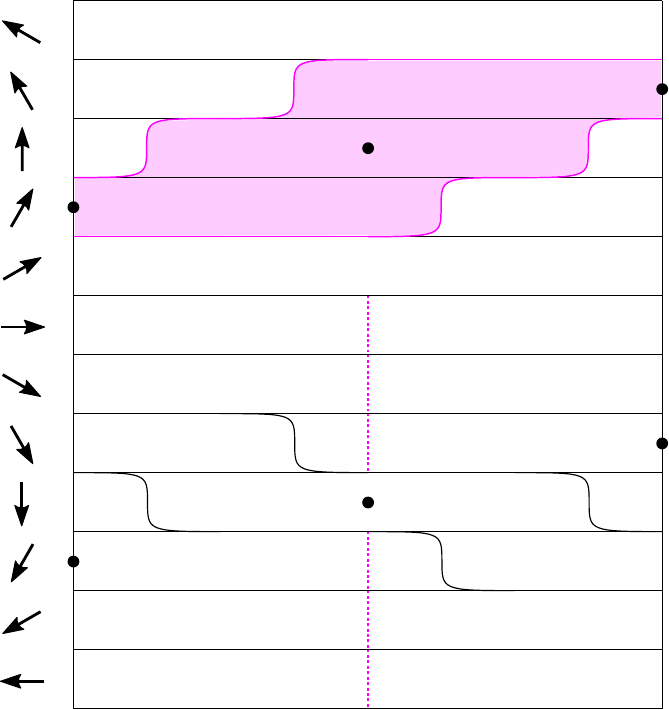}
    \caption{The intersection of $T^1 \mathbb{R}^2|_E$ with $\widetilde{B_0}$. The arrows on the left represent the vertical $\theta$ coordinate by the directions in \Cref{fig:g2gridsur}. The pink region defines a connecting annulus $\widetilde{A}$ of $\widetilde{B_0}$. The black dots denote where the full lifts of the grid lines intersect $T^1 \mathbb{R}^2|_E$. The pink vertical dotted lines denote where the other translates of $\widetilde{A}$ intersect $T^1 \mathbb{R}^2|_E$.}
    \label{fig:g2section}
\end{figure}

Consider the pink region in \Cref{fig:g2section}. By taking the union of these regions in the 6 copies of $T^1 \mathbb{R}^2|_E$ that form $T^1 \mathbb{R}^2|_H$, we can define an annulus $\widetilde{A}$ on $T^1 \mathbb{R}^2|_H$. It is straightforward to check that $\widetilde{A}$ is a connecting annulus between a pair of parallel horizontal surgery curves in $\widetilde{B_0}$, and that $\widetilde{A}$ is preserved by the stabilizer of $(x_0,y_0)$ in $G$. Hence upon quotienting by $G^+$, $\widetilde{A}$ descends to an equivariant connecting annulus $A$ of $B_0$.

$A$ is embedded in $T^1 S^2(2,3,6)$ since its preimages, which are the $\widetilde{A}$ constructed for the different orbits of $(0,0)$, are disjoint from one another. We demonstrate this disjointness by drawing pink dotted lines on \Cref{fig:g2section} where the other preimages come through $T^1 \mathbb{R}^2|_H$, and seeing that they lie away from the shaded region. Hence $A$ is a (trivial) system of equivariant connecting annuli, and one can read off from \Cref{fig:g2section} that, in the notation of \Cref{prop:hsursumtriplepoints}, $n=2, m=2, q=0$, hence $C=[0], d=[4]$.

Now for $k > 0$, we can apply concurrent halved $\frac{1}{-k}$ horizontal surgery on this system to get an almost veering branched surface $B_k$ on 3-manifold $M_k$ with involution $\iota_k$. We analyze what 3-manifold $M_k$ is by analyzing the operation at the level of $T^1 S^2(2,3,6)/\langle \iota \rangle$, using \Cref{rmk:hsurinvdownstairs}. In \Cref{fig:monteblocksur}, we draw half of our connecting annulus $A$ in the picture of \Cref{fig:monteblock}. Taking the quotient as in \Cref{fig:monteblock}, we see that $A/\langle \iota \rangle$ is a disc around the SE and SW strands of the rational tangle $\frac{1}{6}$, with the boundary oriented clockwise when viewed from above. Hence when we do $\frac{1}{-k}$ surgery, we add $k$ half twists to arrive at the Montesinos link $M(\frac{3}{2}, -\frac{2}{3}, \frac{1}{6+k}-1)$. More precisely, we mean that $B_k/\langle \iota_k \rangle$ sits inside $S^3$ with the cores of its complementary regions given by $M(\frac{3}{2}, -\frac{2}{3}, \frac{1}{6+k}-1)$. Taking the branched double cover, we see that $M_k=T^1 S^2(2,3,6+k)$ and the cores of the complementary regions of $B_k$ are given by the full lift of the curve $c$ on $S^2(2,3,6+k)$.

\begin{figure}
    \centering
    \resizebox{!}{4.5cm}{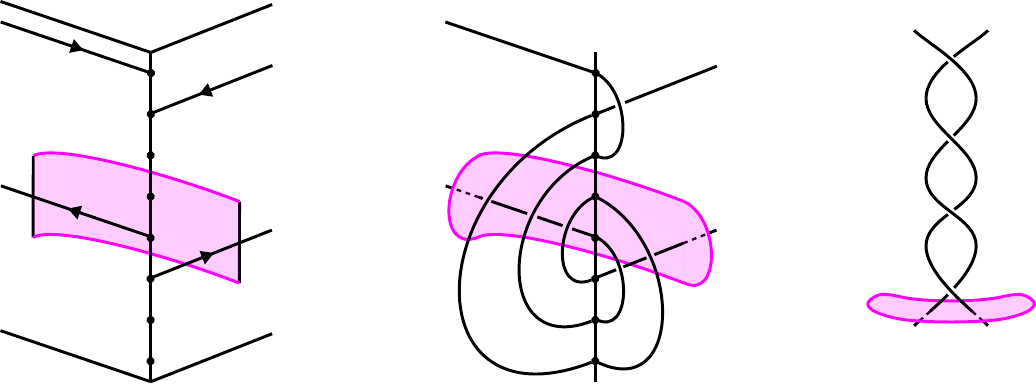}
    \caption{Performing halved $\frac{1}{-k}$ horizontal surgery on an annulus around a cone point adds $k$ half twists to the corresponding rational tangle.}
    \label{fig:monteblocksur}
\end{figure}

An alternative way to reach this conclusion would be to first argue that $M_k$ is the Seifert fibered space $S^2((2,3), (3,-2),(6+k,-5-k))$ by working out the effects of surgery along $\partial A$ as a curve on $T^1 \mathbb{R}^2|_H /G^+$, which is the boundary of a fibered neighborhood around the singular orbit above the cone point of order $6$. Then using \Cref{thm:vtpAcorr} and \Cref{thm:sfspageod}, one can work out the cores of the complementary regions of $B_k$ by using the fact that an orbit of the geodesic flow is uniquely determined by its image on the orbifold, and tracing out the image of the cores on $S^2(2,3,6+k)$ directly.

By \Cref{thm:fulllifthyp}, $T^1 S^2(2,3,6+k) \backslash \overset{\leftrightarrow}{c}$ is hyperbolic, hence $B_k$ is veering by \Cref{prop:avbshyp}, so $B_k/\langle \iota_k \rangle$ on $S^3 \backslash M(\frac{3}{2}, -\frac{2}{3}, \frac{1}{6+k}-1)$ is veering as well. Taking its dual ideal triangulation, we have proven \Cref{thm:genus0} in this case.

We can count the number of tetrahedra and the number of blue/red edges in these veering triangulations: Take the orientation on $T^1 \mathbb{R}^2$ to be that given by $(\frac{\partial}{\partial x},\frac{\partial}{\partial y}, \frac{\partial}{\partial \theta}) $. The triple points on $\widetilde{B_0}$ are of the form in \Cref{fig:vbslocal} right, so the triple points in $B_0$ must all be red, and one can count that there are exactly $2$ of them. By \Cref{prop:hsursumtriplepoints}, concurrent halved $\frac{1}{-k}$ horizontal surgery on our system produces $4k$ blue triple points, so $B_k$ has $4k$ blue triple points and $2$ red triple points, and $B_k/\langle \iota_k \rangle$ has $2k$ blue triple points and $1$ red triple point. We conclude that the dual veering triangulation on $S^3 \backslash M(\frac{3}{2}, -\frac{2}{3}, \frac{1}{6+k}-1)$ has $2k+1$ tetrahedra, $2k$ blue edges and $1$ red edge.

\subsection{Case 2: $n=3, p_1=2, (p_2,p_3)>(4,4)$} \label{subsec:diag}

The strategy here is the same as the last case, so we will be more brief. First we will construct a branched surface on $T^1 \mathbb{R}^2 \cong \mathbb{R}^2_{x,y} \times (\mathbb{R}/2 \pi \mathbb{Z})_\theta$. This time, consider the diagonal grid on $\mathbb{R}^2$. This is given by taking the union of the lines
\begin{align*}
\{ y=2n \}_{n \in \mathbb{Z}} \cup \{ y= x+4n \}_{n \in \mathbb{Z}} \cup \{x=2n \}_{n \in \mathbb{Z}} \cup \{ y= -x+4n \}_{n \in \mathbb{Z}} 
\end{align*}

See the black lines in \Cref{fig:diaggridsur}.

\begin{figure}
    \centering
    \resizebox{!}{7.5cm}{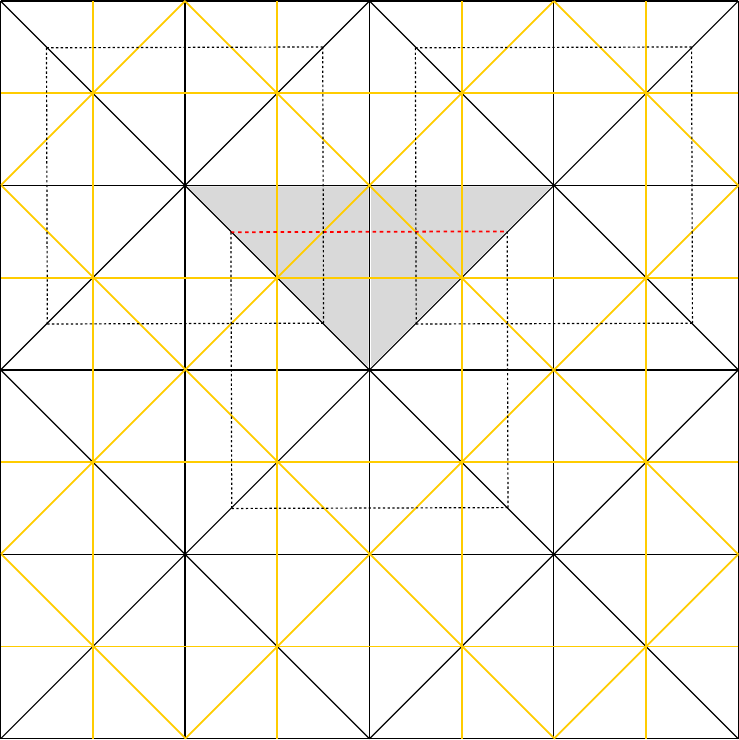}
    \caption{Using the diagonal grid to construct an almost veering branched surface on $T^1 S^2(2,4,4)$.}
    \label{fig:diaggridsur}
\end{figure}

Similar to \Cref{subsec:g2}, we construct a branched surface $\widetilde{B_0}$ on $T^1 \mathbb{R}^2$ by taking the horizontal planes at the levels inbetween the slopes of the grid line and attaching infinite strips lying over the midway lines (in yellow in \Cref{fig:diaggridsur}). We let the interested reader fill in precise descriptions of the surfaces that make up $\widetilde{B_0}$ as in \Cref{subsec:g2} for themselves. 

As in \Cref{subsec:g2}, $\widetilde{B_0}$ can be quotiented down to an almost veering branched surface $B_0$ on $T^1 S^2(2,4,4)$ with the cores of its complementary regions given by $\overset{\leftrightarrow}{c}$, and to an almost veering branched surface $B_0/\langle \iota \rangle $ on $S^3$ with the cores of its complementary regions given by the Montesinos link $M(-\frac{1}{2}, \frac{1}{4}, \frac{1}{4})$.

Next we will locate a system of equivariant connecting annuli in $B_0$. This is the point where this case becomes more interesting than the last one, since this time the system has 2 annuli. For each vertex $(x_0,y_0)$ in the orbit of $(0,0)$ under $G$, consider the square $H$ in $\mathbb{R}^2$ with vertices $(x_0 \pm \frac{3}{2},y_0 \pm \frac{3}{2})$ and consider the torus that is $T^1 \mathbb{R}^2|_H$. The square $H$ for the point $(0,0)$ is drawn in dotted lines in the center of \Cref{fig:diaggridsur}.

Let $E$ be the edge of $H$ between $(x_0 - \frac{3}{2}, y_0 + \frac{3}{2})$ and $(x_0 + \frac{3}{2}, y_0 + \frac{3}{2})$, which is the red segment in \Cref{fig:diaggridsur}. 
We draw $T^1 \mathbb{R}^2|_E$ in \Cref{fig:diagsection}, where the black lines in the interior denote its intersection with the branched surface $\widetilde{B_0}$. Similarly as in the last case, $T^1 \mathbb{R}^2|_H$ can be obtained by taking 4 copies of \Cref{fig:diagsection}, and gluing them up cyclically with a shift of $\frac{\pi}{2}$. 

\begin{figure}
    \centering
    \resizebox{!}{6cm}{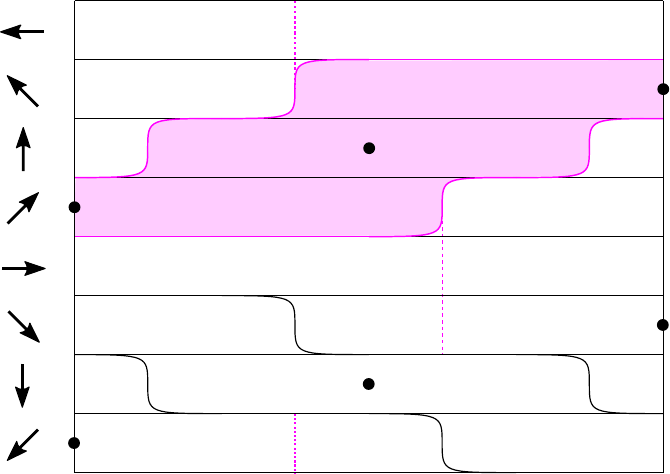}
    \caption{The intersection of $T^1 \mathbb{R}^2|_E$ with $\widetilde{B_0}$. The pink region defines a connecting annulus of $\widetilde{B_0}$. The black dots denote where the full lifts of the grid lines intersect $T^1 \mathbb{R}^2|_E$. The pink vertical dotted lines denote where the other connecting annuli intersect $T^1 \mathbb{R}^2|_E$.}
    \label{fig:diagsection}
\end{figure}

We can find an annulus $\widetilde{A_1}$ on $T^1 \mathbb{R}^2|_H$ by taking the union of the pink regions in the 4 copies of \Cref{fig:diagsection}. It is straightforward to check that $\widetilde{A_1}$ is a connecting annulus between a pair of parallel horizontal surgery curves in $\widetilde{B_0}$, and that $\widetilde{A_1}$ is preserved by the stabilizer of $(x_0,y_0)$ in $G$. Hence upon quotienting by $G^+$, $\widetilde{A_1}$ descends to an equivariant connecting annulus $A_1$ of $B_0$.

But now notice that everything done above can be repeated for a vertex $(x_0,y_0)$ in the orbit of $(2,2)$ under $G$ instead. (In fact, there is an element of $\Isom(\mathbb{R}^2)$ sending $(0,0)$ to $(2,2)$ which preserves the whole setup.) In particular, we obtain another connecting annulus $\widetilde{A_2}$ of $\widetilde{B_0}$ which descends to an equivariant connecting annulus $A_2$ of $B_0$.

Each $A_i$ can be seen to be embedded in $T^1 S^2(2,4,4)$, but $A_1$ meets $A_2$ transversely along their boundaries. This can be deduced from looking at \Cref{fig:diagsection}, where we drew pink dotted lines where the preimages of $A_2$ come through one of the $T^1 \mathbb{R}^2|_E$ defining $\widetilde{A_1}$. It is straightforward to check that these intersection points satisfy (4) of \Cref{defn:hsursim}, hence $A_1, A_2$ is a system of equivariant connecting annuli, and one can read off from \Cref{fig:diagsection} that, in the notation of \Cref{prop:hsursumtriplepoints}, $n_1=n_2=2, m_1=m_2=2, q_{11}=q_{22}=0, q_{12}=q_{21}=1$, hence $C=\begin{bmatrix} 0 & 1 \\ 1 & 0 \end{bmatrix}, d=\begin{bmatrix} 4 \\ 4 \end{bmatrix}$.

Now for $(k_1,k_2) > (0,0)$, we can apply concurrent halved $(\frac{1}{-k_1}, \frac{1}{-k_2})$ horizontal surgery on this system to get an almost veering branched surface $B_{(k_1,k_2)}$ on a 3-manifold $M_{(k_1,k_2)}$ with involution $\iota_{(k_1,k_2)}$. As in \Cref{subsec:g2}, one can see that $A_i/\langle \iota \rangle$ are discs around the base of the two rational tangles $\frac{1}{4}$, with the boundary oriented clockwise when viewed from above. Hence when we do $(\frac{1}{-k_1}, \frac{1}{-k_2})$ surgery, we add $k_1$ half twists to one tangle and $k_2$ half twists to another, to arrive at the Montesinos link $M(\frac{3}{2}, \frac{1}{4+k_1}-1, \frac{1}{4+k_2}-1)$. Taking the branched double cover, we see that $M_{(k_1,k_2)}=T^1 S^2(2,4+k_1,4+k_2)$ and the cores of the complementary regions of $B_{(k_1,k_2)}$ are given by the full lift of the curve $c$ on $S^2(2,4+k_1,4+k_2)$. 

By \Cref{thm:fulllifthyp}, $T^1 S^2(2,4+k_1,4+k_2) \backslash \overset{\leftrightarrow}{c}$ is hyperbolic, hence $B_{(k_1,k_2)}$ and thus $B_{(k_1,k_2)}/\langle \iota_{(k_1,k_2)} \rangle$ are veering. Taking the dual ideal triangulation of the latter, this proves \Cref{thm:genus0} in this case.

Again, we count the number of tetrahedra and the number of blue/red edges in these veering triangulations. We take the orientation on $T^1 \mathbb{R}^2$ to be that given by $(\frac{\partial}{\partial x},\frac{\partial}{\partial y}, \frac{\partial}{\partial \theta}) $. One can check that $B_0$ has $4$ red triple points. By \Cref{prop:hsursumtriplepoints}, concurrent halved $(\frac{1}{-k_1}, \frac{1}{-k_2})$ horizontal surgery on our system produces $2k_1k_2+4k_1+4k_2$ blue triple points, and so $B_{(k_1,k_2)}$ has $2k_1k_2+4k_1+4k_2$ blue triple points and $4$ red triple points, and $B_{(k_1,k_2)}/\iota_{(k_1,k_2)}$ has $k_1k_2+2k_1+2k_2$ blue triple points and $2$ red triple points. Dual to this, the veering triangulation on $S^3 \backslash M(\frac{3}{2}, \frac{1}{4+k_1}-1, \frac{1}{4+k_2}-1)$ has $k_1k_2+2k_1+2k_2+2$ tetrahedra, $k_1k_2+2k_1+2k_2$ blue edges and $2$ red edges.

\subsection{Case 3: $n=3, (p_1,p_2,p_3)>(3,3,3)$} \label{subsec:hex}

The strategy here is again the same. This time we use the hexagonal grid on $\mathbb{R}^2$. This is given by taking the union of the lines
\begin{align*}
\{ y=2n \}_{n \in \mathbb{Z}} \cup \{ y= \sqrt{3}x+4n \}_{n \in \mathbb{Z}} \cup \{ y= -\sqrt{3}x+4n \}_{n \in \mathbb{Z}} 
\end{align*}

See the black lines in \Cref{fig:hexgridsur}.

\begin{figure}
    \centering
    \resizebox{!}{7.5cm}{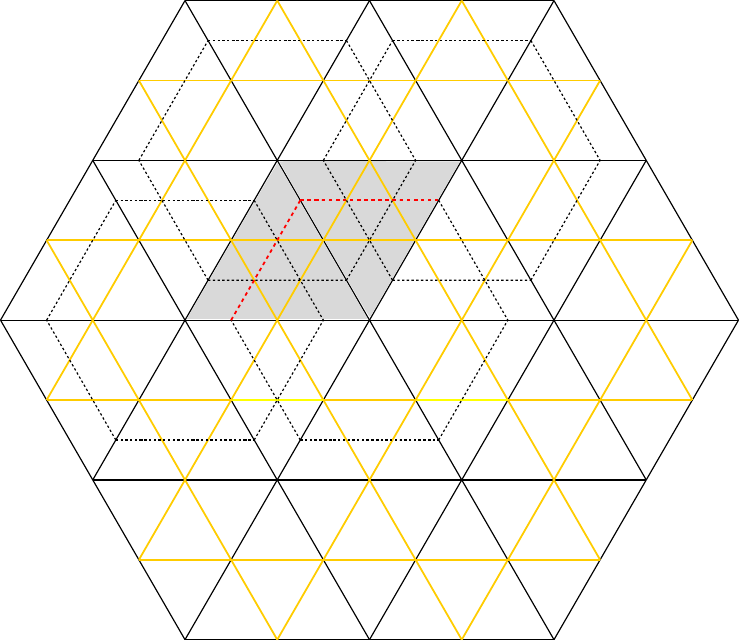}
    \caption{Using the hexagonal grid to construct an almost veering branched surface on $T^1 S^2(3,3,3)$.}
    \label{fig:hexgridsur}
\end{figure}

The same construction gives a branched surface $\widetilde{B_0}$ on $T^1 \mathbb{R}^2$, which quotients down to an almost veering branched surface $B_0$ on $T^1 S^2(3,3,3)$ and an almost veering branched surface $B_0/\langle \iota \rangle$ on $S^3$.

Next we locate a system of equivariant connecting annuli in $B_0$. For each vertex $(x_0,y_0)$ in the orbit of $(0,0)$ under $G$, consider the hexagon $H$ in $\mathbb{R}^2$ with vertices $(x_0 \pm \sqrt{3},y_0), (x_0 \pm \frac{\sqrt{3}}{2}, y_0 \pm \frac{3}{2})$ and consider the torus $T^1 \mathbb{R}^2|_H$. The hexagon $H$ for the point $(0,0)$ is drawn in dotted lines in the center of \Cref{fig:hexgridsur}.

Let $E$ be the union of the edge of $H$ between $(x_0-\sqrt{3}, y_0)$ and $(x_0 - \frac{\sqrt{3}}{2}, y_0 + \frac{3}{2})$ and the edge of $H$ between $(x_0 - \frac{\sqrt{3}}{2}, y_0 + \frac{3}{2})$ and $(x_0 + \frac{\sqrt{3}}{2}, y_0 + \frac{3}{2})$, which we draw in red in \Cref{fig:hexgridsur}. We draw $T^1 \mathbb{R}^2|_E$ in \Cref{fig:hexsection}.

\begin{figure}
    \centering
    \resizebox{!}{4.5cm}{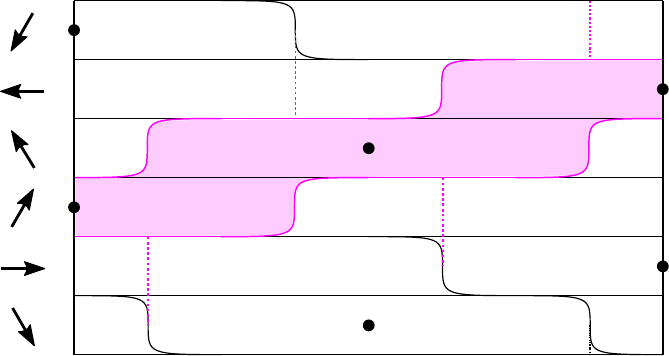}
    \caption{The intersection of $T^1 \mathbb{R}^2|_E$ with $\widetilde{B_0}$. The pink region defines a connecting annulus of $\widetilde{B_0}$. The black dots denote where the full lifts of the grid lines intersect $T^1 \mathbb{R}^2|_E$. The pink vertical dotted lines denote where the other connecting annuli intersect $T^1 \mathbb{R}^2|_E$.}
    \label{fig:hexsection}
\end{figure}

We can find an annulus $\widetilde{A_1}$ on $T^1 \mathbb{R}^2|_H$ by taking a union of the pink regions in copies of \Cref{fig:hexsection}. $\widetilde{A_1}$ is a connecting annulus between a pair of parallel horizontal surgery curves on $\widetilde{B_0}$. Upon quotienting by $G^+$, $\widetilde{A_1}$ descends to an equivariant connecting annulus $A_1$ of $B_0$.

We can repeat everything for a vertex in the orbit of $(\frac{2}{\sqrt{3}},2)$ or in the orbit of $(-\frac{2}{\sqrt{3}}, 2)$ under $G$ instead. (In fact, there are elements of $\Isom(\mathbb{R}^2)$ sending $(0,0)$ to $(\frac{2}{\sqrt{3}},2)$ and to $(-\frac{2}{\sqrt{3}},2)$ which preserves the whole setup.) In particular, we obtain two more connecting annuli $\widetilde{A_2}, \widetilde{A_3}$ of $\widetilde{B_0}$ which descend to equivariant connecting annuli $A_2, A_3$ of $B_0$.

Each $A_i$ is embedded in $T^1 S^2(3,3,3)$, but meets each other transversely along their boundaries. This can be deduced from looking at \Cref{fig:hexsection}, where we drew pink dotted lines where the preimages of $A_2, A_3$ come through one of the $T^1 \mathbb{R}^2|_E$ defining $\widetilde{A_1}$. Each of these intersection points satisfies (4) of \Cref{defn:hsursim}, hence $A_1, A_2, A_3$ is a system of equivariant connecting annulus, and one can read off from \Cref{fig:hexsection} that $n_i=2, m_i=2$ for all $i$, $q_{ij}=0$ for $i=j$, $q_{ij}=1$ for $i \neq j$, hence $C=\begin{bmatrix} 0 & 1 & 1 \\ 1 & 0 & 1 \\ 1 & 1 & 0 \end{bmatrix}, d=\begin{bmatrix} 4 \\ 4 \\ 4 \end{bmatrix}$, in the notation of \Cref{prop:hsursumtriplepoints}.

Now for $(k_1,k_2,k_3) > (0,0,0)$, apply concurrent halved $(\frac{1}{-k_1}, \frac{1}{-k_2}, \frac{1}{-k_3})$ horizontal surgery on this system. This gives a veering branched surface $B_{(k_1, k_2, k_3)}$ on $T^1 S^2(3+k_1,3+k_2,3+k_3)$, which, after drilling out the cores of the complementary regions, descends to a veering branched surface on $S^3 \backslash M(\frac{1}{3+k_1}+1, \frac{1}{3+k_2}-1, \frac{1}{3+k_3}-1)$. Taking its dual ideal triangulation, this proves \Cref{thm:genus0} in this case.

Taking the orientation on $T^1 \mathbb{R}^2$ to be that given by $(\frac{\partial}{\partial x},\frac{\partial}{\partial y}, \frac{\partial}{\partial \theta}) $, $B_0$ has $6$ red triple points and $(\frac{1}{-k_1}, \frac{1}{-k_2}, \frac{1}{-k_3})$-surgery on our system produces $2k_1k_2+2k_1k_3+2k_2k_3+4k_1+4k_2+4k_3$ blue triple points, and so $B_{(k_1,k_2, k_3)}$ has $2k_1k_2+2k_1k_3+2k_2k_3+4k_1+4k_2+4k_3$ blue triple points and $6$ red triple points. Hence the veering triangulation on $S^3 \backslash M(\frac{1}{3+k_1}+1, \frac{1}{3+k_2}-1, \frac{1}{3+k_3}-1)$ has $k_1k_2+k_1k_3+k_2k_3+2k_1+2k_2+2k_3+3$ tetrahedra, $k_1k_2+k_1k_3+k_2k_3+2k_1+2k_2+2k_3$ blue edges and $3$ red edges. 

\subsection{Case 4: $n=4, (p_1,p_2,p_3,p_4)>(2,2,2,2)$, or $n \geq 5$} \label{subsec:ngeq4}

The way we construct the initial branched surface $B_0$ will be different in this last case. Instead of prescribing surfaces that comprise $\widetilde{B_0}$, we describe the branched surface using a movie of train tracks. In the same vein, we will describe the system of equivariant connecting surgery using the trace of an interval in the movie. We remark that the subcase when $n=4$ and $(p_1,p_2,p_3,p_4)>(2,2,2,2)$ can be tackled using the same strategy as in the first three cases via a rectangular grid on $\mathbb{R}^2$, but for the benefit of \Cref{sec:highergenus} we choose to present the construction in the way that we do.

Let $Q$ be a disc with $n$ corners. Choose some orientation on $Q$. Let $d_1,...,d_n$ be the vertices along $\partial Q$, taken in some cyclic and counterclockwise order, and let $c_i$ be the side of $Q$ going from $d_i$ to $d_{i+1}$. Choose some Riemannian metric on $Q$ such that each $c_i$ is a geodesic and each $d_i$ is a right angle. 

Let $T$ be the solid torus that is $T^1 Q$. One can construct a branched surface in $T$ by starting from the desired intersection of the branched surface with $\partial T$, then sweeping inwards to get a movie of train tracks on tori, and finally `capping off' by specifying what the branched surface looks like near the core of $T$. We will be illustrating these movies as frames obtained by looking from outside $T$ as we sweep the torus inwards. In addition, we will want to orient the components of the branch locus of this branched surface. This can be done by consistently coorienting every switch of the train tracks in the movie into or out of the page. If one wants the orientations to satisfy (iii) of \Cref{defn:vbs} (or \Cref{defn:avbs}), only certain moves are allowed in the movie. These moves are illustrated in \Cref{fig:moves1}, where $\otimes$ means that the orientation is into the page while $\odot$ means that the orientation is out of the page. The portion of the branched surface near the core of $T$, which we use to cap off the movie, should be verified to satisfy (iii) as well.

\begin{figure}
    \centering
    \fontsize{8pt}{8pt}\selectfont
    \resizebox{!}{6cm}{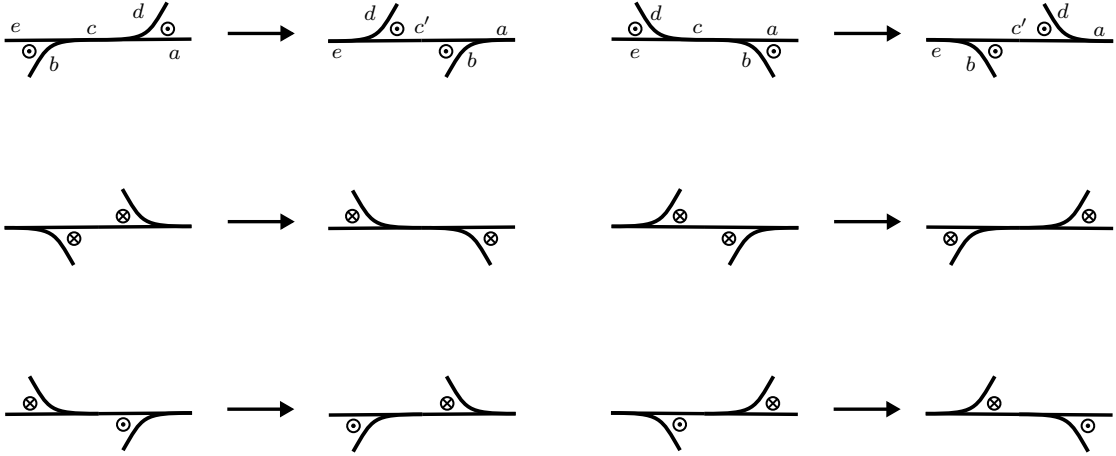}
    \caption{Moves allowed in the movie of train tracks in order to satisfy (iii) of \Cref{defn:vbs} (or \Cref{defn:avbs}).}
    \label{fig:moves1}
\end{figure}

Actually, for the rest of this paper, we will only need to use the two moves in the top row in \Cref{fig:moves1}. We will refer to this move as \textit{splitting branches $b$ and $d$ across $c$}, with the branches labelled as in \Cref{fig:moves1}. We will also label the branches after the move as in \Cref{fig:moves1}, namely, we add a prime to the branch being split along and retain the labels of the rest of the branches.

With this strategy in mind, we can construct a branched surface inside $T$ using the movie in \Cref{fig:highergenusfillmovie} (ignoring the pink cooriented intervals for now) and capping off by \Cref{fig:moviecore}. 

To describe the movie precisely, suppose each $c_i$ is parametrized to go from $t=0$ to $t=2$. Let $R_\theta:T^1 S|_c \to T^1 S|_c$ be the map that rotates vectors by $\theta$ counterclockwise. 

Construct a train track $\tau_0$ on $\partial T$ by first taking the union of the $4n$ line segments $\{R_{k \frac{\pi}{2}+\frac{\pi}{4}} c'_i(t)\}$, where $k=1,...,4, i=1,...,n$. Then add the $n$ branches $\{R_{\theta} c'_i(1):-\frac{3\pi}{4} \leq \theta \leq -\frac{\pi}{4} \}$, with the top endpoint $R_{-\frac{\pi}{4}} c'_i(1)$ on $R_{-\frac{\pi}{4}} c'_i(t)$ combed in the direction of increasing $t$, and the bottom endpoint $R_{-\frac{3\pi}{4}} c'_i(1)$ on $R_{-\frac{3\pi}{4}} c'_i(t)$ combed in the direction of decreasing $t$. Name these branches $c_{i,N}$ respectively, and coorient the switches which these branches meet out of the page. Similarly, add the $n$ branches $\{R_{\theta} c'_i(1):\frac{\pi}{4} \leq \theta \leq \frac{3\pi}{4} \}$, with the top endpoint $R_{\frac{3\pi}{4}} c'_i(1)$ on $R_{\frac{3\pi}{4}} c'_i(t)$ combed in the direction of decreasing $t$, and the bottom endpoint $R_{\frac{\pi}{4}} c'_i(1)$ on $R_{\frac{\pi}{4}} c'_i(t)$ combed in the direction of increasing $t$. Name these branches $c_{i,S}$ respectively, and coorient the switches which these branches meet into the page. Notice that the branches we add divide the horizontal lines $\bigcup \{R_{k \frac{\pi}{2}+\frac{\pi}{4}} c'_i(t)\}$ into $4n$ branches. Name the branch which $R_{-\frac{3\pi}{4}} c'_i(0)$ lies on $d_{i,N}$, the branch which $R_{-\frac{\pi}{4}} c'_i(0)$ lies on $d_{i,W}$, the branch which $R_{\frac{\pi}{4}} c'_i(0)$ lies on $d_{i,S}$, the branch which $R_{\frac{3\pi}{4}} c'_i(0)$ lies on $d_{i,E}$. See the first frame in \Cref{fig:highergenusfillmovie}. 

Now split $d_{i-2,W}$ and $d_{i,E}$ across $d_{i-1,N}$ to get $\tau_1$. Note that this is the move illustrated in \Cref{fig:moves1} top. See the second frame of \Cref{fig:highergenusfillmovie} for $\tau_1$. Upon simplification, one can see that $\tau_1$ is the same as the train track illustrated in the third frame of \Cref{fig:highergenusfillmovie}. Then split $d_{i-3,W}$ and $d_{i,E}$ across $c_{i-2,N}$ to get $\tau_2$. See the fourth frame of \Cref{fig:highergenusfillmovie}. $\tau_2$ can in turn be simplified to look like the fifth frame of \Cref{fig:highergenusfillmovie}. Inductively, to get from $\tau_{2s}$ to $\tau_{2s+1}$, split $d_{i-2s-2,W}$ and $d_{i,E}$ across $d^{(s)}_{i-s-1,N}$; to get from $\tau_{2s+1}$ to $\tau_{2s+2}$, split $d_{i-2s-3,W}$ and $d_{i,E}$ across $c^{(s)}_{i-s-2,N}$. Continue until we get to $\tau_{n-4}$.

Intuitively, what we are doing is taking the closed curve $\bigcup c_{i,N} \cup d_{i,N}$ and unwinding $\tau_0$ along it. One can consider this to be a variant of the horizontal surgery we described in \Cref{subsec:hsur}, where we spin sectors around a curve. Each step `adds a layer' to the picture, so after $n-4$ steps, $\tau_{n-4}$ has $n$ layers with two branches lying within each layer. At this point we do something different to wrap up the movie.

Namely, if $n$ is even, split $c^{(\frac{n-4}{2})}_{i-1,N}$ and $c^{(\frac{n-4}{2})}_{i,N}$ across $d^{(\frac{n-4}{2})}_{i,N}$ to get $\tau_{n-3}$; if $n$ is odd, split $d^{(\frac{n-3}{2})}_{i,N}$ and $d^{(\frac{n-3}{2})}_{i+1,N}$ across $c^{(\frac{n-5}{2})}_{i,N}$ to get $\tau_{n-3}$. See the sixth frame of \Cref{fig:highergenusfillmovie}. Upon simplification, one can see that $\tau_{n-3}$ is the same as the train track illustrated in the last frame of \Cref{fig:highergenusfillmovie}, which is the boundary of a branched surface of the form of that illustrated in \Cref{fig:moviecore} and with $n$ layers. Hence we can cap off by placing that branched surface in the core of $T$.

More formally, if $n$ is even, take arcs $b_{i+1,W}$ going from the switch between $c_{i,S}$ and $d_{i,S}$ to the switch between $c^{(\frac{n-4}{2})}_{i-\frac{n}{2},N}$ and $d^{(\frac{n-2}{2})}_{i-\frac{n}{2},N}$, and arcs $b_{i-1,E}$ going from the switch between $c_{i-1,S}$ and $d_{i,S}$ to the switch between $c^{(\frac{n-4}{2})}_{i-\frac{n}{2}-1,N}$ and $d^{(\frac{n-2}{2})}_{i-\frac{n}{2},N}$ which are parallel to $\tau_1$. 

We observe that the closed curve $d_{i+1,W} \cup d^{(\frac{n-2}{2})}_{i-\frac{n}{2},N} \cup d_{i-1, E} \cup d_{i,S}$ is homotopically trivial in $T$. This follows from the fact that if we rewind the movie, this curve has the same isotopy class as $d_{i+1,W} \cup \bigcup_{j=i+2}^{i+n-3} d_{j,N} \cup c_{j,N} \cup d_{i-2,N} \cup d_{i-1,E} \cup d_{i,S} \subset \partial T$, which defines a vector field on $\partial Q$ with index $1$, hence extends to a vector field within $Q$. 

Therefore we can complete the branched surface by adding disc sectors bounded by $b_{i+1,W} \cup d_{i+1,W}$, by $b_{i-1, E} \cup d_{i-1, E}$, by $c^{(\frac{n-4}{2})}_{i-\frac{n}{2},N} \cup b_{i+1,W} \cup c_{i,S} \cup b_{i,E}$, and by $d^{(\frac{n-2}{2})}_{i-\frac{n}{2},N} \cup b_{i+1,W} \cup d_{i,S} \cup b_{i-1,E}$. 

Similarly, if $n$ is odd, take arcs $b_{i+1,W}$ going from the switch between $c_{i,S}$ and $d_{i,S}$ to the switch between $c^{(\frac{n-3}{2})}_{i-\frac{n+1}{2},N}$ and $d^{(\frac{n-3}{2})}_{i-\frac{n-1}{2},N}$, and arcs $b_{i-1,E}$ going from the switch between $c_{i-1,S}$ and $d_{i,S}$ to the switch between $c^{(\frac{n-3}{2})}_{i-\frac{n+1}{2},N}$ and $d^{(\frac{n-3}{2})}_{i-\frac{n+1}{2},N}$ which are parallel to $\tau_1$. By observing that the closed curve $d_{i+1,W} \cup c^{(\frac{n-3}{2})}_{i-\frac{n+1}{2},N} \cup d_{i-1, E} \cup d_{i,S}$ is homotopically trivial in $T$, we can complete the branched surface by adding disc sectors bounded by $b_{i+1,W} \cup d_{i+1,W}$, by $b_{i-1, E} \cup d_{i-1, E}$, by $d^{(\frac{n-3}{2})}_{i-\frac{n-1}{2},N} \cup b_{i+1,W} \cup c_{i,S} \cup b_{i,E}$ and by $c^{(\frac{n-3}{2})}_{i-\frac{n+1}{2},N} \cup b_{i+1, W} \cup d_{i,S} \cup b_{i-1,E}$. 

\begin{figure}
    \centering
    \fontsize{10pt}{10pt}\selectfont
    \resizebox{!}{16cm}{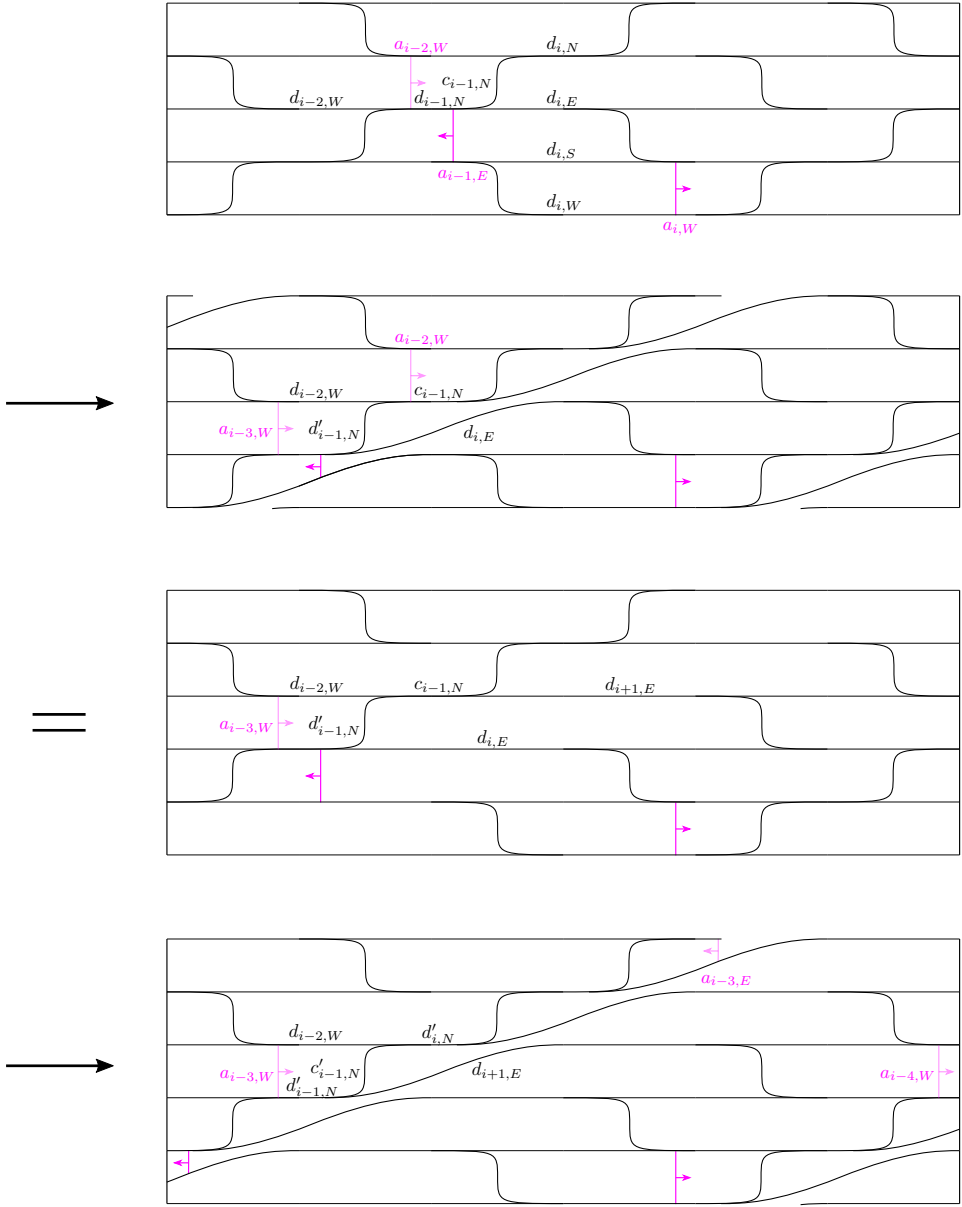}
    \caption{The movie of train tracks we use to construct $B_Q \cap T$.}
\end{figure}

\begin{figure}
    \ContinuedFloat
    \centering
    \fontsize{10pt}{10pt}\selectfont
    \resizebox{!}{20cm}{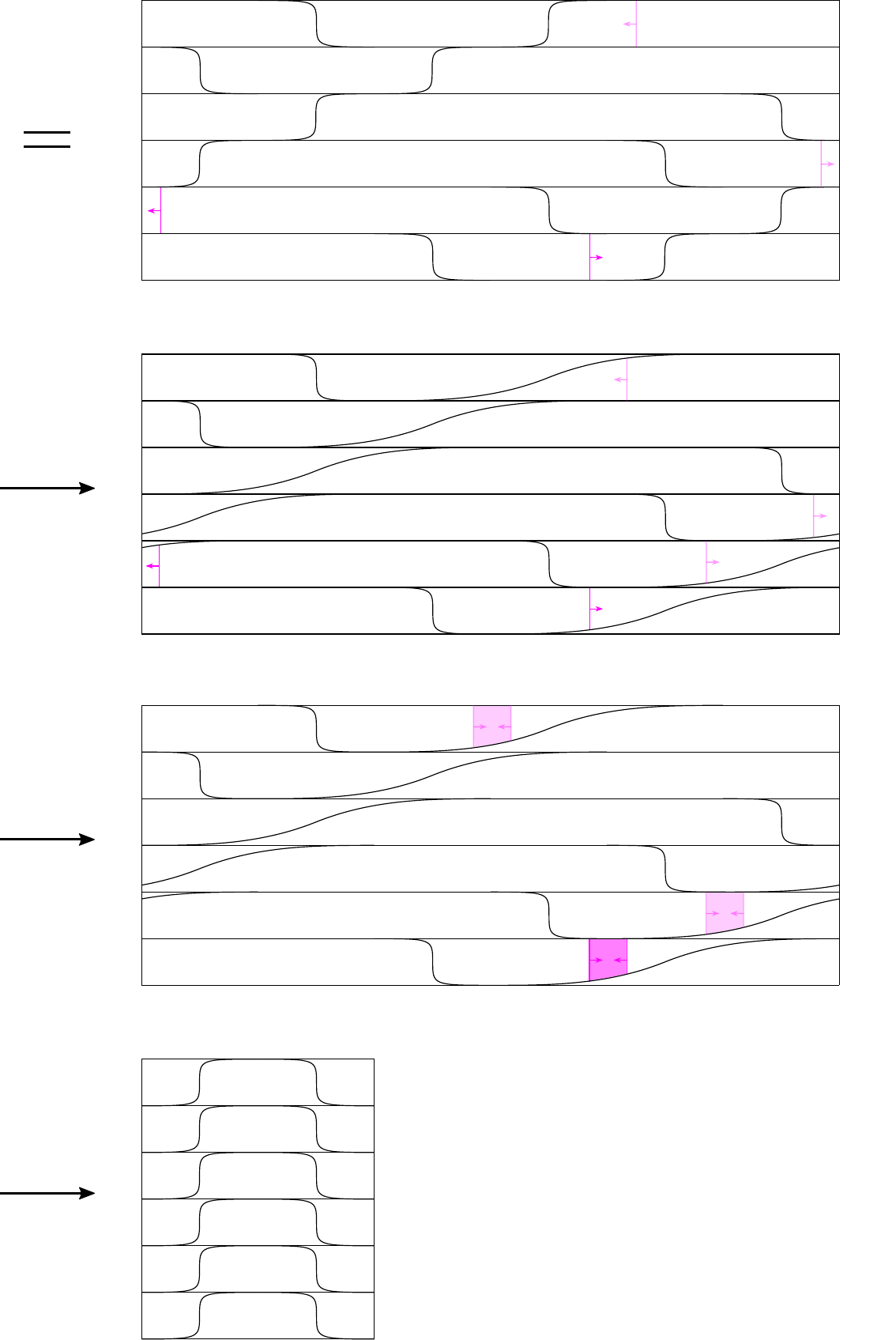}
    \caption{The movie of train tracks we use to construct $B_Q \cap T$.}
    \label{fig:highergenusfillmovie}
\end{figure}

\begin{figure}
    \centering
    \resizebox{!}{5cm}{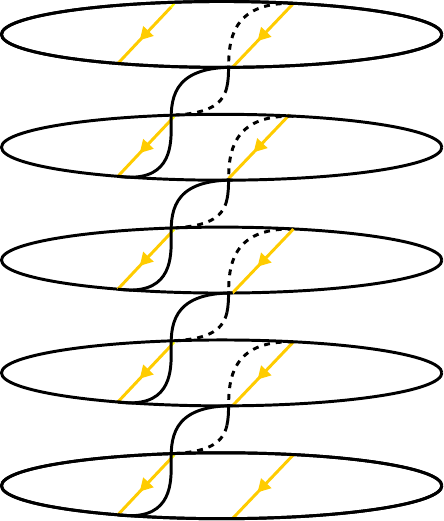}
    \caption{The portion of the branched surface near the core of $T$, which we use to cap off the movie in \Cref{fig:highergenusfillmovie}.}
    \label{fig:moviecore}
\end{figure}

In any case, we now have a branched surface in $T$ with orientations defined on the components of its branch locus. We call this branched surface $B_Q$.

The orbifold $S^2(2,...,2)$ is the union of two copies of $Q$, and $T^1 S^2(2,...,2)$ is the union of two copies of $T$ where each face $c_i \times S^1$ is glued to its other copy by reflection across $\{\pm c'_i(t)\}$. The initial train track $\tau_0=B_Q \cap \partial T$ is preserved under this reflection, hence we can glue together the two copies of $B_Q$ to get a branched surface $B_0$ on $T^1 S^2(2,...,2)$. Interchanging the two copies of $T$ defines an involution $\iota$ that preserves $B_0$.

Next we have to describe a system of connecting annuli. Observe that one can construct a path lying on $B_Q$ by tracing out points lying on the train tracks in the movie. Similarly, a cooriented surface can be constructed by tracing out cooriented intervals in the movie. If one wants the constructed surface to be part of a connecting annulus surgery along which produces blue triple points, then the endpoints of the intervals can only perform the moves illustrated in \Cref{fig:moves2}, and the interior of the intervals must lie away form the switches of the train tracks. Of course, there are symmetric moves if one wants to construct a connecting annulus surgery along which produces red triple points, but we will let the reader fill those out.

For a collection of such surfaces to form a system of connecting annuli, the interior of the intervals must be disjoint from each other, and the endpoints of the intervals must only move past each other in the moves illustrated in \Cref{fig:moves3}. 

We remark that for the rest of this paper, we will only need to use the moves in the top rows of \Cref{fig:moves2} and \Cref{fig:moves3}.

\begin{figure}
    \centering
    \resizebox{!}{8cm}{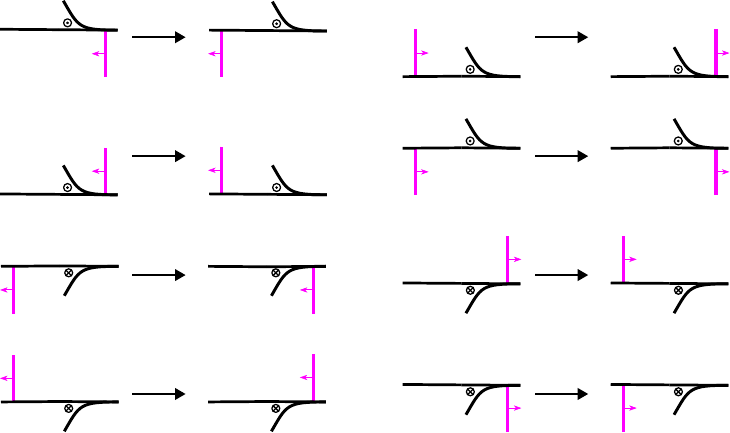}
    \caption{Moves allowed in the movie of train tracks in order for the path traced out by the endpoint of the cooriented interval to be a horizontal surgery curve, surgery on which produces blue triple points.}
    \label{fig:moves2}
\end{figure}

\begin{figure}
    \centering
    \resizebox{!}{4cm}{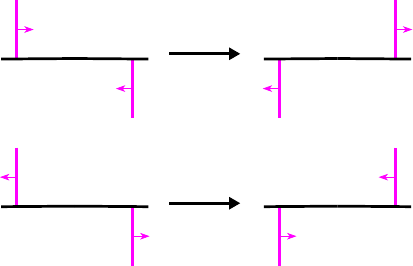}
    \caption{Moves allowed in the movie of train tracks in order for the cooriented intervals to determine a system of connecting annuli, surgery on which produces blue triple points.}
    \label{fig:moves3}
\end{figure}

Using this idea, we build surfaces that lie in $T$ using the pink cooriented intervals and patches in \Cref{fig:highergenusfillmovie}; the patches cap off traces of the vertical intervals to determine the surfaces.

Again, we will describe this formally. On $\tau_0$, let $a_{i,W}=\{R_{\theta} c'_i(\frac{3}{2}):-\frac{\pi}{4} \leq \theta \leq \frac{\pi}{4} \}$, cooriented in the direction of increasing $t$ along $c_i(t)$, and let $a_{i,E}=\{R_{\theta} c'_{i-1}(\frac{1}{2}):\frac{3\pi}{4} \leq \theta \leq \frac{5\pi}{4} \}$, cooriented in the direction of decreasing $t$ along $c_{i-1}(t)$.

On $\tau_0$, move the top endpoint of $a_{i,E}$ past the bottom endpoint of $a_{i-2,W}$ on $d_{i-1,N}$, then from $d_{i-1,N}$ to $c_{i-2,N}$. Symmetrically, move the bottom endpoint of $a_{i,W}$ past the top endpoint of $a_{i+2,E}$ on $d_{i+1,N}$, then from $d_{i+1,N}$ to $c_{i+1,N}$. Then do splitting moves to get $\tau_1$. See the second frame of \Cref{fig:highergenusfillmovie}. On $\tau_1$, move the top endpoint of $a_{i,E}$ past the bottom endpoint of $a_{i-3,W}$ on $c_{i-2,N}$, then from $c_{i-2,N}$ to $d'_{i-2,N}$. Symmetrically, move the bottom endpoint of $a_{i,W}$ past the top endpoint of $a_{i+3,E}$ on $c_{i+1,N}$, then from $c_{i+1,N}$ to $d'_{i+2,N}$. Then do splitting moves to get $\tau_2$. See the fourth frame of \Cref{fig:highergenusfillmovie}. Inductively, on $\tau_{2s}$, move the top endpoint of $a_{i,E}$ past the bottom endpoint of $a_{i-2s-2,W}$ on $d^{(s)}_{i-s-1,N}$, then from $d^{(s)}_{i-s-1,N}$ to $c^{(s)}_{i-s-2,N}$, and symmetrically move the bottom endpoint of $a_{i,W}$ past the top endpoint of $a_{i+2s+2,E}$ on $d^{(s)}_{i+s+1,N}$, then from $d^{(s)}_{i+s+1,N}$ to $c^{(s)}_{i+s+1,N}$, before splitting to get $\tau_{2s+1}$. On $\tau_{2s+1}$, move the top endpoint of $a_{i,E}$ past the bottom endpoint of $a_{i-2s-3,W}$ on $c^{(s)}_{i-s-2,N}$, then from $c^{(s)}_{i-s-2,N}$ to $d^{(s+1)}_{i-s-2,N}$, and symmetrically move the top endpoint of $a_{i,W}$ past the top endpoint of $a_{i+2s+3,E}$ on $c^{(s)}_{i+s+1,N}$, then from $c^{(s)}_{i+s+1,N}$ to $d^{(s+1)}_{i+s+2,N}$ before splitting to get $\tau_{2s+2}$.

When we get to $\tau_{n-4}$, if $n$ is even, move the top endpoint of $a_{i,E}$ past the bottom endpoint of $a_{i-n+2,W}$ on $d^{(\frac{n-4}{2})}_{i-\frac{n-2}{2},N}$, then from $d^{(\frac{n-4}{2})}_{i-\frac{n-2}{2},N}$ to $c^{(\frac{n-4}{2})}_{i-\frac{n}{2},N}$. Symmetrically, move the bottom endpoint of $a_{i,W}$ past the top endpoint of $a_{i+n-2,E}$ on $d^{(\frac{n-4}{2})}_{i+\frac{n-2}{2},N}$, then from $d^{(\frac{n-4}{2})}_{i+\frac{n-2}{2},N}$ to $c^{(\frac{n-4}{2})}_{i+\frac{n-2}{2},N}$, before splitting to get $\tau_{n-3}$. If $n$ is odd, move the top endpoint of $a_{i,E}$ past the bottom endpoint of $a_{i-n+2,W}$ on $c^{(\frac{n-5}{2})}_{i-\frac{n-1}{2},N}$, then from $c^{(\frac{n-5}{2})}_{i-\frac{n-1}{2},N}$ to $d^{(\frac{n-3}{2})}_{i-\frac{n-1}{2},N}$. Symmetrically, move the bottom endpoint of $a_{i,W}$ past the top endpoint of $a_{i+n-2,E}$ on $c^{(\frac{n-5}{2})}_{i+\frac{n-3}{2},N}$, then from $c^{(\frac{n-5}{2})}_{i+\frac{n-3}{2},N}$ to $d^{(\frac{n-3}{2})}_{i+\frac{n-1}{2},N}$, before splitting to get $\tau_{n-3}$. See the sixth frame in \Cref{fig:highergenusfillmovie}. 

Then, if $n$ is even, move the top endpoint of $a_{i,E}$ past the bottom endpoint of $a_{i-n+1,W}$ on $c^{(\frac{n-4}{2})}_{i-\frac{n}{2},N}$, then from $c^{(\frac{n-4}{2})}_{i-\frac{n}{2},N}$ to $d_{i+1,W}$, and move the bottom endpoint of $a_{i,E}$ from $d_{i-1,E}$ to $c^{(\frac{n-4}{2})}_{i-\frac{n+2}{2},N}$, pushing $a_{i,E}$ across $d^{(\frac{n-2}{2})}_{i-\frac{n}{2},N}$ at the same time. Now $a_{i,E}, a_{i,W}$ and subintervals of $c^{(\frac{n-4}{2})}_{i-\frac{n+2}{2},N}$ and $d_{i+1,W}$ bound a rectangle, whose interior is disjoint from $\tau_1$, and to which the coorientations of $a_{i,E}, a_{i,W}$ are pointing inwards. If $n$ is odd, move the top endpoint of $a_{i,E}$ past the bottom endpoint of $a_{i-n+1,W}$ on $d^{(\frac{n-3}{2})}_{i-\frac{n-1}{2},N}$, then from $d^{(\frac{n-3}{2})}_{i-\frac{n-1}{2},N}$ to $d_{i+1,W}$, and move the bottom endpoint of $a_{i,E}$ from $d_{i-1,E}$ to $d^{(\frac{n-3}{2})}_{i-\frac{n+1}{2},N}$, pushing $a_{i,E}$ across $c^{(\frac{n-3}{2})}_{i-\frac{n+1}{2},N}$ at the same time. Now $a_{i,E}, a_{i,W}$ and subintervals of $d^{(\frac{n-3}{2})}_{i-\frac{n+1}{2},N}$ and $d_{i+1,W}$ bound a rectangle, whose interior is disjoint from $\tau_1$, and to which the coorientations of $a_{i,E}, a_{i,W}$ are pointing inwards. See the second to last frame in \Cref{fig:highergenusfillmovie}. The traces of $a_{i,E}$ and $a_{i,W}$ together with this rectangle gives a rectangular surface $A_{i,Q}$ in $T$.

For each $i$, the intervals $A_{i,Q} \cap \partial T$ are preserved under reflection across $\{\pm c'_i(t)\}$, hence in $T^1 S^2(2,...,2)$ we can glue together the two copies of $A_{i,Q}$ to get an annuli $A_i$ for $B_0$ on $T^1 S^2(2,...,2)$. 

\begin{claim} \label{lemma:ngeq4vbs}
$B_0$ is an almost veering branched surface and $\{A_i\}$ is a system of equivariant connecting annuli. If $n \geq 5$ then $B_0$ is veering. The cores of the complementary regions of $B_0$ are given by $\overset{\leftrightarrow}{c}$.
\end{claim}

\begin{proof}
Notice that we can give $T$ a natural structure of a 3-manifold with cusps and corners by declaring that the fibers above $c_i$ are smooth faces which meet along corner edges that are the fibers above $d_i$. This induces corners on the components of $T \cut B_Q$, namely between faces that lie on $\partial T$ and between faces that lie on $\partial T$ and faces that lie on $B_Q$. Under this corner structure, we claim that components of $T \cut B_Q$ are 1-cusped triangles times an interval. 

To explain this, let us first disregard the corner structures. Call the 3-manifold that is a 1-cusped triangle times an interval but forgetting the corner structure a \textit{taco}. See \Cref{fig:taco}. Notice that throughout the movie of train tracks, the topology of the complementary regions of the train tracks do not change. Hence inside $T$ but outside of the core, the complementary regions of $B_Q$ are cusped bigons times an interval. Now the complementary regions of $B_Q$ inside the core are tacos, so adding in the products, the same is true for complementary regions of $B_Q$ in $T$. Finally, to put the corners back into the picture, each cusped bigon complementary region of $\tau_0$ on $\partial T$ meets the corner edges in two intervals, hence the topology of the complementary regions of $B_Q$ in $T$ is as claimed.

\begin{figure}
    \centering
    \resizebox{!}{3cm}{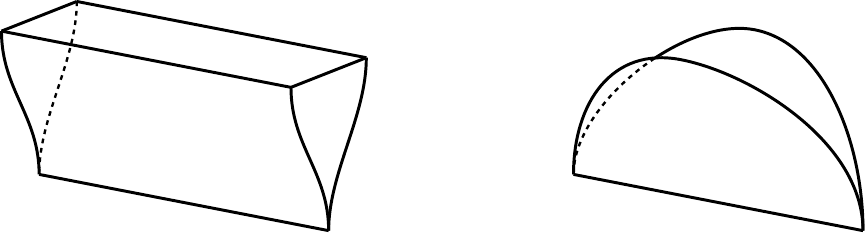}
    \caption{Forgetting the corner structure of a 1-cusped triangle times interval (left) results in a taco (right).}
    \label{fig:taco}
\end{figure}

One glues together faces of these complementary regions to form complementary regions of $B_0$, with four of these pieces coming together at each corner. The only possible result of such gluings are 2-cusped solid tori. This proves \Cref{defn:avbs}(ii).

Notice that we can obtained the cores of these solid tori by taking the union of the line segments $\{\pm c'_i(t)\}$ above $c_i$. This shows that the cores of the complementary regions of $B_0$ are exactly given by $\overset{\leftrightarrow}{c}$. 

We now show \Cref{defn:avbs}(i). Suppose there is a disc sector without corners. The boundary of the sector cannot be cooriented outwards, otherwise one of the complementary regions of $B_0$ adjacent to the sector cannot be a 2-cusped solid torus. But if the boundary of the sector is cooriented inwards,  then the complementary region of $B_0$ that contains the component of the branch locus that is the boundary would be a 2-cusped solid torus whose core is homotopically trivial in $T^1 S^2(2,...,2)$. But we know that no component of $\overset{\leftrightarrow}{c}$ is homotopically trivial, hence we reach a contradiction.

\Cref{defn:vbs}(iii) is true by construction, since in the movie we only performed moves as listed in \Cref{fig:moves1}, and for the portion of the branched surface we used to cap off the movie, there are no triple points. This proves that $B_0$ is an almost veering branched surface.

For $n \geq 5$, $T^1 S^2(2,...,2) \backslash \overset{\leftrightarrow}{c}$ is hyperbolic by \Cref{thm:fulllifthyp}. Hence by \Cref{prop:avbshyp}, $B_0$ is veering in this case.

Finally, that the $A_i$ form a system of equivariant connecting annuli follows from the fact that we only did the moves listed in \Cref{fig:moves2} and \Cref{fig:moves3}, and the rectangles we added during $\tau_{n-3}$ lie away from the branch locus.
\end{proof}

Tracing through the movie, we see that we did $n(n-3)$ splitting moves (more specifically, $n$ moves for every frame advanced), hence there are $n(n-3)$ triple points in $B_Q$, and $2n(n-3)$ triple points in $B_0$ in total. Among these, $2n(n-4)$ are blue and $2n$ triple points are red, since the moves done during the first $n-4$ frames give blue triple points while the moves done in the last frame give red triple points. 

Within the movie, the top boundary component of $A_{i,Q}$ meets the branch locus of $B_Q$ for a total of $n-3$ times on the side away from $A_i$ (once for each of the first $n-3$ frames) and 1 time on the side of $A_i$ (on the last frame). Similarly, the bottom component of $A_{i,Q}$ meets the branch locus of $B_Q$ for a total of $n-3$ times on the side away from $A_i$ and 1 time on the side of $A_i$. Hence by symmetry, a boundary component of $A_i$ meets the branch locus of $B_0$ a total of $2(n-3)$ times on the side away from $A_i$ and 2 times on the side of $A_i$. In other words, in the notation of \Cref{subsec:hsurvar}, $n_i=2, m_i=2(n-3)$ for all $i$. 

Finally, the top boundary component of $A_i$ meets the bottom boundary components of $A_{i-2},...,A_{i-n+1}$ each once inside $T$. The bottom boundary component of $A_i$ meets the top boundary components of $A_{i+2},...,A_{i+n-1}$ each once inside $T$. Hence in the notation of \Cref{subsec:hsurvar}, 
$$c_{ij}=\begin{cases} 0, i=j \\ 1, |i-j|=1 \\ 2, |i-j|\geq 2 \end{cases}$$ and $d_i=4(n-3)$ for all $i$.

Now for $k_i \geq 0$, perform concurrent halved $\frac{1}{-k_i}$ horizontal surgery on the system $\{A_i\}$, and use the same reasoning as in cases (1)-(3) to see that we get a veering branched surface $B_{(k_i)}$ on $T^1 S^2(2+k_1,...,2+k_n)$ with $\sum_{|i-j|=1} k_ik_j+2\sum_{|i-j| \geq 2} k_ik_j+4(n-3)\sum k_i+2n(n-4)$ blue triple points and $2n$ red triple points (unless $n=4$ and $k_1=k_2=k_3=k_4=0$). This descends to a veering branched surface on $S^3$ with the cores of its complementary regions given by $M(\frac{1}{2+k_1}+1,\frac{1}{2+k_2}-1,...,\frac{1}{2+k_n}-1)$, having $\frac{1}{2} \sum_{|i-j|=1} k_ik_j+\sum_{|i-j| \geq 2} k_ik_j+2(n-3)\sum k_i+n(n-4)$ blue triple points and $n$ red triple points.

Taking the dual ideal triangulation, we get a veering triangulation on $S^3 \backslash M(\frac{1}{2+k_1}+1,\frac{1}{2+k_2}-1,...,\frac{1}{2+k_n}-1)$ with $\frac{1}{2} \sum_{|i-j|=1} k_ik_j+\sum_{|i-j| \geq 2} k_ik_j+2(n-3)\sum k_i + n(n-3)$ tetrahedra, $\frac{1}{2} \sum_{|i-j|=1} k_ik_j+\sum_{|i-j| \geq 2} k_ik_j+2(n-3)\sum k_i+n(n-4)$ blue edges, and $n$ red edges. We have finally completed the proof of \Cref{thm:genus0}.

\section{Geodesic flows II: Higher genus surfaces} \label{sec:highergenus}

In this section, we explain \Cref{constr:highergenus} and \Cref{constr:markovpart}. The key observation is that the construction in \Cref{subsec:ngeq4} can be applied more generally. Namely, let $c$ be a filling collection of curves on a surface $S$ which has no triple intersections and for which the complementary regions $S \cut c$ are $(n \geq 4)$-gons. By fitting in the portions of branched surfaces we constructed in \Cref{subsec:ngeq4} into the unit tangent bundles over these $n$-gons, we can construct almost veering branched surfaces on $T^1 S$ with the cores of the complementary regions being the full lift of $c$. If $c$ has no parallel elements, then by \Cref{thm:fulllifthyp} and \Cref{prop:avbshyp}, the branched surface is in fact veering and hence dual to a veering triangulation. Again, we record this last fact as a theorem.

\begin{thm} \label{thm:highergenus}
Let $c$ be a filling collection of mutually nonparallel curves on a surface $S$ which has no triple intersections and for which the complementary regions $S \cut c$ are $(n \geq 4)$-gons. Then $T^1 S \backslash \overset{\leftrightarrow}{c}$ admits a veering triangulation.
\end{thm}

Like \Cref{thm:genus0}, this theorem by itself is not very interesting. It is already known that these full lift complements admit veering triangulations just from the fact that they are fibered with fully-punctured pseudo-Anosov monodromy (by \cite[Theorem D]{CD20}). In fact, the monodromies of some of the fibering on these manifolds has been studied in \cite{DL19} and \cite{Mar20}, and in the former paper, invariant (up to folding moves) train tracks were found for certain cases where all complementary regions are ($n \geq 5$)-gons. 

The more significant point behind the theorem is that we make explicit the description of the veering branched surfaces and how they sit inside the manifolds. A consequence of this is that since the Anosov flows corresponding to the dual veering triangulations as given by \Cref{thm:vtpAcorr} must be the geodesic flow on $T^1 S$ (by \Cref{thm:sfspageod}), the reduced flow graphs of these veering branched surfaces, which we can explicitly describe, will encode Markov partitions for geodesic flows. 

We will explain the construction of the veering branched surfaces, as well as discuss some generalizations of the construction, in \Cref{subsec:notriangles} and explain the Markov partitions of geodesic flows we get from these in \Cref{subsec:Markovpart}.

\subsection{Construction of the branched surfaces} \label{subsec:notriangles}

As above, let $c$ be a filling collection of curves on $S$ which has no triple intersections and for which the complementary regions $S \cut c$ are $(n \geq 4)$-gons. We consider $c$ as a $4$-valent graph on $S$ by taking the vertices to be the set of intersections points among elements of $c$ and the edges to be segments of elements of $c$ between intersection points. 

For each complementary region $Q$ of $c$ in $S$, we place the branched surface $B_Q$ (along with the choice of orientations on the components of its branch locus) constructed in \Cref{subsec:ngeq4} inside $T^1 S|_Q$. If $Q_1$ and $Q_2$ are two adjacent complementary regions, we claim that the train tracks $B_{Q_1} \cap T^1 S|_{\partial Q_1}$ and $B_{Q_2} \cap T^1 S|_{\partial Q_2}$ agree along $T^1 S|_{\partial Q_1 \cap \partial Q_2}$. Indeed, if we let $e:[0,2] \to Q_1 \cap Q_2$ be a parametrization of $Q_1 \cap Q_2$, then both of these train tracks consist of the four horizontal lines $\{R_{k \frac{\pi}{2}} e'(t) : k=1,...,4\}$ and the two branches $\{R_\theta e'(1) : -\frac{3\pi}{4} \leq \theta \leq -\frac{\pi}{4}\}$ and $\{R_\theta e'(1) : \frac{\pi}{4} \leq \theta \leq \frac{3\pi}{4}\}$, with the same choice of combings. Thus the branched surfaces $B_Q$ glue up into a branched surface $B$ in $T^1 S$.

\begin{claim} \label{lemma:highergenusavbs}
$B$ is an almost veering branched surface with the cores of its complementary regions given by $\overset{\leftrightarrow}{c}$.
\end{claim}

\begin{proof}
This proof can be adapted from \Cref{lemma:ngeq4vbs} easily.

For \Cref{defn:avbs}(ii), recall that the complementary regions of $B_Q$ in $T^1 S|_Q$ are 1-cusped triangles times an interval. As in \Cref{lemma:ngeq4vbs}, this implies that the complementary regions of $B$ are 2-cusped solid tori. This also shows that the cores of the complementary regions of $B$ is given by $\overset{\leftrightarrow}{c}$.

Note that under our assumptions, $c$ cannot contain any homotopically trivial elements, otherwise the bounded disc in $S$ will be divided by $c$ into $n$-gons for $n \geq 4$, but the index of these are nonpositive thus cannot add up to the index of a disc, which is $1$. This fact together with \Cref{defn:avbs}(ii) proved above implies \Cref{defn:avbs}(i) as in \Cref{lemma:ngeq4vbs}.

For \Cref{defn:avbs}(iii), this follows from the moves that we did in the movie.
\end{proof}

If $c$ does not contain any parallel elements, then by \Cref{thm:fulllifthyp}, $T^1 S \backslash \overset{\leftrightarrow}{c}$ is hyperbolic. Hence by \Cref{prop:avbshyp}, $B$ is in fact veering in this case, so taking the dual ideal triangulation proves \Cref{thm:highergenus}.

\begin{rmk} \label{rmk:othersfs}
If, in the movie of train tracks we use to extend the branched surface into $T^1 S|_Q$, we instead unwind along $\bigcup c_{i,N} \cup d_{i,N}$ for $k-4$ times, for $k \geq 4$ dividing $n$, then, if $k$ is even, split $c^{(\frac{k-4}{2})}_{i-1,N}$ and $c^{(\frac{k-4}{2})}_{i,N}$ across $d^{(\frac{k-4}{2})}_{i,N}$, or if $k$ is odd, split $d^{(\frac{k-3}{2})}_{i,N}$ and $d^{(\frac{k-3}{2})}_{i+1,N}$ across $c^{(\frac{k-5}{2})}_{i,N}$, we will arrive at a train track on a torus which is the boundary of a branched surface on a solid torus as in \Cref{fig:othersfscore}, with $k$ levels and $\frac{n}{k}$ corners on each level. Notice, however, that if we cap off the train track using such a branched surface, the solid torus will be filled along a slope different from the meridian which recovers $T^1 S$. Nonetheless, we will still get a branched surface on a Seifert fibered manifold, and indeed, many of these if we vary the number of times we unwind above each complementary region.

\begin{figure}
    \centering
    \resizebox{!}{5cm}{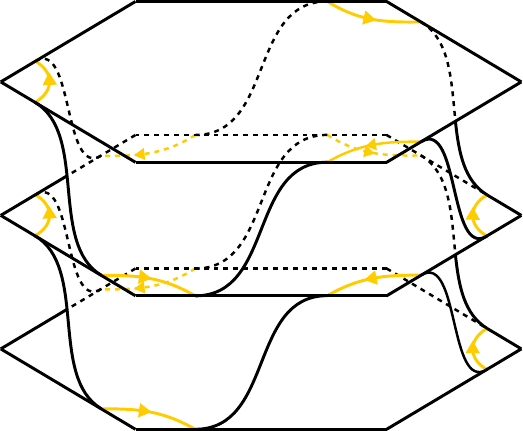}
    \caption{One can cap off the movies of train tracks using these cores instead, to produce foliar branched surfaces on other Seifert fibered manifolds. This particular core has $2$ levels and $3$ corners.}
    \label{fig:othersfscore}
\end{figure}

Moreover, these branched surfaces still satisfy \Cref{defn:vbs}(iii). They also satisfy a modified version of \Cref{defn:vbs}(ii): their complementary regions are surfaces with boundary $L$ times an interval $I$, with $\partial L \times I$ being the cusp circles. From these two properties, one can check that these branched surfaces are foliar, i.e. they are laminar and have product complementary regions, hence carry foliations by \cite{Li02}. It might be interesting to study whether \Cref{defn:vbs}(iii) implies any special properties of the foliations these branched surfaces carry.
\end{rmk}

\begin{rmk} \label{rmk:vbsorbifold}
One can obtain veering branched surfaces on unit tangent bundles of orbifolds via a generalization of the construction. Let $c$ be a collection of curves on an orbifold $S$ that lie away from the cone points. Suppose $c$ is filling, that is, components of $S \cut c$ are polygons with at most one cone point each, and suppose that each complementary region has nonpositive index. (Here, the \textit{index} of a disc with $n$ corners and a cone point of order $p$ is defined to be $\frac{1}{p}-\frac{n}{4}$.) Then one can define the portion of the branched surface above $c$ as before, and when extending into $T^1 S|_Q$, we can take a manifold cover $\widetilde{Q}$ of the region $Q$. We lift the train track on $T^1 Q|_{\partial Q}$ to $T^1 \widetilde{Q}|_{\partial \widetilde{Q}}$, then extend it inside $T^1 \widetilde{Q}$ using the same movie and same block to cap off the movie, and quotient it down to $T^1 Q$ by observing that the movie is equivariant under deck transformations of $T^1 \widetilde{Q} \to T^1 \widetilde{Q}$, which are just lifts of rotations of $Q$.
\end{rmk}

\begin{rmk} \label{rmk:sometriangles}
Here is yet another generalization of the construction. Let $c$ be a filling collection of curves on a surface $S$. We allow $c$ to have multiple intersections now, and consider $c$ as a graph with even valence at each vertex. Consider each complementary region of $c$ as a disc with angle $\frac{\pi}{p_i}$ at each $2p_i$-valent vertex of $c$.

If each complementary region of $c$ has at least $4$ sides, it is possible to construct a veering branched surface on $T^1 S$ with the cores of the complementary regions given by $\overset{\leftrightarrow}{c}$. This is done by inserting, for each complementary region $Q$ of $c$, half of the branched surface we constructed in \Cref{subsec:ngeq4} for the genus $0$ orbifold obtained by doubling $Q$, namely, the half that lies over $Q$. 

Similarly, if each complementary region of $c$ has $3$ sides and the valences of its vertices are $2p_1=4, 2p_2=6, 2p_3>12$ (or $2p_1=4, (2p_2,2p_3)>(8,8)$, or $(2p_1, 2p_2, 2p_3)>(6,6,6)$, respectively), then we can insert halves of the branched surfaces in \Cref{subsec:g2} (or \Cref{subsec:diag}, or \Cref{subsec:hex}, respectively).

Taking the dual ideal triangulations of these veering branched surfaces, we get veering triangulations on $T^1 S \backslash \overset{\leftrightarrow}{c}$ for a larger set of $c$ than those considered in \Cref{thm:highergenus}. However, it seems that this still does not cover all the cases of a filling collection $c$, for which it is possible to construct veering triangulations on $T^1 S \backslash \overset{\leftrightarrow}{c}$, as predicted by \Cref{thm:noperfectfits}.
\end{rmk}

\subsection{Markov partitions for geodesic flows} \label{subsec:Markovpart}

If one applies \Cref{thm:vtpAcorr} to one of the veering triangulations in \Cref{thm:highergenus}, we get an Anosov flow on $T^1 S$. By \Cref{thm:sfspageod}, this flow must be orbit equivalent to the geodesic flow. Hence by \Cref{thm:vtpAcorr}(c), the reduced flow graph encodes a Markov partition for the geodesic flow on $T^1 S$. In this section, we will explicitly work out these Markov partitions.

We will do this by computing the flow graph of $B_Q$ over each complementary region $Q$, then piecing them together. This is simply a task of tracing through the construction in \Cref{subsec:ngeq4}. 

We use the notation as in \Cref{subsec:ngeq4}. For convenience, we will refer to the sector of $B$ containing a branch of some train track $\tau_i$, as well as the corresponding vertex of the flow graph, by the same name as the branch. The drawback to this is that a sector (or its corresponding vertex in the flow graph) may have a number of different names, but we will point this out whenever it happens. 

To capture the full information of the Markov partition, we will also specify the framing of the edges and the planar orderings on the sets of incoming and outgoing edges at each vertex. The framing of the edges is simply given by the fibers of $T^1 S$, since the veering branched surfaces are transverse to the fibers. Alternatively, this can also be deduced from the fact that the unstable foliation of the geodesic flow is transverse to the fibers. We will describe the planar orderings on the incoming and outgoing edges by thinking of the anticlockwise direction on the fibers as the upwards direction, and refer to down, left, and right correspondingly as one goes along the oriented edges of the flow graph.

Also, we will think of the vertex of the flow graph at $d_{i,N}$ as an arrow pointing outwards of $Q$ at $d_i$, the vertex of the flow graph at $c_{i,N}$ as an arrow pointing outwards of $Q$ at the midpoint of $c_i$, and similarly for sectors of the other compass directions. 

We first consider each stage going from $\tau_0$ to $\tau_{n-4}$. Going from $\tau_{2s}$ to $\tau_{2s+1}$, we have to add in $n$ vertices $d^{(s+1)}_{i-s-1,N}$ and add in edges going from $d_{i-2s-2,W}, d^{(s+1)}_{i-s-1,N}, d_{i,E}$ to $d^{(s)}_{i-s-1,N}$. These edges are ordered from top to bottom at $d^{(s)}_{i-s-1,N}$, while the edge from $d_{i-2s-2,W}$ exits from the left side of $d_{i-2s-2,W}$ and the edge from $d_{i,E}$ exits from the right side of $d_{i,E}$. Going from $\tau_{2s+1}$ to $\tau_{2s+2}$, we have to add in $n$ vertices $c^{(s+1)}_{i-s-2,N}$ and add in edges going from $d_{i-2s-3,W}, c^{(s+1)}_{i-s-2,N}, d_{i,E}$ to $c^{(s)}_{i-s-2,N}$. These edges are ordered from top to bottom at $d^{(s)}_{i-s,N}$, while the edge from $d_{i-2s-3,W}$ exits from the left side of $d_{i-2s-2,W}$ and the edge from $d_{i,E}$ exits from the right side of $d_{i,E}$.

Then we consider going from $\tau_{n-4}$ to $\tau_{n-3}$. If $n$ is even, we have to add in $n$ vertices $d^{(\frac{n-2}{2})}_{i,N}$ and add in edges going from $c^{(\frac{n-4}{2})}_{i-1,N}, d^{(\frac{n-2}{2})}_{i,N}, c^{(\frac{n-4}{2})}_{i,N}$ to $d^{(\frac{n-4}{2})}_{i,N}$. These edges are ordered from bottom to top at $d^{(\frac{n-4}{2})}_{i,N}$, while the edge from $c^{(\frac{n-4}{2})}_{i-1,N}$ exits from the left side of $c^{(\frac{n-4}{2})}_{i-1,N}$ and the edge from $c^{(\frac{n-4}{2})}_{i,N}$ exits from the right side of $c^{(\frac{n-4}{2})}_{i,N}$. If $n$ is odd, we have to add in $n$ vertices $c^{(\frac{n-3}{2})}_{i,N}$ and add in edges going from $d^{(\frac{n-3}{2})}_{i,N}, c^{(\frac{n-3}{2})}_{i,N}, d^{(\frac{n-3}{2})}_{i+1,N}$ to $c^{(\frac{n-5}{2})}_{i,N}$. These edges are ordered from bottom to top at $c^{(\frac{n-5}{2})}_{i,N}$, while the edge from $d^{(\frac{n-3}{2})}_{i,N}$ exits from the left side of $d^{(\frac{n-3}{2})}_{i,N}$ and the edge from $d^{(\frac{n-3}{2})}_{i+1,N}$ exits from the right side of $d^{(\frac{n-3}{2})}_{i+1,N}$.

Finally, the core doesn't contain any triple points, hence doesn't contribute any edges to the flow graph, but it does connect up certain sectors. As a result, if $n$ is even, $c^{(\frac{n-4}{2})}_{i-\frac{n}{2},N}$ should be identified with $c_{i,S}$, and $d^{(\frac{n-2}{2})}_{i-\frac{n}{2},N}$ should be identified with $d_{i,S}$. If $n$ is odd, $d^{(\frac{n-3}{2})}_{i-\frac{n-1}{2},N}$ should be identified with $c_{i,S}$, and $c^{(\frac{n-3}{2})}_{i-\frac{n+1}{2},N}$ should be identified with $d_{i,S}$.

See \Cref{fig:hexnonamarkov} top for an illustration of the complete graph projected onto $Q$ for the cases $n=6$ and $n-9$.

\begin{figure}
    \begin{minipage}[c]{7cm}
    \centering
    \resizebox{!}{5cm}{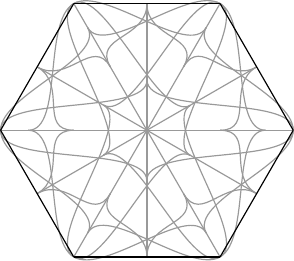}
    \end{minipage}
    \begin{minipage}[c]{7cm}
    \centering
    \resizebox{!}{6.5cm}{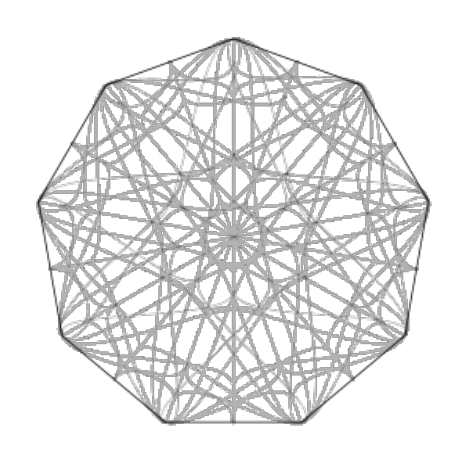}
    \end{minipage}
    \caption{Portions of the flow graphs of veering branched surfaces constructed on $T^1 S \backslash \overset{\leftrightarrow}{c}$ projected onto $n$-gon complementary regions of $c$, for $n=6$ (left) and $n=9$ (right).}
    \label{fig:hexnonamarkov}
\end{figure}

Notice that most vertices of the flow graph that lie along $c$ have more than one outgoing edge, except for the arrows at vertices for which the complementary regions to their left and right are squares. Indeed, for arrows at a side, there are 3 edges branching out of it inside the complementary region which the arrow is pointing inwards. For arrows at a vertex, there are edges branching out of it inside the complementary region which the arrow is pointing inwards, and also inside the complementary regions on its left or right, provided that those complementary regions are not squares. This implies that the only infinitesimal cycles of the flow graph consist of arrows at vertices for which the complementary regions to their left and right are squares, and the diagonals of complementary regions connecting them up. By removing all these cycles and the edges that enter them, we obtain the reduced flow graph.

To obtain the number of flow boxes as promised in the introduction. One can take $c$ to be a collection of curves which divides the surface $S$ into right-angle hexagons. (This can be found, in turn, by taking a pants decomposition and dividing each pair of pants into two hexagons.) Then performing the construction above, the corresponding flow graph can be obtained by piecing up the graphs in each hexagon illustrated in \Cref{fig:hexnonamarkov} left. Since there are no squares, the flow graph is equal to the reduced flow graph. This gives a Markov partition of the geodesic flow on $T^1 S$ which can be encoded by a graph with $-36 \chi(S)$ vertices and $-108 \chi(S)$ edges.

\begin{rmk} \label{rmk:Markovpartorbifold}
We can also obtain Markov partitions for the geodesic flow on the unit tangent bundle of a negatively curved orbifold $S$ by taking a manifold covering $\hat{S}$ of $S$ and constructing the Markov partition on $T^1 \hat{S}$ as above. Then as long as the curve $c$ chosen on $S$, for which the Markov partition is constructed from, is invariant under deck transformations of $\hat{S} \to S$, we can take the quotient to get a Markov partition on $T^1 S$.
\end{rmk}

\begin{rmk} \label{rmk:Markovpartcusped}
One can define geodesic flows for surfaces with a complete Riemannian metric in general. In particular, one can consider hyperbolic surfaces with cusps. Geodesic flows for these surfaces still have stable and unstable line bundles as in \Cref{defn:pAflow}, but a crucial qualitative difference from the closed surface case is that some orbits will be wandering, namely the geodesics that escape to infinity towards a cusp. Now, there is a broader class of flows called \textit{Axiom A flows} to which these flows belong, and there is still a notion of Markov partitions for Axiom A flows (see \cite{Bow70} for details). Using the techniques described here, we can obtain Markov partitions for these geodesic flows on hyperbolic surfaces with cusps as well.

Let $\Sigma$ be a hyperbolic surface with cusps. Let $D=\overline{\Sigma} \cup -\overline{\Sigma}$ be the closed surface obtained by compactifying $\Sigma$ along its cusps, then doubling across the resulting boundary components. The key observation here is that for the geodesic flow on $D$, the hyperbolic set of orbits staying within $\Sigma$ is orbit equivalent to the hyperbolic set of nonwandering points on $\Sigma$ by $\Omega$-stability, which in turn is implied by Axiom A and the no cycle property in this case (see \cite{Sma67}).

Thus let $c \in D$ be a filling collection of curves which contains $\partial \overline{\Sigma}$, and for which there are no triple intersections and the complementary regions $D \cut c$ are $(n \geq 4)$-gons. Perform the construction for this $c \subset D$ to get the reduced flow graph which encodes a Markov partition for the geodesic flow on $T^1 D$. 

Now notice that a cycle of the reduced flow graph is isotopic into $T^1 \overline{\Sigma}$ or into $T^1 (-\overline{\Sigma})$ if and only if it has intersection number zero with the annuli $T^1 D|_{\partial \overline{\Sigma}} \backslash \overset{\leftrightarrow}{\partial \overline{\Sigma}}$. Intersection with these annuli can in turn be represented by a positive cocycle on the reduced flow graph which is nonzero on all edges exiting a vertex on $\partial \overline{\Sigma}$ except for the edges that straddle $\partial \overline{\Sigma}$, that is, in the notation of \Cref{subsec:ngeq4}, the edges from $d_{i,W}$ to $d_{i+1,N}$ for a side $c_i$ of a complementary region that lies on $\partial \overline{\Sigma}$. If one follows these edges, one recovers the lifts of $\partial \overline{\Sigma}$, which are not closed orbits of the geodesic flow on $T^1 \Sigma$. The conclusion is that a cycle of the reduced flow graph in $T^1 D$ is isotopic to a closed orbit of the geodesic flow on $T^1 \Sigma$ or on $T^1(- \Sigma)$ if and only if it does not pass through the vertices on $\partial \overline{\Sigma}$. Hence we can discard all vertices and edges of the reduced flow graph in $T^1 D$ that lie in $-\overline{\Sigma}$ and on its boundary to get the desired Markov partition.
\end{rmk}

\section{Discussion and further questions} \label{sec:questions}

We discuss some future directions coming out of this paper. 

In \Cref{subsec:vsur}, we conjectured that under \Cref{thm:vtpAcorr}, vertical surgery should correspond to Goodman-Fried-Dehn surgery along a suitable closed orbit. However we do not have a good guess for what horizontal surgery should correspond to. We remark that the horizontal surgery curves that we considered in our constructions seem to correspond to curves of the form $\{R_{\pm \frac{\pi}{2}} c'(t) \}$, where $c$ is a small circle around a point. These curves are E-transverse in the terminology of \cite{FH13}. This would suggest that the horizontal surgeries we performed correspond to Foulon-Hasselblatt surgery under \Cref{thm:vtpAcorr}. But since Foulon-Hasselblatt surgery only applies to contact Anosov flows, we are not sure about the general case for horizontal surgeries.

It is known that the unit tangent bundle over the modular orbifold $S^2(2,3,\infty)$ is homeomorphic to the complement of the trefoil. Under this homeomorphism, a collection of closed orbits of the geodesic flow on $T^1 S^2(2,3,\infty)$ is called a Lorenz knot or link, see \cite[Section 3.5]{Ghy07}. It turns out that Lorenz knots and links satisfy some special properties, such as being fibered and prime (see \cite{BW83a} and \cite{Wil84}). These results are proved by analyzing what is called the Lorentz knotholder. In the language of this paper, the Lorentz knotholder can be obtained from a suitable Markov partition $\{I^{(i)}_s \times I^{(i)}_u \times [0,1]_t \}$ by taking the rectangular strips $I^{(i)}_u \times [0,1]_t$ and gluing $J^{(ij,k)}_u \times \{ 1 \}$ to $I^{(j)}_u \times \{0 \}$ in accordance with how the flow boxes meet along their top and bottom faces; or alternatively, thickening up the graph encoding the Markov partition in the $u$ direction.

Now, one can similarly obtain knotholders by thickening up the flow graphs of the veering branched surfaces we constructed in \Cref{thm:genus0}. It might be interesting to ask if the knots or links carried by these knotholders have any special properties. This would be equivalent to studying the knottedness and linkedness of closed orbits of geodesic flows of negatively curved genus zero orbifolds.

More generally, it would be interesting to study the veering branched surfaces themselves, and in particular identify all the vertical and horizontal surgery curves on them. One could then see if other families of knots or links admit veering triangulations on their complements.

In \Cref{sec:highergenus}, we constructed veering triangulations on $T^1 S \backslash \overset{\leftrightarrow}{c}$ for certain, but not all, filling multicurves $c$. However, we will show in \Cref{sec:directproof} that such veering triangulations should exist for all filling multicurves $c$, and furthermore we will characterize when such veering triangulations exist on $T^1 S \backslash \overset{\rightarrow}{c}$, for oriented multicurves $c$. This prompts the question of how one can construct such veering triangulations, or their dual veering branched surfaces, in a reasonably explicit way.

The following observation may serve as a starting point for the questions asked in the last two paragraphs. In the veering triangulation census, one can find other Montesinos knots complements which are not covered by our constructions. For example, K10n14 $= M(\frac{1}{3}, \frac{1}{3}, -\frac{3}{5})$ admits the veering triangulation \texttt{gLLMQaedfdffjxaxjkn\_200211} and K12n121 $= M(\frac{1}{2}, \frac{1}{3}, -\frac{9}{11})$ admits the veering triangulation \texttt{hLAPzkbcbeefgghhwjsahr\_2112212}. We remark that the double branched covers of these knots are fiberwise double covers of unit tangent bundles, hence the Anosov flows on these double branched covers which one obtains from \Cref{thm:vtpAcorr} must be the lift of a geodesic flow by \Cref{thm:sfspageod}. By understanding these triangulations or ways to construct them, one might gain insight on how to modify the constructions in this paper.

We already mentioned that the veering triangulations we constructed in \Cref{thm:genus0} and \Cref{thm:highergenus} are necessarily layered. This is because the flow in $T^1 S \backslash \overset{\leftrightarrow}{c}$ admits Birkhoff sections (see \cite[Theorem E and Theorem D]{CD20}). A possibly interesting way of investigating the monodromies of the Birkhoff sections is to take their intersections with the veering branched surfaces we constructed, in order to obtain periodic folding sequences of train tracks which then allows one to deduce the corresponding monodromies.

\appendix

\section{Characterization of no perfect fits} \label{sec:directproof}

In this appendix, we will characterize the orbits relative to which the geodesic flow is without perfect fits. This determines when it is possible to construct veering triangulations on drilled unit tangent bundles which give the geodesic flow under \Cref{thm:vtpAcorr}. As pointed out in \Cref{sec:questions}, it would be interesting to be able to describe these constructions explicitly. 

Before we state the result, we set up some notation.

\begin{defn}
Let $\Sigma$ be an oriented hyperbolic surface. Let $c, d$ be two oriented geodesics intersecting at point $x$. We say that $c$ intersects $d$ \textit{positively} at $x$ if $(c', d')$ is a positive basis at $x$, otherwise we say $c$ intersects $d$ \textit{negatively} at $x$. We also say that $x$ is a \textit{positive or negative intersection point} of $c$ with $d$, respectively.

More generally, we will call the signed angle from $c'$ to $d'$ at $x$ the \textit{angle at the intersection}. Note that this quantity only makes sense mod $2\pi$. Also note that $c$ intersects $d$ positively or negatively at $x$ when the angle at $x$ is in $(0,\pi)$ or $(-\pi,0)$ respectively.
\end{defn}

For $\Sigma=\mathbb{H}^2$, complete geodesics have a forward and backward endpoint on $\partial \mathbb{H}^2=S^1_\infty$. We will orient $S^1_\infty$ anticlockwisely, and use notation such as $(\xi_1,\xi_2)$, for $\xi_1,\xi_2 \in S^1_\infty$, to mean the interval from $\xi_1$ to $\xi_2$ on $S^1_\infty$ under this orientation. Under this notation, a complete geodesic $c$ in $\mathbb{H}^2$ with forward endpoint $\xi_1$ and backward endpoint $\xi_2$ intersects another complete geodesic $d$ positively if and only if the forward endpoint of $d$ lies in $(\xi_1,\xi_2)$ and the backward endpoint of $d$ lies in $(\xi_2,\xi_1)$. 

We are now ready to state the theorem.

\begin{thm} \label{thm:noperfectfits}
Let $\Sigma$ be a closed oriented hyperbolic surface and $c$ be a collection of oriented closed geodesics. Then the geodesic flow on $T^1 \Sigma$ has no perfect fits relative to the lift $\overset{\rightarrow}{c}$ if and only if every oriented closed geodesic $d$ on $\Sigma$ has a positive intersection point with some element of $c$. 

In particular, if $c$ is a collection of closed geodesics, then the geodesic flow on $T^1 \Sigma$ has no perfect fits relative to the full lift $\overset{\leftrightarrow}{c}$ if and only if $c$ is filling.
\end{thm}

\begin{proof}
This relies on an interpretation of the orbit space of the universal cover of the geodesic flow on $T^1 \Sigma$ in terms of the circle at infinity $\partial \widetilde{\Sigma}=S^1_\infty$. This viewpoint is well known to experts, but we explain it here for completeness.

Orbits of $T^1 \Sigma$ which lie in the same stable leaf lift to lifts of oriented geodesics in $\widetilde{\Sigma}=\mathbb{H}^2$ which converge exponentially. Equivalently, these are the oriented geodesics which share a forward endpoint on $\partial \mathbb{H}^2 =S^1_\infty$. Hence we can canonically parametrize the stable leaves in $T^1 \mathbb{H}^2$ by $S^1_\infty$, where we send each stable leaf to the common forward endpoints of the orbits in it. Lifting this to the universal cover $\widetilde{T^1 \Sigma}=\widetilde{T^1 \mathbb{H}^2}$, the stable leaf space in the universal cover can be parametrized by the line $\widetilde{S^1_\infty}$. Similarly, one can canonically parametrize the unstable leaves in $T^1 \mathbb{H}^2$ by $S^1_\infty$ by sending each unstable leaf to the common backward endpoints of the orbits in it.

We can now define a map from the orbit space of the geodesic flow on $T^1 \mathbb{H}^2$ to $S^1_\infty \times S^1_\infty$. We claim that this map is a homeomorphism onto $(S^1_\infty \times S^1_\infty) \backslash \Delta$ where $\Delta = \{(\xi, \xi): \xi \in S^1_\infty\}$ is the diagonal. This boils down to the fact that given $\xi_1, \xi_2 \in S^1_\infty$, if $\xi_1 \neq \xi_2$, there is a unique oriented geodesic in $\mathbb{H}^2$ with forward endpoint $\xi_1$ and backward endpoint $\xi_2$; and if $\xi_1=\xi_2$ there cannot be such a geodesic. 

One can lift this to a map which embeds the orbit space of the geodesic flow on $\widetilde{T^1 \Sigma}=\widetilde{T^1 \mathbb{H}^2}$ into $\widetilde{S^1_\infty} \times \widetilde{S^1_\infty}$ as a diagonal strip. However, we will choose to operate on the level of $T^1 \mathbb{H}^2$, since it makes the language a bit simpler.

We return to proving the theorem. Let $c$ be a collection of oriented closed geodesics. If there is an oriented closed geodesic $d$ in $\Sigma$ which does not have positive intersection points with any element of $c$, then lifting to $\mathbb{H}^2$, there is a geodesic $\widetilde{d}$ which does not have positive intersection points with any element in the set of lifts of elements of $c$, which we denote as $\widetilde{c}$. Let $\xi_1, \xi_2$ be the forward and backward endpoints of $\widetilde{d}$ respectively. Consider the region $([\xi_1, \xi_2] \times [\xi_2, \xi_1]) \backslash \Delta$ in $S^1_\infty \times S^1_\infty$, this is a rectangle with two opposite ideal vertices, which is called a \textit{lozenge} in \cite{Fen99}. We claim that there are no lifts of elements in $\widetilde{c}$ in the interior of the lozenge. Otherwise there is an element of $\widetilde{c}$ that has forward endpoint on $(\xi_1, \xi_2)$ and backward endpoint on $(\xi_2, \xi_1)$. $\widetilde{d}$ must intersect such a curve positively at some point, thus on $\Sigma$ there must be an element of $c$ for which $d$ intersects positively with. Now by restricting to near an ideal vertex of the lozenge (and lifting to the universal cover), we get a perfect fit rectangle which is disjoint from $\widetilde{c}$.

Conversely, suppose that there is a perfect fit rectangle disjoint from $c$. We want to find an oriented closed geodesic $d$ that has no positive intersection points with any element of $c$. We can assume that every element of $c$ intersects some element of $c$ positively and some element of $c$ negatively, since otherwise we can just pick $d$ to be some element of $c$, or its reverse.

With this assumption in place, without loss of generality, let the image of the perfect fit rectangle in the orbit space of $T^1 \mathbb{H}^2$ be $([\xi_1,\xi_0] \times [\xi_0, \xi_2]) \backslash \Delta$. Choose a geodesic ray $\widetilde{d}$ in $\mathbb{H}^2$ that has forward endpoint at $\xi_0$, and project it down to a geodesic ray $d$ on $\Sigma$. Notice that $d$ may not be periodic. 

Let $R_\theta: T^1 \Sigma \to T^1 \Sigma$ be the map that rotates vectors by $\theta$ counterclockwise. We claim that for every $\epsilon>0$, $\overset{\rightarrow}{d}$ eventually stays a bounded distance away from $\{R_\theta c'_i(t): \theta \in (-\pi+\epsilon, -\epsilon), \text{$c_i$ is an element of $c$}\}$ in $T^1 \Sigma$. For otherwise, lifting to $\mathbb{H}^2$, there are positive intersection points of $\widetilde{d}$ with elements of $\widetilde{c}$ which limit to infinity along $\widetilde{d}$ and have angles at the intersections bounded away from $0$ and $\pi$. The corresponding elements of $\widetilde{c}$ must have forward and backward endpoints converging to $\xi_0$ hence eventually fall into $(\xi_1, \xi_0)$ and $(\xi_0, \xi_2)$ respectively and lie in the interior of the perfect fit rectangle.

We further claim that $\overset{\rightarrow}{d}$ stays a definite distance away from $\overset{\leftrightarrow}{c}$ in $T^1 \Sigma$. For otherwise there must be arbitrarily long segments of $d$ that fellow-travel with an element of $c$ (possibly with reversed orientation). But we have assumed that every element of $c$ has both positive and negative intersection points with some elements of $c$, and since $c$ is a collection of closed geodesics, these intersection points have angle bounded away from $0$ and $\pi$. Hence near one of these intersection points, the fellow-travel segment of $d$ will intersect an element of $\widetilde{c}$ at a positive angle bounded away from $0$ and $\pi$, contradicting our claim in the last paragraph.

The two claims together imply that $\overset{\rightarrow}{d}$ eventually stays a definite distance away from $\{R_\theta c'_i(t): \theta \in [-\pi, 0], \text{$c_i$ is an element of $c$}\}$. Now we can apply the closing lemma for Anosov flows (\cite[Theorem 2.4]{Bow72}) on $d$ to get a closed geodesic $d'$ on $\Sigma$ which only has negative intersection points (if any) with elements of $c$, thus concluding the proof in the reverse direction.
\end{proof}

We remark that \Cref{thm:noperfectfits} can be applied to a non-orientable surface by lifting to its orientable double cover.

\section{Table of triangulations on Montesinos link complements in the census} \label{sec:table}

In this appendix, we compile the veering triangulations we constructed on Montesinos link complements in \Cref{thm:genus0} that appear in the veering triangulation census \cite{VeeringCensus}.

We remind the reader of our notation: $M(\frac{1}{p_1}+1,\frac{1}{p_2}-1,...,\frac{1}{p_n}-1)$ is the Montesinos link whose double branched cover is the unit tangent bundle of the orbifold $S^2(p_1,...,p_n)$, and we constructed veering branched surfaces on complements of these Montesinos links for which $e:=\sum \frac{1}{p_i}-n+2<0$. In the tables we present here, we consider a finite collection of these knots and links. Each one of these is represented by a cell in the table, with the data within being read as:

\begin{tiny}
\begin{center}
\begin{tabular}{|c|}
\hline

$(p_1,...,p_n)$ \\
Name of knot/link [Name of knot/link exterior] \\
IsoSig code of triangulation (if applicable) \\
\hline

\end{tabular}
\end{center}
\end{tiny}

The tables have been organized in an attempt to balance aesthetics and efficiency. As a result, the values of $(p_1,...,p_n)$ in some cells are such that $e:=\sum \frac{1}{p_i}-n+2 \geq 0$. The Montesinos links for these values are not hyperbolic, hence their complements cannot admit veering triangulations at all. For these cells we put `/' for the IsoSig code.

We describe how we compiled this data. For each knot or link, we first obtained its PD code using Kyle Miller's KnotFolio \cite{KnotFolio}. Then we input this code into SnapPy \cite{SnapPy} and asked it to identify both the name of the knot or link (among the census of all knots and links with 14 crossings or less), and the name of the knot or link exterior (among the census of cusped hyperbolic 3-manifolds that can be triangulated with 9 tetrahedra or less). 

For the IsoSig codes of the triangulations, we make use of the veering census. As described in \Cref{sec:genus0}, we know the number of tetrahedra and the number of blue and red edges in our triangulations. We also know that each end has exactly one ladderpole curve, since in the double cover $T^1 S \backslash \overset{\leftrightarrow}{c}$ each end has two ladderpole curves, and the involution switches the two. As remarked in \Cref{sec:genus0}, we know that these triangulations are layered. Finally, in terms of the 3-manifold, we also know that the homology of these Montesinos link complements are $\mathbb{Z}^b$ for $b =$ number of components in the knot or link. With these pieces of information, we can reasonably filter out the possible candidates in the veering triangulation census. 

Then we inputted each candidate triangulation into SnapPy and asked it to try to identify the 3-manifold, or at least compute its hyperbolic volume. (Most of the identification work is already done in the census.) Meanwhile from before, we already have the data of the actual knot and link complements in SnapPy, so we can eliminate those candidate triangulations that have the wrong census name or wrong hyperbolic volume. In most cases this directly identifies the triangulation we were looking for.

There were 2 cases which we had to do extra work. For these we inspected the remaining candidate triangulations more carefully using Regina \cite{regina}. We describe the analysis in both cases below.

For $(p_1,p_2,p_3)=(2,6,6)$, the above procedure leaves us undecided between \[\texttt{oLLvAwQMLQcbeehgiijjlnlmnnxxxavccaaaxcavc\_21112002212120}\] and \[\texttt{ovvLALQLQQchgggkijmnllnmnmaaaaaggaaggaaaa\_10000111111100}.\] We inputted these triangulations into Regina and checked their dual graph. Recall that the dual graph of a veering triangulation is the same as the branch locus of its dual veering branched surface. Hence from the descriptions in \Cref{sec:genus0}, one can work out the dual graph of the actual triangulation we are looking for. We can then eliminate \[\texttt{oLLvAwQMLQcbeehgiijjlnlmnnxxxavccaaaxcavc\_21112002212120}\] since its dual graph has a pair of vertices with two edges between them, whereas the dual graph of the actual triangulation does not.

For $(p_1,p_2,p_3,p_4)=(2,2,4,4)$, the above procedure leaves us undecided between \[\texttt{qvLAMAwPLzQkdcegfghiklmonppopbbbahabhbhabbhhga\_2011022001120201}\] and \[\texttt{qvvLLMLzQQQkfgfjiloknoplmnoppaaaavvavaaavvaaav\_1020212211211200}.\] As in the last case, we worked out the dual graph of the actual triangulation and compared it with that of the two candidates. In this case, we can eliminate \[\texttt{qvLAMAwPLzQkdcegfghiklmonppopbbbahabhbhabbhhga\_2011022001120201}\] because its dual graph has triangles (i.e. 3 vertices which have edges between each of them), whereas the dual graph of the actual triangulation does not.

We emphasize that we were able to identify \textit{all} the veering triangulations we constructed on the Montesinos link complements which appear in the census, albeit using this somewhat backwards methodology. With some expertise in using software such as Snappy or Regina, one can probably directly construct the triangulations then extract their IsoSig codes and identify them in the census more directly.

There are some obvious patterns exhibited by the compiled data that we would be remiss not to point out. Firstly, the minimal crossing number of all the listed knots and links equal to $\sum p_i$. By observing that $M(\frac{1}{p_1}+1, \frac{1}{p_2}-1,...,\frac{1}{p_n}-1)=M(-1+\frac{1}{p_1},...,-1+\frac{1}{p_{n-2}},\frac{1}{p_{n-1}},\frac{1}{p_n})$, and using the diagram for the latter with continued fraction expansions $\frac{1}{p_i}=0+\frac{1}{p_i}, -1+\frac{1}{p_i}=0+\frac{1}{1+\frac{1}{p_i-1}}$, as explained in \Cref{subsec:montesinos}, we see that the minimal crossing number is indeed at most $\sum p_i$. We conjecture that this upper bound is realized for all Montesinos links of the form $M(\frac{1}{p_1}+1, \frac{1}{p_2}-1,...,\frac{1}{p_n}-1)$.

Secondly, all the listed triangulations have a unique veering structure (up to reversing the transverse data). In fact, this is the reason why we have just listed the IsoSig codes without the data of the taut angle structure, which is what the census does, since in general there are triangulations with multiple veering structures. It would be interesting if this uniqueness holds in general, or even for the triangulations constructed in \Cref{sec:highergenus}. 

Lastly, we note that all the listed triangulations are reported to be geometric by SnapPy. Now, this might not be very indicative, since the vast majority of veering triangulations listed in the census are geometric (but their proportion is conjectured to tend to zero, and this is proven for layered veering triangulations in \cite{FTW20}). Nevertheless, it would be very interesting if all the veering triangulations we have constructed in this paper are geometric. Among other things, this would imply lower bounds for volumes of hyperbolic 3-manifolds of the form $T^1 S \backslash \overset{\leftrightarrow}{c}$, using the results of \cite{FG13}. We remark that lower bounds for volumes of such 3-manifolds have been recently obtained in \cite{CKMP21} and \cite{CRY20}. We also remark that Nimershiem has constructed geometric triangulations of $S^3 \backslash M(\frac{3}{2}, -\frac{2}{3}, \frac{1}{6+k}-1)$ in \cite{Nim21}. It is not clear to us at this point whether Nimershiem's triangulations are the ones dual to the veering branched surfaces we constructed.

\includepdf[pages=-,landscape=true]{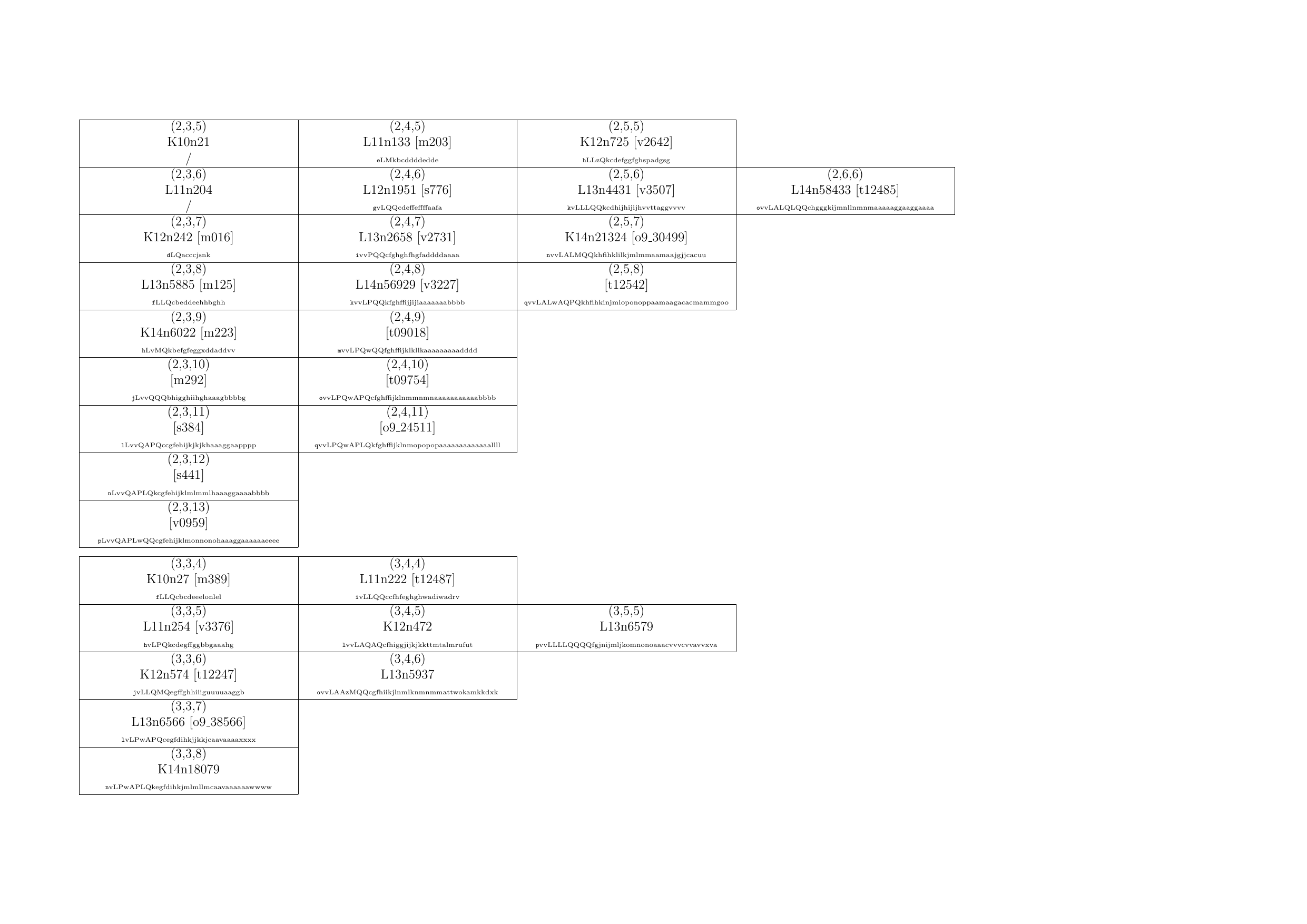}

\bibliographystyle{alpha}

\bibliography{bib.bib}

\end{document}